\documentclass[11pt]{article}

\usepackage{preamble}

\author{Andrew Krapivin \\ Carnegie Mellon University \\ \url{andrew@krapivin.net}  \and Benjamin Przybocki \\ Carnegie Mellon University \\ \url{benjamin.przybocki@gmail.com} \and Nicolás Sanhueza-Matamala \\ Universidad de Concepción \\ \url{nsanhuezam@udec.cl} \and Bernardo Subercaseaux \\ Carnegie Mellon University \\ \url{bersub@cmu.edu}}

\date{}

\title{Optimal and Efficient Partite Decompositions of Hypergraphs}

\usepackage{booktabs,tabularx,threeparttable,amsmath}

\newcolumntype{Y}{>{\raggedright\arraybackslash}X}

\mleftright
\begin{document}

\maketitle

\begin{abstract}
We study the problem of partitioning the edges of a $d$-uniform hypergraph $H$ into a family $\mathcal{F}$ of complete $d$-partite hypergraphs (\emph{$d$-cliques}). We show that there is a partition $\mathcal{F}$ in which every vertex $v \in V(H)$ belongs to at most $(\frac{1}{d!} + o_d(1))n^{d-1}/\lg n$ members of $\mathcal{F}$. Together with a simple information-theoretic lower bound, this settles the central question of a line of research initiated by~Erd\H{o}s and Pyber (1997) for graphs, and more recently by Csirmaz, Ligeti, and Tardos (2014) for hypergraphs. The $d=2$ case of this theorem answers a 40-year-old question of Chung, Erd\H{o}s, and Spencer (1983). Furthermore, our construction is algorithmically efficient: such optimal partitions can be constructed in time $O(n^d/d!)$. An immediate corollary of our result is an improved upper bound for the maximum share size for binary secret sharing schemes on uniform hypergraphs.

Building on results of Nechiporuk (1969), we prove that every graph with fixed edge density $\gamma \in (0,1)$ has a biclique partition of total weight at most $(\tfrac{1}{2}+o(1))\cdot h_2(\gamma) \frac{n^2}{\lg n}$, where $h_2$ is the binary entropy function. This result is asymptotically tight and answers a further question of Chung, Erd\H{o}s, and Spencer. Our construction implies that such biclique partitions can be constructed in time $O(m)$, which answers a question of Feder and Motwani (1995) and also improves upon results of Mubayi and Turán (2010) as well as Chavan, Rabinia, Grosu, and Brocanelli (2025). Using similar techniques, we also give an $n^{1+o(1)}$ algorithm for finding a subgraph $K_{t,t}$ with $t = (1-o(1)) \frac{\gamma}{h_2(\gamma)} \lg n$, which matches the celebrated K\H{o}v\'ari--Sós--Turán guarantee for small $\gamma$.

Our results show that biclique partitions are information-theoretically optimal representations for graphs at every fixed density, which makes them a natural succinct data structure. We show that with this succinct representation one can answer independent set queries and cut queries in time $O(n^2/ \lg n)$; prior work of Bansal, Williams, and Vassilevska Williams gave subquadratic algorithms for independent set queries at the cost of $\omega(n^2)$ preprocessing. We also show that if we increase the space usage by a constant factor, we can compute a modification of Charikar's $2$-approximation algorithm for the densest subgraph problem that runs in time $O(n^2/\lg \alpha)$ and gives a $2\alpha$-approximation for any $\alpha > 1$, thus establishing the first approximation guarantees that can be obtained in subquadratic time.

Finally, we show that graphs with \emph{polynomially bounded shattering}, a class including graphs of bounded VC-dimension, admit biclique partitions of weight $O(n^{2-1/(d+1)})$, where $d$ is the shattering exponent, which extends recent results of Cardinal and Yuditsky (2025).
\end{abstract}

\thispagestyle{empty}
\setcounter{page}{0}

\newpage
\setcounter{tocdepth}{2}
\tableofcontents
\thispagestyle{empty}
\setcounter{page}{0}
\newpage

\section{Introduction}\label{sec:intro}

In 1956, Lupanov~\cite{lupanov} considered the problem of minimizing the number of wires in a directed circuit, and observed that if $a$ input ports were fully connected to $b$ output ports, then the $a \cdot b$ wires could be replaced by adding one intermediary \emph{switch} $s$ to reduce the number of wires down to $a + b$. A repeated application of this idea is illustrated in~\Cref{fig:rectifier}.

\begin{figure}[h]
    \centering
      \begin{tikzpicture}
  \node[fill=\colTop,text=white,inner sep=2pt,minimum height=14pt] at (0,0) (a0) { $i_1$};
  \node[fill=\colTop,text=white,inner sep=2pt,minimum height=14pt] at (1,0) (a1) { $i_2$};
  \node[fill=\colTop,text=white,inner sep=2pt,minimum height=14pt] at (2,0) (a2) { $i_3$};
  \node[fill=\colTop,text=white,inner sep=2pt,minimum height=14pt] at (3,0) (a3) { $i_4$};

  \node[fill=\colTop,text=white,inner sep=2pt,minimum height=14pt] at (4,0) (a4) { $i_5$};

    \node[fill=\colTop,text=white,inner sep=2pt,minimum height=14pt] at (5,0) (a5) { $i_6$};
  \node[fill=\colBottom,text=white,inner sep=2pt,minimum height=14pt] at (0,-2) (b0) { $o_1$};
  \node[fill=\colBottom,text=white,inner sep=2pt,minimum height=14pt] at (1,-2) (b1) { $o_2$};
  \node[fill=\colBottom,text=white,inner sep=2pt,minimum height=14pt] at (2,-2) (b2) { $o_3$};
  \node[fill=\colBottom,text=white,inner sep=2pt,minimum height=14pt] at (3,-2) (b3) { $o_4$};
  \node[fill=\colBottom,text=white,inner sep=2pt,minimum height=14pt] at (4,-2) (b4) { $o_5$};

    \node[fill=\colBottom,text=white,inner sep=2pt,minimum height=14pt] at (5,-2) (b5) { $o_6$};
  \draw[dblarw={\colA}{1.0pt}{0.3pt}] (a0) -- (b1);
  \draw[dblarw={\colA}{1.0pt}{0.3pt}] (a0) -- (b2);
  \draw[dblarw={\colA}{1.0pt}{0.3pt}] (a1) -- (b2);
  \draw[dblarw={\colB}{1.0pt}{0.3pt}] (a1) -- (b3);
  \draw[dblarw={\colB}{1.0pt}{0.3pt}] (a1) -- (b4);
  \draw[dblarw={\colA}{1.0pt}{0.3pt}] (a2) -- (b0);
  \draw[dblarw={\colA}{1.0pt}{0.3pt}] (a2) -- (b1);
  \draw[dblarw={\colA}{1.0pt}{0.3pt}] (a2) -- (b2);
  \draw[dblarw={\colB}{1.0pt}{0.3pt}] (a2) -- (b3);
  \draw[dblarw={\colA}{1.0pt}{0.3pt}] (a3) -- (b0);
  \draw[dblarw={\colA}{1.0pt}{0.3pt}] (a3) -- (b1);
  \draw[dblarw={\colA}{1.0pt}{0.3pt}] (a3) -- (b2);
  \draw[dblarw={\colB}{1.0pt}{0.3pt}] (a4) -- (b3);
  \draw[dblarw={\colA}{1.0pt}{0.3pt}] (a1) -- (b1);
 \draw[dblarw={\colA}{1.0pt}{0.3pt}] (a0) -- (b0);
 \draw[dblarw={\colA}{1.0pt}{0.3pt}] (a1) -- (b0);

  \draw[dblarw={\colB}{1.0pt}{0.3pt}] (a2) -- (b4);
  \draw[dblarw={\colB}{1.0pt}{0.3pt}] (a4) -- (b4);
   
  \draw[dblarw={\colC}{1.0pt}{0.3pt}] (a3) -- (b3);
  \draw[dblarw={\colC}{1.0pt}{0.3pt}] (a5) -- (b5);

    \draw[dblarw={\colC}{1.0pt}{0.3pt}] (a3) -- (b5);
  \draw[dblarw={\colC}{1.0pt}{0.3pt}] (a5) -- (b3);

    \draw[lipicsYellow, line width=4.5pt, -latex] (5.8, -1) -- (7.5, -1);
\end{tikzpicture}
        \begin{tikzpicture}
  \node[fill=\colTop,text=white,inner sep=2pt,minimum height=14pt] at (0,0) (a0) { $i_1$};
  \node[fill=\colTop,text=white,inner sep=2pt,minimum height=14pt] at (1,0) (a1) { $i_2$};
  \node[fill=\colTop,text=white,inner sep=2pt,minimum height=14pt] at (2,0) (a2) { $i_3$};
  \node[fill=\colTop,text=white,inner sep=2pt,minimum height=14pt] at (3,0) (a3) { $i_4$};
  \node[fill=\colTop,text=white,inner sep=2pt,minimum height=14pt] at (4,0) (a4) { $i_5$};
   \node[fill=\colTop,text=white,inner sep=2pt,minimum height=14pt] at (5,0) (a5) { $i_6$};
  \node[fill=\colBottom,text=white,inner sep=2pt,minimum height=14pt] at (0,-2) (b0) { $o_1$};
  \node[fill=\colBottom,text=white,inner sep=2pt,minimum height=14pt] at (1,-2) (b1) { $o_2$};
  \node[fill=\colBottom,text=white,inner sep=2pt,minimum height=14pt] at (2,-2) (b2) { $o_3$};
  \node[fill=\colBottom,text=white,inner sep=2pt,minimum height=14pt] at (3,-2) (b3) { $o_4$};
  \node[fill=\colBottom,text=white,inner sep=2pt,minimum height=14pt] at (4,-2) (b4) { $o_5$};
 \node[fill=\colBottom,text=white,inner sep=2pt,minimum height=14pt] at (5,-2) (b5) { $o_6$};

    \node[circle, fill=\colA,text=black,inner sep=1pt ] at (1,-1) (s1) { $s_1$};

\draw[dblarw={gray}{0.6pt}{0.2pt}] (a0) -- (s1);
\draw[dblarw={gray}{0.6pt}{0.2pt}] (a1) -- (s1);
\draw[dblarw={gray}{0.6pt}{0.2pt}] (a2) -- (s1);
\draw[dblarw={gray}{0.6pt}{0.2pt}] (a3) -- (s1);

\draw[dblarw={gray}{0.6pt}{0.2pt}] (s1) -- (b0);
\draw[dblarw={gray}{0.6pt}{0.2pt}] (s1) -- (b1);
\draw[dblarw={gray}{0.6pt}{0.2pt}] (s1) -- (b2);

    \node[circle, fill=\colB,text=black,inner sep=1pt ] at (2.5,-1) (s2) { $s_2$};

        \node[circle, fill=\colC,text=black,inner sep=1pt ] at (4,-1) (s3) { $s_3$};

\draw[dblarw={gray}{0.6pt}{0.2pt}] (a1) -- (s2);
\draw[dblarw={gray}{0.6pt}{0.2pt}] (a2) -- (s2);
\draw[dblarw={gray}{0.6pt}{0.2pt}] (a4) -- (s2);

\draw[dblarw={gray}{0.6pt}{0.2pt}] (s2) -- (b3);
\draw[dblarw={gray}{0.6pt}{0.2pt}] (s2) -- (b4);

\draw[dblarw={gray}{0.6pt}{0.2pt}] (a3) -- (s3);
\draw[dblarw={gray}{0.6pt}{0.2pt}] (a5) -- (s3);

\draw[dblarw={gray}{0.6pt}{0.2pt}] (s3) -- (b3);
\draw[dblarw={gray}{0.6pt}{0.2pt}] (s3) -- (b5);



    
\end{tikzpicture}
    \caption{Illustration of the application of biclique partitions in circuit design.}
    \label{fig:rectifier}
\end{figure}
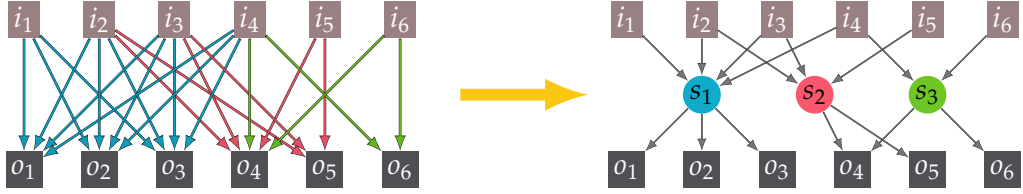

Without switches, a circuit with $n$ input ports and $n$ output ports may require $n^2$ wires, but Lupanov proved that with switches, such a circuit can be realized with only $(1 + o(1)) \frac{n^2}{\lg n}$ wires.\footnote{All logarithms in this paper are base 2.}  While Lupanov formulated his result in terms of circuits, it can be naturally recast in terms of \emph{biclique partitions} (i.e., writing a graph as an edge-disjoint union of bicliques) for bipartite graphs.

\begin{theorem}[Lupanov~\cite{lupanov}]\label{thm:lupanov}
    Let $G = (X \sqcup Y, E)$ be a bipartite graph with $|X| \ge |Y|$ and $|Y| = \omega(\lg |X|)$. Then, $E$ can be partitioned into bicliques $B_1, \ldots B_k$ with weight 
    \(
    \sum_{i=1}^k |V(B_i)| \le (1 + o(1)) \frac{|X| \cdot |Y|}{\lg |X|}.
    \)
    Moreover, there is some such $G$ for which every biclique partition has weight
    \(
     (1 - o(1)) \frac{|X| \cdot |Y|}{\lg |X|}.
    \)
\end{theorem}
The lower bound of~\Cref{thm:lupanov},
applied to a bipartite graph with $|X| = |Y| = n/2$, yields a general lower bound of $(1 - o(1)) \frac{n/2 \cdot n/2}{\lg n} = \left(\frac{1}{4} - o(1)\right) \frac{n^2}{\lg n}$ for $n$-vertex graphs. On the other hand, Lupanov's upper bound can be easily extended to general graphs by the following observation: from a graph $G = (V, E)$, choose an arbitrary total order $\prec$ on $V$, and then create a bipartite graph $G' = (V_1 \sqcup V_2, E')$ where each part $V_i$ is a copy of $V$, and for each edge $\{u, v\} \in E$ with $u \prec v$, add to $G'$ the edge $\{u_1, v_2\}$ where $u_1 \in V_1$ is the element corresponding to $u$, and $v_2 \in V_2$ the element corresponding to $v$. Then, observe that bicliques in $G'$ correspond to bicliques in $G$ (but not conversely), and thus~\Cref{thm:lupanov} applied to $G'$ implies that $G$ admits a biclique partition of weight $(1 + o(1))\frac{n^2}{\lg n}$, with $n := |V(G)|$. 
Therefore, if we denote by $\totalWeight(G)$ the minimum weight of a biclique partition of a graph $G$, and then $\totalWeight(n) := \displaystyle \max \{ \totalWeight(G) \!:\! G \text{ a graph with } |V(G)| = n\}$, Lupanov's result has the following consequence.

\begin{corollary}\label{cor:general-lupanov}
    $(\tfrac{1}{4} - o(1))\frac{n^2}{\lg n} \leq \totalWeight(n) \leq (1+o(1))\frac{n^2}{\lg n}$.
\end{corollary}

Unfortunately, Lupanov's result, originally published in Russian, continues to remain largely unknown in the West~\cite{Jukna2012}, and it was only in 1983 that Chung, Erd\H{o}s, and Spencer published bounds for general graphs, unaware of Lupanov's previous work.
Denoting by $\cover(G)$, and analogously, $\cover(n)$, the minimum weight of a biclique \emph{covering} (i.e., edges of $G$ may belong to multiple bicliques), we trivially have $\cover(n) \leq \totalWeight(n)$, and the Chung--Erd\H{o}s--Spencer (CES) theorem states: 

\begin{theorem}[\cite{chungDecompositionGraphsComplete1983}]\label{thm:CES}
     $(\frac{\lg e}{2e} - o(1))\frac{n^2}{\lg n} \leq \cover(n) \leq  \totalWeight(n) \leq (\frac{\lg e}{2} + o(1))\frac{n^2}{\lg n}$.
\end{theorem}

Since, $\lg(e)/(2e) \approx 0.265 > \nicefrac{1}{4}$, and $\lg(e)/2 \approx 0.721 < 1$,~\Cref{thm:CES} improved~\Cref{cor:general-lupanov} on both accounts. Tuza~\cite{Tuza1984} and Bublitz~\cite{Bublitz1986} independently proved similar results, albeit with worse quantitative bounds. Chung, Erd\H{o}s, Spencer, and Tuza asked whether $\lim_{n \to \infty} \totalWeight(n)/(n^2/\lg(n))$ and $\lim_{n \to \infty} \cover(n)/(n^2/\lg(n))$ exist and if so what their values are. Our first contribution, presented in \Cref{sec:optimal_biclique}, is an answer to these questions:
\begin{theorem} \label{thm-ces-optimal}
    $(\frac{1}{2} - o(1))\frac{n^2}{\lg n} \leq \cover(n) \leq \totalWeight(n) \leq (\frac{1}{2} + o(1))\frac{n^2}{\lg n}$.
\end{theorem}

In fact, in~\Cref{sec:random_graphs}, we show that the same bounds hold for $\cover(G)$ and $\totalWeight(G)$ for $G \sim G(n,1/2)$ with high probability, which also answers a question of Chung, Erd\H{o}s, and Spencer.

An important extension of the CES theorem is the Erd\H{o}s--Pyber theorem~\cite{erdos-pyber}, which says that every graph can be partitioned into bicliques so that each vertex is in at most $O(n/\lg n)$ bicliques, a result that directly implies that $\totalWeight(n) = O(n^2/\lg n)$. For example, in the aforementioned circuit minimization problem, the Erd\H{o}s--Pyber theorem says that one can not only reduce the number of wires to $O(n^2/\lg n)$ but also ensure that each input/output port is connected to $O(n/\lg n)$ wires, and thus no part of the circuit is very cluttered.

Interestingly, a key application of the Erd\H{o}s--Pyber theorem arises in cryptography, and more specifically, in the design of \emph{secret sharing schemes}~\cite{Beimel2014, blundoInformationRateSecret1996}. In a secret sharing scheme, the goal is to split a secret string $s$ into pieces, and distribute them to agents, in such a way that only certain specific groups of agents have enough information to reconstruct the secret if they join their pieces. These `allowed' groups are often all of the same size; if that size is $2$, then they can be specified by edges of a graph where agents are the vertices, but for a larger size $d$, they are specified by $d$-uniform hypergraphs. In this context, covering results like the CES theorem allow us to bound the \emph{total share} of the secret pieces used, and Erd\H{o}s--Pyber theorem bounds the \emph{maximum share} some agent is required to receive. This has motivated a line of work on Erd\H{o}s--Pyber results for $d$-uniform hypergraphs~\cite{ep-hypergraph,beimel:LIPIcs.ITC.2023.16}.

More precisely, define the \emph{local biclique partition number} of a graph $G$, denoted $\lbp(G)$, as the least natural number such that there is a biclique partition of $G$ for which every vertex is in at most $\lbp(G)$ bicliques. Similarly, let $\lbp(n) := \displaystyle \max \{ \lbp(G) : G \text{ a graph with } |V(G)| = n\}$, and observe that $\totalWeight(n) \le n \cdot \lbp(n)$. With this notation, the Erd\H{o}s--Pyber theorem says that $\lbp(n) = O(n/\lg(n))$. This was improved by Csirmaz, Ligeti, and Tardos~\cite{ep-hypergraph} to $\lbp(n) \leq (1+o(1)) \frac{n}{\lg n}$. Our proof of \Cref{thm-ces-optimal} in fact shows that $\lbp(n) \leq (\frac{1}{2}+o(1)) \frac{n}{\lg n}$.
\Cref{thm-ces-optimal} also implies that this is tight, since $\lbp(n) \ge \totalWeight(n)/n \ge (\frac{1}{2}-o(1)) \frac{n}{\lg n}$.

Csirmaz, Ligeti, and Tardos also generalized the Erd\H{o}s--Pyber theorem to $d$-uniform hypergraphs (\emph{$d$-graphs}).
Let a \emph{$d$-clique} be a $d$-graph $H = (V_1 \sqcup \dots \sqcup V_d, E)$ such that
\[
    E(H) = \{(v_1,\dots,v_d) : v_i \in V_i \ \text{for all} \ i \in [d]\}.
\]
This is the natural generalization of a biclique to $d$-uniform hypergraphs.
We also will use the notation $(V_1, \dotsc, V_d)$ to refer to this $d$-clique, and say that $V_1, \dotsc, V_d$ are its \emph{parts}.
Given a $d$-uniform hypergraph $H$, let $\lmp_d(H)$ be the least natural number such that there is a $d$-clique partition of $H$ for which every vertex is in at most $\lmp_d(H)$ $d$-cliques. We define $\mp_d(H)$, $\mc_d(H)$, $\mp_d(n)$, $\mc_d(n)$, and $\lmp_d(n)$ analogously to the definitions of $\cover$ and $\totalWeight$ above. Csirmaz, Ligeti, and Tardos proved that
\[
    \left(\frac{0.53}{d!} - o_d(1)\right) \frac{n^{d-1}}{\lg n} \le \lmp_d(n) \le \left(\frac{1}{(d-2)!} + o_d(1)\right) \frac{n^{d-1}}{\lg n}.
\]
Earlier, Bublitz~\cite{Bublitz1986} proved that
\(
    \mp_d(n) \le \left(\frac{32}{(d-1)!} + o_d(1)\right) \frac{n^d}{\lg n}.
\)
Our main theorem, proved in \Cref{sec:hypergraphs}, improves on these results and completely settles the asymptotics of the relevant functions:
\begin{theorem} \label{theorem:ces-hypergraph-optimal}
    $\left(\frac{1}{d!} - o_d(1)\right)\frac{n^{d-1}}{\lg n} \le \frac{\mc_d(n)}{n} \le \frac{\mp_d(n)}{n} \le \lmp_d(n) \le \left(\frac{1}{d!} + o_d(1)\right)\frac{n^{d-1}}{\lg n}$.
\end{theorem}
Moreover, we show that a $d$-clique partition realizing the upper bound on $\lmp_d(n)$ can be constructed deterministically in time $O(n^d / d!)$. Using Stinson's decomposition theorem~\cite{stinson} as discussed in \cite{ep-hypergraph}, an immediate consequence of our upper bound on $\lmp_d(n)$ is an improved upper bound for the maximum share size for secret sharing schemes with binary secrets on uniform hypergraphs.

\Cref{sec:nechiporuk} is devoted
to biclique partitions that consider the edge density $\gamma := |E(G)|/\binom{|V(G)}{2}$ of the input graph; both for dense graphs $\gamma >\nicefrac{1}{2}$, and sparse graphs $\gamma < \nicefrac{1}{2}$ it is possible to obtain smaller biclique partitions than those implied by~\Cref{thm-ces-optimal}.
In their 1983 paper, Chung, Erd\H{o}s, and Spencer also asked about the quantities $\cover(n, \gamma) := \max \{ \cover(G) \!:\! G \text{ a graph with } |V(G)| = n \ \text{and} \ |E(G)| = \gamma \binom{n}{2}\}$ and the similarly defined $\totalWeight(n, \gamma)$. Unbeknownst to Chung, Erd\H{o}s, and Spencer, Nechiporuk~\cite{nechiporuk} solved the analogous problems for bipartite graphs using a clever generalization of Lupanov's construction. Let $h_2(x) := -x\lg(x) - (1-x)\lg(1-x)$ be the binary entropy function.
\begin{restatable}[Nechiporuk~\cite{nechiporuk}]{theorem}{nechiporuk} \label{thm:nechiporuk_bip}
Let $G = (X \sqcup Y, E)$ be a bipartite graph with $|Y| \ge |X|$, and let $\gamma \in (0, 1)$ be such that $|X| = \omega(\frac{\lg |Y|}{h_2(\gamma)})$, $\max\{\gamma^{-1}, (1-\gamma)^{-1}\} = |Y|^{o(1)}$, and $|E| = \gamma |X| |Y|$. Then, $G$ admits a biclique partition of weight
\(
(1 + o(1)) \cdot h_2(\gamma) \frac{|X| |Y|}{\lg |Y|},
\)
and this is asymptotically best possible.
\end{restatable}
By the same reasoning as above, this implies that $(\frac{1}{4} - o(1)) \cdot h_2(\gamma) \frac{n^2}{\lg n} \le \totalWeight(n, \gamma) \le (1 + o(1)) \cdot h_2(\gamma) \frac{n^2}{\lg n}$ for $\max\{\gamma^{-1}, (1-\gamma)^{-1}\} = n^{o(1)}$. Building on Nechiporuk's construction, we solve the problem asymptotically in \Cref{sec:nechiporuk}:
\begin{restatable}{theorem}{sparseces} \label{thm:gamma-2}
    Let $\gamma \in (0, 1)$ be such that $\max\{\gamma^{-1}, (1-\gamma)^{-1}\} = n^{o(1)}$. Then,
    \(
        \totalWeight(n, \gamma) \sim \frac{h_2(\gamma)}{2} \cdot \frac{n^2}{\lg n}.
    \)
    Furthermore, given a graph represented as an adjacency list, a biclique partition realizing the upper bound can be constructed deterministically in time $O(m)$.
\end{restatable}

In \Cref{sec-density-ep}, we also prove a variant of this result for graphs with bounds on the degrees of each vertex, which can be seen as blending \Cref{thm:gamma-2} with the Erd\H{o}s--Pyber theorem. Specifically, we show that if $G$ is a regular graph with edge density $\gamma$, then $\lbp(G) \le (\frac{1}{2} + o(1)) \cdot h_2(\gamma) \frac{n}{\lg n}$; the actual result we prove, \Cref{thm-density-ep}, is a bit stronger than this.

A decomposition of a graph into bicliques can be fruitfully viewed as a compressed representation of the graph, and the above theorems provide theoretical guarantees on the quality of such a compression scheme. Feder and Motwani's seminal work showed that the runtime of several graph algorithms---including those for matchings, vertex- and edge-connectivity, and shortest paths---can be asymptotically improved by operating on the compressed representation of a graph~\cite{feder-motwani}. The effectiveness of this approach depends on two parameters: \emph{runtime} (i.e., the complexity of computing a biclique partition) and \emph{quality} (i.e., the weight of the resulting partition). Feder and Motwani developed a deterministic polynomial-time algorithm for constructing biclique partitions of weight $O(h_2(\gamma) n^2 / \lg n)$, and similarly, Mubayi and Turán~\cite{mubayiFindingBipartiteSubgraphs2010} independently gave a deterministic polynomial-time algorithm for constructing biclique partitions of weight $O(n^2 / \lg n)$. The algorithm of Feder and Motwani was recently improved by Chavan, Rabinia, Grosu, and Brocanelli~\cite{chavan2025}. In all three cases, the runtime is $\omega(n^2)$ and depends on the desired weight of the biclique partition. While each algorithm produces a biclique partition of weight $O(n^2/\lg(n))$, none produces a biclique decomposition with the strong quantitative bounds given by the CES theorem or Nechiporuk's theorem. It seems to have been hitherto unnoticed that Lupanov's proof of \Cref{thm:lupanov} yields a deterministic $O(n^2)$ algorithm for constructing a biclique partition of weight $O(n^2/\lg(n))$, one that is much simpler than the previous algorithms and achieves the upper bound stated in \Cref{cor:general-lupanov}; similarly, Nechiporuk's proof of \Cref{thm:nechiporuk_bip} yields a deterministic $O(m)$ algorithm for constructing a biclique partition of weight $O(h_2(\gamma) n^2 / \lg n)$, which answers a question from Feder and Motwani's paper. Our proofs of \Cref{thm-ces-optimal,thm:gamma-2} similarly yield deterministic $O(n^2)$ and $O(m)$ algorithms respectively producing a biclique partition with the optimal bound, improving on all of the previous results and achieving optimality both with respect to runtime and quality.

Building on the idea of using biclique partitions as a form of graph compression, in~\Cref{sec:representations} we further investigate on their usage as succinct or compact data structures; a data structure for a class of objects $U$ is said to be \emph{succinct} if it uses $(1+o(1))\lg(U)$ bits of space and \emph{compact} if it uses $O(\lg(U))$ bits. Since the number of graphs on $n$ vertices with density $\gamma$ is $\binom{\binom{n}{2}}{\gamma \binom{n}{2}}$ and $\lg \binom{\binom{n}{2}}{\gamma \binom{n}{2}} \sim \frac{h_2(\gamma) n^2}{2}$, it follows that a representation for graphs is succinct if it uses $\frac{h_2(\gamma) n^2}{2}$ bits, which is precisely what~\Cref{thm:gamma-2} allows us to do given that a vertex can be represented using $\lg n$ bits. This result gives theoretical support to more applied work in graph compression~\cite{Francisco2022,Hernndez2013}. Furthermore, we show that if we store graphs in succinct space as biclique partitions, certain natural graph queries can be answered faster than using a (compressed) adjacency matrix. Namely, we show that given a subset $S$ of the vertices, we can decide if it is an independent set in time $O(n^2/\lg n)$, as opposed to the naive $\Theta(n^2)$ algorithm. While better asymptotic algorithms exist for this task~\cite{williamsSubquadratic, RegLemmaAlgorithms, subcubicEquiv}, they require $\omega(n^2)$ preprocessing (in fact, the two latter algorithms use Feder and Motwani's algorithm as a subroutine) and use $\omega(n^2)$ space. We also show that the succinct biclique representation allows to answer \emph{cut queries} in time $O(n^2/\lg n)$. In a cut query the input consists of two disjoint subsets of vertices $S$ and $T$, and the output shall be the number of edges between $S$ and $T$. The main result of~\Cref{sec:representations} is that, by maintaining `back references' from each vertex to the list of bicliques it belongs to, which leads to a compact data structure (\emph{CB representation}), we can approximate the densest subgraph problem (i.e., find nonempty $S \subseteq V$ that maximizes $|E(G[S])|/|S|$) in subquadratic time:
\begin{restatable}{theorem}{charikar} \label{thm:charikar-subq}
    Given a CB representation for an $n$-vertex graph $G$, where the biclique partition has weight $O(n^2/\lg n)$, and any $\alpha > 1$, one can compute a $2\alpha$-approximation for the densest subgraph problem in time
    $O(n^2 / \lg \alpha)$.
\end{restatable}
While the densest subgraph problem is polynomial-time solvable~\cite{Charikar2000}, in practice very fast approximation algorithms are used~\cite{surveyDensest}; ours is, to the best of our knowledge, the first one with a subquadratic runtime (e.g., \Cref{thm:charikar-subq} yields a $2\lg \lg n$-approximation in time $O(n^2/\lg\lg\lg n) = o(n^2)$).

\Cref{sec:large_bb} is dedicated to algorithms for finding large balanced bicliques. Cardinal and Yuditsky~\cite{cardinal_et_al:LIPIcs.ESA.2025.67} observed that, as a corollary of the existence of biclique
partitions of weight $O(n^2/\lg n)$,
every graph with $\Omega(n^2)$ edges contains a biclique $K_{t,t}$ with $t = \Omega(\lg n)$ (see \Cref{lemma:e-bp}). 
This is an instance of the celebrated K\H{o}vári--S\'os--Turán (KST) theorem~\cite{Kovari1954} (see \Cref{thm:precise-KST}). In fact, the proof of the CES theorem (Thm.~\ref{thm:CES}), and the similar proofs given by Tuza~\cite{Tuza1984} and Bublitz~\cite{Bublitz1986}, are based on repeated applications of the KST theorem to peel off large bicliques from a graph until the edge density is sufficiently small. Efficiently finding large bicliques in a sufficiently dense graph plays a crucial role in the algorithms of Feder and Motwani~\cite{feder-motwani}, Mubayi and Turán~\cite{mubayiFindingBipartiteSubgraphs2010}, and Chavan, Rabinia, Grosu, and Brocanelli~\cite{chavan2025}. Feder and Motwani's algorithm implies that in a graph with $\Omega(n^2)$ edges, a biclique $K_{t,t}$ with $t = \Omega(\lg n)$ can be found deterministically in time $O(n^2 \log n)$, while Chavan, Rabinia, Grosu, and Brocanelli's work improves this to $O(n^2)$. Mubayi and Turán's algorithm gives a stronger quantitative guarantee: they give a deterministic $O(m)$ that finds a biclique $K_{t,t}$ with $t = (\frac{1}{5 \lg(4 \cdot e / \edgeDensity)} - o(1)) \lg(n)$ in a graph with edge density $\edgeDensity$; they also give a deterministic $O(n^{2.42})$ that finds a biclique $K_{t,t}$ with $t = (\frac{1}{\lg(4 \cdot e / \edgeDensity)} - o(1)) \lg(n)$. In fact, their first algorithm can be made to run in time $O(n^{1.42})$ in a model in which constant-time degree queries are allowed. As a corollary of our constructive proof of \Cref{thm:gamma-2}, we obtain a deterministic $O(m)$ algorithm that finds a biclique $K_{t,t}$ with $t = (1 - o(1)) \frac{\gamma \lg(n)}{h_2(\gamma)}$, which is better than Mubayi and Turán's algorithm. In fact, a more refined analysis allows us to show the following:
\begin{restatable}{theorem}{randalg} \label{thm:rand-alg}
   Given a graph with edge density $\edgeDensity$ such that $\max\{\edgeDensity^{-1}, (1-\gamma)^{-1}\} = n^{o(1)}$, there is a randomized $O( \frac{n \lg n}{h_2(\gamma)}) = n^{1 + o(1)}$ time algorithm that returns a biclique $K_{t,t}$ with $t = (1-o(1)) \frac{\gamma \lg n}{h_2(\edgeDensity)}$ with high probability. Furthermore, if we are allowed to query the degree of a vertex in $O(1)$ time, the algorithm can be made deterministic.
\end{restatable}

Finally, in \Cref{sec:shattering}, we show that our construction used to prove that $\lbp(n) \leq (\frac{1}{2}+o(1)) \frac{n}{\lg n}$ can be naturally adapted to prove stronger bounds for graphs with \emph{polynomially bounded shattering}, a notion closely related to VC-dimension. This result strengthens and generalizes bounds on $\cover(G)$ and $\totalWeight(G)$ for $d$-dimensional semi-algebraic graphs with bounded complexity~\cite{Do-2019,Agarwal-Aronov-Ezra-Katz-Sharir-2025,cardinal_et_al:LIPIcs.ESA.2025.67}, a family of graphs that includes many intersection graphs of geometric objects. As discussed in those references, biclique covers for such graphs have numerous applications in computational geometry, such as range searching and intersection queries for sets of geometric objects in $\mathbb{R}^d$.

\definecolor{theoboxTitle}{RGB}{140, 200, 80}   
\definecolor{theoboxBody}{HTML}{F2F5FA}    
\definecolor{theoboxFrame}{HTML}{12344D}   
\definecolor{theoboxTitleText}{HTML}{FFFFFF}
\definecolor{theoboxText}{HTML}{0B1620}

\definecolor{hypergraphTitle}{RGB}{80, 140, 200}


\tikzset{
  theobox/.style={
    rectangle split,
    rectangle split parts=2,
    rectangle split part fill={theoboxTitle,theoboxBody},
    draw=theoboxFrame,
    rounded corners=2pt,
    line width=0.6pt,
    align=left,
    inner sep=6pt,
  },
  theowidth/.style={text width=#1}, 
  theoarrow/.style={-{Latex}, line width=0.6pt},
}

\newcommand{\thmBox}[5]{
    \node[draw, font=\footnotesize, rectangle split,
     rectangle split parts=2, rounded corners=1pt, rectangle split part fill={#5,theoboxBody}, align=center] (#1) at #2 {\color{theoboxTitleText} \hypersetup{linkcolor=white} #3  \nodepart{second} #4
    };
}


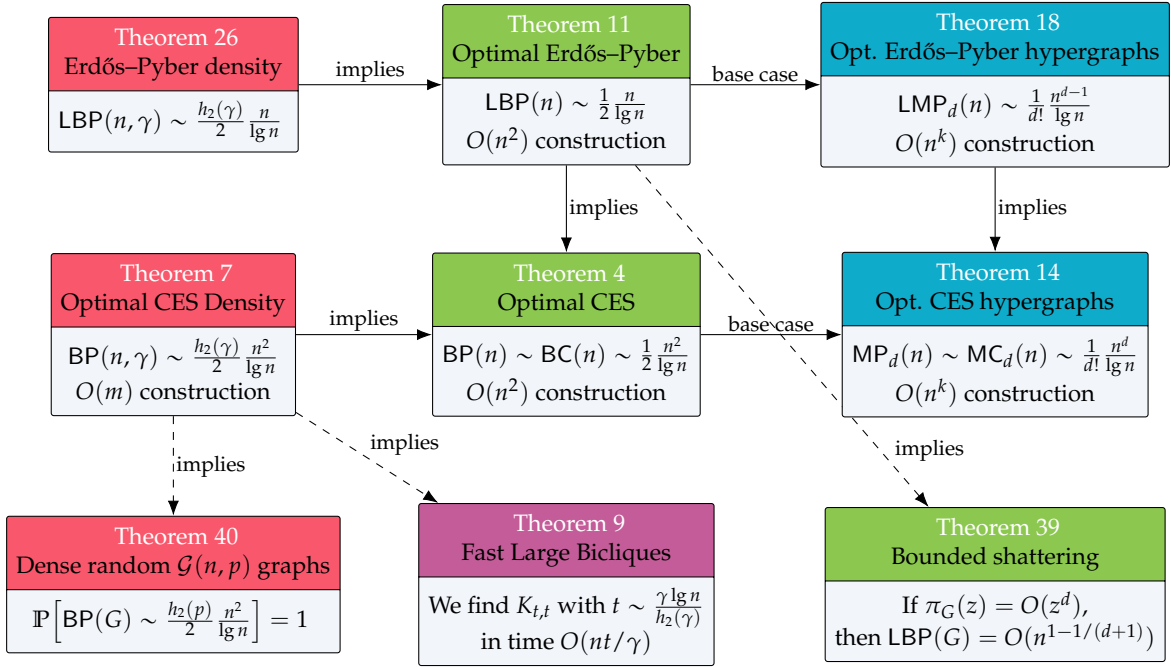
\begin{figure}
    \centering
  \begin{tikzpicture}[scale=0.95,
  >={Latex[length=2mm]},
  box/.style={draw, rounded corners=2pt, align=center, font=\small, inner sep=2.5pt, fill=gray!10},
  main/.style={box, thick, fill=blue!6},
  aux/.style={box, fill=green!5},
  app/.style={box, fill=orange!12},
  hyp/.style={box, fill=purple!10},
  barrier/.style={box, dashed, fill=red!6},
  every edge quotes/.style={auto, font=\scriptsize, inner sep=1pt}
]

\thmBox{OEP}{(0.5,0)}{\Cref{thm:ep} \\ Optimal Erd\H{o}s--Pyber}{$\lbp(n) \sim \frac{1}{2}\tfrac{n}{\lg n}$ \\
$O(n^2)$ construction}{theoboxTitle};

\thmBox{OCES}{(0.5,-3.5)}{\Cref{thm-ces-optimal} \\ Optimal CES}{$\totalWeight(n) \sim \cover(n) \sim \frac{1}{2}\tfrac{n^2}{\lg n}$ \\
$O(n^2)$ construction}{theoboxTitle};

\draw[->] (OEP) edge["implies"] (OCES);

\thmBox{OEPH}{(6.5, 0)}{\Cref{theorem:ep-hypergraph-upper} \\ Opt. Erd\H{o}s--Pyber hypergraphs}{$\lmp_d(n) \sim \frac{1}{d!}\tfrac{n^{d-1}}{\lg n}$ \\
$O(n^k)$ construction}{\colA};

\thmBox{boundedShattering}{(6.5, -7)}{\Cref{lemma:boundedshatter} \\ Bounded shattering}{If $\pi_G(z) = O(z^d)$,\\ then $\lbp(G) = O(n^{1-1/(d+1)})$}
{theoboxTitle};

\draw[->, dashed] (OEP) edge["implies", very near end] (boundedShattering);

\thmBox{CESH}{(6.5, -3.5)}{\Cref{theorem:ces-hypergraph} \\ Opt. CES hypergraphs}{$\mp_d(n) \sim \mc_d(n) \sim \frac{1}{d!}\tfrac{n^{d}}{\lg n}$ \\
$O(n^k)$ construction}{\colA};

\draw[->] (OEPH) edge["implies"] (CESH);

\draw[->] (OCES) edge["base case"] (CESH);

\draw[->] (OEP) edge["base case"] (OEPH);

\thmBox{DCES}{(-5, -3.5)}{\Cref{thm:gamma-2} \\ Optimal CES Density}{$\totalWeight(n, \gamma) \sim \tfrac{h_2(\gamma)}{2} \tfrac{n^2}{\lg n}$ \\
$O(m)$ construction}{\colB};

\thmBox{DEP}{(-5, 0)}{\Cref{thm-density-ep} \\ Erd\H{o}s--Pyber density}{$\lbp(n, \gamma) \sim \tfrac{h_2(\gamma)}{2} \tfrac{n}{\lg n}$}{\colB};

\draw[->] (DEP) edge["implies"] (OEP);


\thmBox{Prob}{(-5.0, -7)}{\Cref{thm:dense-random-graphs} \\ Dense random $\mathcal{G}(n, p)$ graphs}{$\prob\left[\totalWeight(G) \sim \frac{h_2(p)}{2}\frac{n^2}{\lg n}\right] = 1$}{\colB};

\draw[->] (DCES) edge["implies"] (OCES);

\draw[->, dashed] (DCES) edge["implies"] (Prob);


\thmBox{LargeBiclique}{(0.5, -7)}{\Cref{thm:rand-alg} \\ Fast Large Bicliques}{We find $K_{t,t}$ with $t \sim \frac{\gamma \lg n}{h_2(\gamma)}$\\ in  time $O(n t/\gamma)$}{purple!60!white!90!blue};

\draw[->, dashed] (DCES) edge["implies"] (LargeBiclique);




    \end{tikzpicture}
    \caption{Summary of our results. Dashed implications represent that other tools are also required. (\textcolor{\colB}{$\blacksquare$} := density-aware, \textcolor{theoboxTitle}{$\blacksquare$} := graph partitions, \textcolor{\colA}{$\blacksquare$} := hypergraph partitions, \textcolor{purple!60!white!90!blue}{$\blacksquare$} := finding bicliques) }
    \label{fig:summary-results}
\end{figure}

\section{Optimal biclique partitions}\label{sec:optimal_biclique}
Recall that, given a graph\footnote{Unless otherwise specified, all graphs we consider are finite, simple, and undirected.} $G = (V, E)$ a \emph{biclique partition} is a set $\mathcal{B} := \{B_1, \ldots, B_k\}$ where each $B_i$ is a complete bipartite graph, and such that $E(G) = \bigsqcup_{i=1}^k E(B_i)$. The \emph{weight} $w(\mathcal{B})$ of such a partition is simply $w(\mathcal{B}) := \sum_{i=1}^{k} |V(B_i)|$. The \emph{load} of a vertex $v$, denoted $\ell_\mathcal{B}(v)$, is the number of bicliques $B_i$ such that $v \in V(B_i)$.  Therefore, for any partition $\mathcal{B}$ of a graph $G$ on $n$ vertices, we have 
\begin{align*}
w(\mathcal{B}) &= \sum_{B \in \mathcal{B}} |V(B)| = \sum_{B \in \mathcal{B}} \sum_{v \in V(G)} \mathbb{1}_{[v \in V(B)]} = \sum_{v \in V(G)} \sum_{B \in \mathcal{B}} \mathbb{1}_{[v \in V(B)]}
= \sum_{v \in V(G)} \ell_{\mathcal{B}}(v) \leq n
\cdot \max_{v \in V(G)} \ell_{\mathcal{B}}(v),
\end{align*}
from where $\totalWeight(G) \leq n \cdot \lbp(G)$ and thus $\totalWeight(n) \leq n \cdot \lbp(n)$.

\subsection{Upper bound}

First, we will prove the upper bound on \Cref{thm-ces-optimal}. In fact, we will show an optimal upper bound on $\lbp(n)$, and then use the inequality $\totalWeight(n) \leq n \cdot \lbp(n)$.
We shall need the following concept.

\begin{definition}
    We say that a tournament (i.e., an orientation of the complete graph) is \emph{almost regular} if, for every vertex, the indegree and outdegree differ by at most 1.
\end{definition}
It is not hard to see that for every $n$, one can efficiently construct an almost-regular tournament on $n$ vertices.
For example, given $i, j \in [n]$, let $d_n(i,j) = \min(|i-j|, n - |i-j|)$. Then, the directed graph $T_n$ with $V(T_n) = [n]$ and
\begin{align*}
    E(T_n) := \{(i,j) : (i < j \ \text{and} \ d_n(i,j) \equiv 0 \;(\bmod\; 2)) \ \text{or} \ (i > j \ \text{and} \ d_n(i,j) \equiv 1 \;(\bmod\; 2))\}
\end{align*}
is an almost regular tournament. See~\Cref{subfig:almost-regular} for an illustration.

\begin{figure}
    \centering
    \input{figures/partition}
    \caption{Illustration of the biclique partition construction.}
    \label{fig:construction}
\end{figure}

\begin{theorem}\label{thm:ep}
    We have
    \(
        \lbp(n) \le \frac{1}{2}\frac{n}{\lg n} + O\left(\frac{n \cdot \lg \lg n}{\lg^2(n)}\right),
    \)
    and there is a biclique partition realizing this bound in which the number of bicliques is $O(n^2/\lg^3(n))$. Furthermore, given a graph $G$ on $n$ vertices, such a biclique partition can be constructed deterministically in time $O(n^2)$.
\end{theorem}
\begin{proof}
Let $\partSize := \lfloor \lg n - 2\lg \lg n \rfloor$ and partition $V(G)$ into $\lceil n / \partSize \rceil$ parts $P_1, \ldots, P_{\lceil n / \partSize \rceil}$ of size at most $\partSize$ each. Given a vertex $v \in V(G)$, let $g(v)$ be the index of the part $P_i$ containing $v$. Then, we consider an almost-regular tournament $R$ between the parts; so for each pair of part indices $i \neq j$ from $[\ceil{n/\partSize}]$, we write $R(i,j)$ if there is a directed edge from $i$ to $j$ in this tournament. Now, for each part $P_i$, and every nonempty subset $S \subseteq P_i$, we define $A(S) := \{ v \in V : N(v) \cap P_i = S \text{ and } R(g(v), i)\}$. See~\Cref{fig:construction} for an illustration. We now consider the set of bicliques
\[
\mathcal{C} := \underbrace{\{ (S, A(S)) : S \subseteq P_i, i \in [\ceil{n/\partSize}]\}}_{\textcolor{gray!90}{\text{across parts}}} \cup \underbrace{\{(\{u\}, \{v\}) : \{u, v\} \in E(G), g(u) = g(v)\}}_{\textcolor{gray!90}{\text{within parts}}} 
\]

First, let us see that $\mathcal{C}$ is a biclique partition of $G$. Let $\{u, v\} \in E(G)$. If $g(u) = g(v)$, then $\{u,v\}$ is covered uniquely by the biclique $(\{u\}, \{v\}) \in \mathcal{C}$. Otherwise, $g(u) \neq g(v)$, and we assume without loss of generality that $R(g(u),g(v))$. Let $S = N(u) \cap P_{g(v)}$. We claim that $\{u,v\}$ is covered uniquely by the biclique $(S,A(S)) \in \mathcal{C}$. Indeed, $u \in A(S)$ and $v \in S$, so $\{u,v\}$ is covered by $(S,A(S))$. Furthermore, if $\{u,v\}$ were covered by another biclique $(S',A(S')) \in \mathcal{C}$, then we must have $u \in A(S')$ and $v \in S'$, since $R(g(x),g(y))$ for every $x \in A(S')$ and $y \in S'$. Since $v \in P_{g(v)} \cap S'$, we have $S' \subseteq P_{g(v)}$. Then, since $u \in A(S')$, we have $S' = N(u) \cap P_{g(v)} = S$. Hence, $(S',A(S')) = (S,A(S))$, so $\{u,v\}$ is covered uniquely by a biclique in $\mathcal{C}$.

Now, fix a vertex $v$, and let us count how many bicliques from $\mathcal{C}$ contain $v$. First, there are the bicliques of the form $(\{u\}, \{v\})$ for $g(u)=g(v)$, of which there are at most $|P_{g(v)}|-1 = O(\lg(n))$. Second, there are the bicliques of the form $(S,A(S))$ with $v \in S$, of which there are at most
\[
    O\left(2^{|P_{g(v)}|}\right) = O\left(2^{\lg n - 2\lg \lg n}\right) = O\left(\frac{n}{\lg^2(n)}\right).
\]
Third, there are the bicliques of the form $(S,A(S))$ with $v \in A(S)$. By almost regularity of $R$, at most $\ceil{\frac{n}{2\partSize}}$ parts $P_i$ hold $R(g(v), i)$, and thus the number of such bicliques is at most
\[
    \ceil{\frac{n}{2\partSize}} \le \frac{n}{2(\lg n - 2\lg \lg n -1)} + 1 \le \frac{n}{2 \lg n} + O\left(\frac{n \cdot \lg \lg n}{\lg^2(n)}\right)
\]
This concludes the first part of the theorem.
For the total number of bicliques we have
\[
    |\mathcal{C}| = O\left(\frac{n}{\partSize} \cdot (2^{\partSize} + \partSize^2)\right) = O\left(\frac{n^2}{\lg^3(n)}\right).
\]

As the pseudocode in \Cref{alg:ep} demonstrates, the above construction can be implemented in deterministic $O(n^2)$ time. 
\end{proof}

\begin{algorithm}
    \caption{Biclique Partition}
    \label{alg:ep}
    \begin{algorithmic}[1]
    \Require A graph $G = (V, E)$ with $V = [n]$.

    \State $\partSize \gets \lfloor \lg n - 2\lg \lg n  \rfloor$
    \Function{$R$}{$i,j$} \Comment{$R \colon [\ceil{n/\partSize}]^2 \to \textsf{Bool}$}
        \State \Return $(i < j \ \text{and} \ d_n(i,j) \equiv 0 \;(\bmod\; 2)) \ \text{or} \ (i > j \ \text{and} \ d_n(i,j) \equiv 1 \;(\bmod\; 2))$
    \EndFunction
    
    \ForAll{$i \in [\lceil n/\partSize \rceil]$}
        \State $P_i \gets \{(i-1)\cdot \partSize+1,\dots,i\cdot \partSize\}$
        \State $A \gets \varnothing$ \Comment{$A \colon \mathcal{P}(P_i) \to \mathcal{P}(V)$, implemented as an array of linked lists of size $2^{|P_i|} = O\left(\frac{n}{\lg^2 n}\right)$}
        \ForAll{$v \in V$}
            \If{$R(\ceil{v/\partSize},i)$}
                \State $S \gets \varnothing$ \Comment{$S \colon \mathcal{P}(P_i)$}
                \ForAll{$u \in P_i$}
                    \If{$\{u,v\} \in E$}
                        \State $S \gets S \cup \{u\}$
                    \EndIf
                \EndFor
                \State $A(S) \gets A(S) \cup \{v\}$ \Comment{We  think of $S$ as a binary string that indexes $A$}
            \EndIf
        \EndFor
        \ForAll{$S \in \mathcal{P}(P_i)$}
            \If{$|S| > 0$  \text{ and } $|A(S)| > 0$} 
            \State \textbf{Emit} $(S,A(S))$
            \EndIf
        \EndFor
        \ForAll{$\{u,v\} \in \binom{P_i}{2}$}
            \If{$\{u,v\} \in E$}
                \State \textbf{Emit} $(\{u\},\{v\})$
            \EndIf
        \EndFor
    \EndFor
\end{algorithmic}
\end{algorithm}

\subsection{Lower bound}

The lower bound on $\cover$ is based on a simple information-theoretic argument: if every $n$-vertex graph can be covered with bicliques of weight at most $(r + o(1)) \frac{n^2}{\lg n}$, then we could specify an $n$-vertex graph with $(r + o(1))n^2$ bits, which is impossible if $r < \frac{1}{2}$. This argument is given in slightly less detail in~\cite{cardinal_et_al:LIPIcs.ESA.2025.67}.

\begin{theorem} \label{thm:lower-bound-cover}
    We have
    \[
        \cover(n) \ge \left(\frac{1}{2} - o(1)\right) \frac{n^2}{\lg n}.
    \]
\end{theorem}
\begin{proof}
Let $\mathcal{G}_n$ be the set of graphs with vertex set $[n]$. We will show there exists an injective function $f: \mathcal{G}_n \to \{0, 1\}^\star$ such that 
\(
|f(G)| \leq \cover(G) \cdot (\lg n + O(1))
\) for every $G \in \mathcal{G}_n$, and thus $|s| \leq \cover(n) \cdot (\lg n + O(1))
$ for every $s \in \text{range}(f)$. Note that the theorem follows from this: there are at most
\(
 2^{\cover(n) \cdot (\lg n + O(1))}
\)
binary strings of length at most ${\cover(n) \cdot (\lg n + O(1))}$, and thus $|\text{range}(f)| \leq 2^{\cover(n) \cdot (\lg n + O(1))}$, but by the injectivity of $f$ we  deduce
\(
 2^{\binom{n}{2}} = |\text{domain}(f)| \leq |\text{range}(f)| \leq 2^{\cover(n) \cdot (\lg n + O(1))},
\)
from where the result follows by taking logarithms.
\newcommand{\lexprec}{\preceq_{\text{lex}}} 

It only remains to show that such an $f$ exists. Indeed, identify the vertices of an $n$-vertex graph $G$ with binary strings of length $\ceil{\lg n}$, and let $\circ$ denote concatenation. Then, given a biclique $(X,Y)$ in $G$ with $n_1 := |X|$ and $n_2 := |Y|$, let $v_1,\dots,v_{n_1}$ be the vertices of $X$ in lexicographic order, and let $w_1,\dots,w_{n_2}$ be the vertices of $Y$ in lexicographic order. Then, we can represent the biclique $(X,Y)$ with the binary string
    \(
        \mathsf{enc}(X,Y) := 0^{n_1} \circ 1 \circ 0^{n_2} \circ 1 \circ v_1 \circ \dots \circ v_{n_1} \circ w_1 \circ \dots \circ w_{n_2}
    \)
    of length $(n_1 + 1) + (n_2 + 1) + (n_1 + n_2) \ceil{\lg n} = (n_1 + n_2) (\lg n + O(1))$. Now, given a biclique cover $\mathcal{B} = \{(X_1,Y_1),\dots,(X_k,Y_k)\}$, we can represent it with the binary string
    \(
               \mathsf{enc}(X_1,Y_1) \circ \mathsf{enc}(X_2,Y_2) \circ \cdots \circ \mathsf{enc}(X_k,Y_k)
    \)
    of length $w(\mathcal{B}) \cdot (\lg n + O(1))$, and let $g(\mathcal{B})$ be one of these strings chosen arbitrarily. 
    Let $f(G) := g(\mathcal{B})$, where $\mathcal{B}$ is a biclique cover of weight $\cover(G)$ chosen arbitrarily. Then, given an $n$-vertex graph $G$, $f(G)$ is a binary string of length at most $\cover(G) (\lg n + O(1))$. Moreover, $f$ is injective on the set of $n$-vertex graphs, since a biclique cover uniquely specifies an $n$-vertex graph and the encoding $g$ is injective. This concludes the proof.
\end{proof}

\Cref{thm-ces-optimal} follows immediately from \Cref{thm:ep,thm:lower-bound-cover}. 
The lower bound proof can be extended to almost all graphs, i.e., with high probability in the $G(n, \nicefrac{1}{2})$ model, answering another question of Chung, Erd\H{o}s, and Spencer. We generalize this to the $G(n, p)$ model for other values of $p$ in~\Cref{sec:random_graphs}.

It is worth commenting that our upper bound proof is similar to that of~\cite{lupanov}, and also that of~\cite{ep-hypergraph}, but we get a $\tfrac{1}{2}$ factor through the almost-regular tournament idea, whereas~\cite{ep-hypergraph} orient each edge, which ends up being a worse choice. In fact, the almost-regularity idea can be seen as a very basic form of \emph{equitability}, a concept we will develop much more powerfully when generalizing to hypergraphs. 

Furthermore, the previous analysis can easily be extended to directed graphs. We do not allow directed graphs to have loops, since a loop cannot be covered by a directed biclique.

\begin{theorem}
    Every directed graph $G = (V, E)$ can be partitioned into directed bicliques such that every vertex is contained in at most $(1 + o(1))\frac{n}{\lg n}$ of the directed bicliques, and this is asymptotically optimal.
\end{theorem}
\begin{proof}
    The lower bound follows directly from recreating the proof of~\Cref{thm:lower-bound-cover} but noting that there are $2^{n^2-n}$ directed graphs on $n$ vertices.
    For the upper bound, the construction is very similar to the proof of \Cref{thm:ep}, albeit even simpler. We again let $\partSize := \lfloor \lg n - 2\lg \lg n \rfloor$ and partition $V(G)$ into $\lceil n / \partSize \rceil$ parts $P_1, \ldots, P_{\lceil n / \partSize \rceil}$ of size at most $\partSize$ each. Given a vertex $v \in V(G)$, let $g(v)$ be the index of the part $P_i$ containing $v$. Now, for each part $P_i$, and every subset $S \subseteq P_i$, we define $A(S) := \{ v \in V : N(v) \cap P_i = S\}$. Consider the set of directed bicliques
    \(
    \mathcal{C} := \{ (S, A(S)) \mid S \subseteq P_i, i \in [\ceil{n/\partSize}] \}
    \)
    Given $(u,v) \in E(G)$, let $S = N(u) \cap P_{g(v)}$. Then, it is not hard to see that $(u,v)$ is covered uniquely by $(A(S),S)$, so $\mathcal{C}$ is a directed biclique partition of $G$. The proof that each vertex is contained in at most $(1 + o(1))\frac{n}{\lg n}$ directed bicliques from $\mathcal{C}$ is very similar to the proof of \Cref{thm:ep}.
\end{proof}

\section{Generalization to hypergraphs}\label{sec:hypergraphs}
In this section we prove \Cref{theorem:ces-hypergraph-optimal}.
For the sake of exposition, we begin with a weaker result that only bounds $\textsf{MC}_d(n)$ and $\textsf{MP}_d(n)$, but has a much simpler proof.

\subsection{Optimal Chung--Erd\H{o}s--Spencer for hypergraphs}

\begin{theorem} \label{theorem:ces-hypergraph}
    For every $d \geq 2$, we have $\left(\frac{1}{d!} - o_d(1)\right)\frac{n^d}{\lg n} \leq \textsf{MC}_d(n) \leq \textsf{MP}_d(n) \leq \left(\frac{1}{d!} + o_d(1)\right)\frac{n^d}{\lg n}$.
\end{theorem}

To see the lower bound, note that the proof of~\Cref{thm:lower-bound-cover} naturally generalizes to $d$-graphs, with the only essential difference being that there are $2^{\binom{n}{d}}$ different $d$-graphs on $n$ vertices.
For the proof of the upper bound we need the following lemma, which will enable an induction on $d$.

\begin{lemma} \label{lemma:ces-hypergraph-stepup}
    Let $f, g$ be non-decreasing functions and $d \geq 2$.
    Suppose that each $n$-vertex $d$-uniform hypergraph can be partitioned into at most $f(n)$ many $d$-cliques with total weight at most $g(n)$.
    Then every $n$-vertex $(d+1)$-graph can be partitioned into at most $n\cdot f(n)$ many $(d+1)$-cliques with total weight at most $\sum_{i=0}^{n-1} ( f(i) + g(i) )$.
\end{lemma}

\begin{proof}
    Let $H$ be an $n$-vertex $(d+1)$-uniform hypergraph, and order its vertices arbitrarily as $v_1, \dotsc, v_n$.
    For each $i \in [n]$, let $H_i \subseteq H$ be the $(d+1)$-uniform hypergraph defined by $V(H_i) = \{v_i, \dotsc, v_n\}$ and
    $
        E(H_i) = \{e \in E(H) : e \subseteq \{v_i, \dotsc, v_n\} \ \text{and} \ v_i \in e\}.
    $
    Note that $H_1, \dotsc, H_n$ partition $E(H)$. For each $i \in [n]$, let $H'_i$ be the $d$-uniform hypergraph defined by
    $V(H'_i) = \{v_{i+1}, \dotsc, v_n\}$ and
        $E(H'_i) = \{e : e \cup \{v_i\} \in E(H_i)\}.$
    By assumption, each $H'_i$ has a partition $\mathcal{C}'_i$ into at most $f(n-i) \leq f(n)$ $d$-cliques with total weight at most $g(n-i)$.
    Each $d$-clique $(A_1,\dots,A_d)$ in $H'_i$ can be extended to $(d+1)$-clique $(\{v_i\}, A_1,\dots,A_d)$ in $H_i$.
    Applying this to each $d$-clique in $\mathcal{C}'_i$, we obtain a family $\mathcal{C}_i$, which partitions $H_i$, consisting of at most $f(n)$ $d$-cliques, and with total weight at most $f(n-i) + g(n-i)$.
    Then, the family $\mathcal{F}_1 \cup \dotsb \cup \mathcal{F}_n$ partitions $E(H)$ and satisfies the required properties.
\end{proof}

\begin{proof}[Proof of \Cref{theorem:ces-hypergraph}]
    We prove, by induction on $d \geq 2$, that $n$-vertex $d$-graphs can be partitioned into at most $O_d(n^d/\lg^3n)$ many $d$-cliques with total weight at most $(1/d! + o_d(1))n^d/\lg n$.
    We know this is true for $d=2$, by \Cref{thm:ep}.
    Assuming it is true for $d$, we prove it for $d+1$.
    Using \Cref{lemma:ces-hypergraph-stepup} with $f(n) = O_d(n^d/\lg^3 n)$ and $g(n) = (1/d! + o_d(1)) n^d / \lg n$, we get that each $(d+1)$-graph can be partitioned into at most $nf(n) = O_d(n^{d+1}/\lg^3n)$ many $(d+1)$-cliques, and the total weight is at most \[\sum_{i=0}^{n-1}(f(i) + g(i)) = \left( \frac{1}{d!} + o_d(1) \right) \sum_{i=0}^{n-1} \frac{i^d}{\lg i},\]
    where we used that $f(n) = o_d(g(n))$.
    To analyze the last sum, we write it as $S_1 + S_2$, where $S_1$ corresponds to the first $t := \lfloor n/\lg n \rfloor$ terms, and $S_2$ to the remaining terms.
    We clearly have $S_1 \leq \sum_{i=0}^{t} t^d \leq t^{d+1} = o_d( n^{d+1} / \lg n)$.
    On the other hand, for $t \leq i \leq n$ we have $\lg i = (1 + o_d(1))\lg n$, and therefore
    $S_2 \leq \frac{(1 + o_d(1))}{d! \lg n} \sum_{i=t}^{n} i^d$.
    Hence, to conclude it suffices to observe that $\sum_{i=1}^{n} i^d \leq (1 + o_d(1)) \frac{1}{d+1}n^{d+1}$.
    A quick argument follows by estimating the sum using an integral: we have $i^d \leq \int_{i}^{i+1} x^d \mathrm{d}x$ by monotonicity of $x \mapsto x^d$ for $x \in [1, \infty)$, so
    \begin{align*}
    \sum_{i=1}^{n} i^d &\leq \int_{1}^{n+1} x^d \mathrm{d}x = \frac{x^{d+1}}{d+1}\Big|_1^{n+1} \leq \frac{(n+1)^{d+1}}{d+1} = \frac{n^{d+1}}{d+1} \cdot (1 + \tfrac{1}{n})^{d+1} 
    \leq \frac{n^{d+1}}{d+1} e^{(d+1)/n},
     \end{align*}
     we conclude by observing that $e^{(d+1)/n} = 1 + o_d(1)$.
\end{proof}

The above argument gives a deterministic $O(n^d)$-time algorithm to find a $d$-clique partition of weight at most $\left(\frac{1}{d!} + o_d(1)\right)\frac{n^d}{\lg n}$: in the $d=2$ case we use \Cref{alg:ep}.
For $d > 2$ we call $n$ times the $O(n^{d-1})$-time algorithm for $(d-1)$-graphs, as done in the proof of \Cref{lemma:ces-hypergraph-stepup}.

\subsection{Optimal Erd\H{o}s--Pyber for hypergraphs}

To complete the proof of \Cref{theorem:ces-hypergraph-optimal} we need to show the bound $\lmp_d(n) \leq \left(\frac{1}{d!} + o_d(1)\right) \frac{n^{d-1}}{\lg n}$ for each $d \geq 2$.
This is the content of \Cref{theorem:ep-hypergraph-upper}, for which we need some preparations.

First, let us analyze why the proof of~\Cref{theorem:ces-hypergraph} in the preceding subsection does not allow for a good upper bound on the number of $d$-cliques an arbitrary vertex $v$ participates in. 
Note that, if we unfold the induction from~\Cref{lemma:ces-hypergraph-stepup}, the constructed $d$-cliques have exactly two parts of size larger than $1$, and $d-2$ parts of size $1$. The construction can be thought of as creating, for each subset $S := \{s_1, \ldots, s_{d-2}\} \subset V(H)$ of $d-2$ vertices, a ``link graph'' $G_S$ whose edges are the pairs $\{v_i, v_j\}$ such that $\{v_i, v_j\} \cup S \in E(H)$. Then, each of these link graphs $G_S$ can be partitioned into bicliques $(L_i, R_i)$ using~\Cref{thm:ep},  resulting in $d$-cliques $(\{s_1\}, \ldots, \{s_{d-2}\}, L_i, R_i)$. For these to be an actual partition of $E(H)$, it is required that each hyperedge gets placed into exactly one link graph, which~\Cref{lemma:ces-hypergraph-stepup} ensures by using a fixed ordering of the vertices and placing a hyperedge in the link graph corresponding to its first $d-2$ vertices in the ordering. Note that the load (i.e., number of $d$-cliques) of every vertex $v$ is thus coming from two sources. First, there are the $d$-cliques in link graphs $G_{S}$ where $v \in S$, which for a fixed $S$ are exactly the bicliques into which $G_S$ gets partitioned by~\Cref{thm:ep}, of which there are $O(n^2/\lg^3 n)$.
On the other hand, there are the $d$-cliques in link graphs for which $v \not\in S$, which for a fixed $S$ are counted by the load of $v$ in the biclique partition of $G_{S}$; this is at most $(1+o(1))\frac{n}{2\lg n}$. The former are not many: there are at most $\binom{n-1}{d-3} = O(n^{d-3})$ possibilities for $S \ni v$, and therefore their associated load is at most $O(n^{d-3} \cdot n^2/\lg^3 n) = O(n^{d-1}/\lg^3 n) = o_d(\frac{n^{d-1}}{\lg n})$. The bottleneck is when $v \not\in S$, since then if $v$ is the last vertex in the fixed ordering, there are about $\binom{n-1}{d-2}$ choices for $S$, and thus the upper bound we get on the load of $v$ is at most
\[
\binom{n-1}{d-2} \cdot (1+o(1))\frac{n}{2\lg n} \leq (1 + o(1))\frac{n^{d-1}}{2(d-2)!\lg n},
\]
which is not good enough; that bound is only a factor $2$ better than the upper bound of~\cite{ep-hypergraph}, and this factor comes only from our improvement for graphs. The cause of this issue is that the strategy used to decide in which link graph to place each hyperedge is very ``inequitable''; vertices that come late in the ordering get too many, and the ones that come early get too few. This motivates the design of equitable strategies for placing hyperedges into link graphs. 

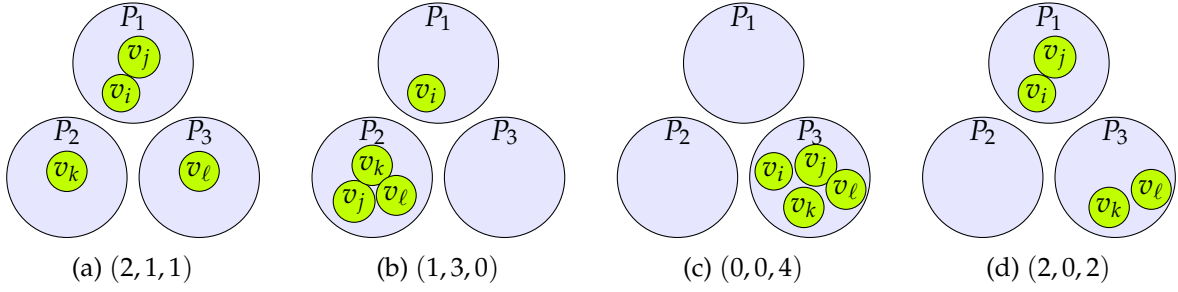
\begin{figure}
    \begin{subfigure}{0.24\linewidth}
        \centering
        \begin{tikzpicture}[scale=0.8]
        \def\r{1cm}
    
        \draw[fill=blue, fill opacity=0.1] (0,0) circle (\r);        
        \draw[fill=blue, fill opacity=0.1] (2.2,0) circle (\r);      
        \draw[fill=blue, fill opacity=0.1] (1.1,1.9) circle (\r);    
    
      \node (p1) at (0, 0.75) {$P_2$};
      \node (p2) at (1.1, 2.65) {$P_1$};
      \node (p3) at (2.2, 0.75) {$P_3$};
    
      \node[draw,circle,fill=lime,inner sep=1pt,minimum size=1pt] (v1) at (0.9, 1.4) {$v_i$};
    \node[draw,circle,fill=lime,inner sep=1pt,minimum size=1pt] (v2) at (1.2, 2.0) {$v_j$};
      \node[draw,circle,fill=lime,inner sep=1pt,minimum size=1pt] (v3) at (0, 0.1) {$v_k$};
     \node[draw,circle,fill=lime,inner sep=1pt,minimum size=1pt] (v4) at (2.2, 0.1) {$v_\ell$};
    \end{tikzpicture}
    \caption{$(2, 1, 1)$}\label{subfig:d4-case1}
    \end{subfigure}
    \begin{subfigure}{0.24\linewidth}
        \centering
        \begin{tikzpicture}[scale=0.8]
        \def\r{1cm}
    
        \draw[fill=blue, fill opacity=0.1] (0,0) circle (\r);        
        \draw[fill=blue, fill opacity=0.1] (2.2,0) circle (\r);      
        \draw[fill=blue, fill opacity=0.1] (1.1,1.9) circle (\r);    
    
      \node (p1) at (0, 0.75) {$P_2$};
      \node (p2) at (1.1, 2.65) {$P_1$};
      \node (p3) at (2.2, 0.75) {$P_3$};
    
      \node[draw,circle,fill=lime,inner sep=1pt,minimum size=1pt] (v1) at (0.9, 1.4) {$v_i$};
    \node[draw,circle,fill=lime,inner sep=1pt,minimum size=1pt] (v2) at (-0.3, -0.4) {$v_j$};
      \node[draw,circle,fill=lime,inner sep=1pt,minimum size=1pt] (v3) at (0, 0.2) {$v_k$};
     \node[draw,circle,fill=lime,inner sep=1pt,minimum size=1pt] (v4) at (0.4, -0.3) {$v_\ell$};
    \end{tikzpicture}
    \caption{$(1,3,0)$}\label{subfig:d4-case2}
    \end{subfigure}
    \begin{subfigure}{0.24\linewidth}
        \centering
        \begin{tikzpicture}[scale=0.8]
        \def\r{1cm}
    
        \draw[fill=blue, fill opacity=0.1] (0,0) circle (\r);        
        \draw[fill=blue, fill opacity=0.1] (2.2,0) circle (\r);      
        \draw[fill=blue, fill opacity=0.1] (1.1,1.9) circle (\r);    
    
      \node (p1) at (0, 0.75) {$P_2$};
      \node (p2) at (1.1, 2.65) {$P_1$};
      \node (p3) at (2.2, 0.75) {$P_3$};
    
      \node[draw,circle,fill=lime,inner sep=1pt,minimum size=1pt] (v1) at (1.6, 0.1) {$v_i$};
    \node[draw,circle,fill=lime,inner sep=1pt,minimum size=1pt] (v2) at (2.3, 0.2) {$v_j$};
      \node[draw,circle,fill=lime,inner sep=1pt,minimum size=1pt] (v3) at (2.1, -0.5) {$v_k$};
     \node[draw,circle,fill=lime,inner sep=1pt,minimum size=1pt] (v4) at (2.8, -0.2) {$v_\ell$};
    \end{tikzpicture}
    \caption{$(0,0,4)$}\label{subfig:d4-case3}
    \end{subfigure}
      \begin{subfigure}{0.24\linewidth}
        \centering
        \begin{tikzpicture}[scale=0.8]
        \def\r{1cm}
    
        \draw[fill=blue, fill opacity=0.1] (0,0) circle (\r);        
        \draw[fill=blue, fill opacity=0.1] (2.2,0) circle (\r);      
        \draw[fill=blue, fill opacity=0.1] (1.1,1.9) circle (\r);    
    
      \node (p1) at (0, 0.75) {$P_2$};
      \node (p2) at (1.1, 2.65) {$P_1$};
      \node (p3) at (2.2, 0.75) {$P_3$};
    
       \node[draw,circle,fill=lime,inner sep=1pt,minimum size=1pt] (v1) at (0.9, 1.4) {$v_i$};
    \node[draw,circle,fill=lime,inner sep=1pt,minimum size=1pt] (v2) at (1.2, 2.0) {$v_j$};
      \node[draw,circle,fill=lime,inner sep=1pt,minimum size=1pt] (v3) at (2.1, -0.5) {$v_k$};
     \node[draw,circle,fill=lime,inner sep=1pt,minimum size=1pt] (v4) at (2.8, -0.2) {$v_\ell$};
    \end{tikzpicture}
    \caption{$(2,0,2)$}\label{subfig:d4-case4}
    \end{subfigure}
    \caption{The 4 distributions for $d = 4$ up to permutation.}\label{fig:d4-cases}
\end{figure}
Now, before proving~\Cref{theorem:ep-hypergraph-upper} for arbitrary $d$, let us consider a particular case such as $d = 4$ to motivate the conceptual design of our proof. In this case, the main question is \emph{in which link graph should we place a given hyperedge $e := \{v_i, v_j, v_k, v_\ell\}$?} Suppose we first partition $V(H)$ into $d-1 = 3$ parts $P_1, P_2, P_3$ of roughly equal size. Then, a ``distribution'' tuple $(|P_1 \cap e|, |P_2 \cap e|, |P_3 \cap e|)$ will capture how $e$ is distributed along the parts, as illustrated in~\Cref{fig:d4-cases}. We will construct link graphs whose vertices are completely contained within a part, since this way each link graph will have roughly $n/(d-1) = n/3$ vertices instead of $n$. In consequence, the case presented in~\Cref{subfig:d4-case1} has only one possible choice: to place the edge $\{v_i, v_j\}$ in the link graph $G_{\{v_k, v_\ell\}}$. To handle the case presented in~\Cref{subfig:d4-case2}, however, we will consider a 3-uniform link hypergraph $H_{\{v_i\}}$, in which the edge $\{v_j, v_k, v_\ell\}$ will be placed, and then recursion with $d = 3$ over $H_{\{v_i\}}$ will take care of splitting hyperedges such as $\{v_j, v_k, v_\ell\}$ in such a way that each of those $3$ vertices ends up with the same load. Next, the case presented in~\Cref{subfig:d4-case3} will be handled by recursion with $n' := |P_3| \approx n/3$, which will further partition $P_3$. The last case, in~\Cref{subfig:d4-case4}, is the most interesting 
since for the first time there is a choice: shall we place $\{v_i, v_j\}$ into $G_{\{v_k, v_\ell\}}$ (thus loading $P_1$ more) or $\{v_k, v_\ell\}$ into $G_{\{v_i, v_j\}}$ (thus loading $P_3$ more)? The important aspect here is that we need to make a choice that is not biased towards any particular part, so we can choose to e.g., load $P_3$ more for distribution $(2, 0, 2)$, as long as $P_1$ gets more load from $(2, 2, 0)$ and $P_2$ get more from $(0, 2, 2)$.

Note that, a key idea used here is that the decision of what to do with $e$ is purely based on its distribution tuple, and not on the precise identity of its vertices. Thus, the number of link graphs in which edge $\{v_i, v_j\}$ is placed by a distribution tuple $(2,1,1)$, as in~\Cref{subfig:d4-case1}, is simply $|P_2| \cdot |P_3| \approx n^2/9$, and each of those incurs in load roughly $|P_1|/(2\lg |P_1|) \approx n/(6 \lg n)$ by~\Cref{thm:ep}, thus resulting in load $n^3/(54\lg n)$ associated to tuples $\{(2,1,1), (1,2,1), (1,1,2)\}$ for each vertex. In turn, when considering a distribution tuple like $(1,3,0)$, the vertices in $P_2$ will be part of $|P_1| \approx n/3$ link hypergraphs like $H_{\{v_i\}}$, and by inductive hypothesis, this will result in a load of roughly $\frac{n}{3} \cdot \frac{(n/3)^{2}}{3! \lg n} = n^3/(162\lg n)$, which after accounting for the symmetric tuple $(0,3,1)$, which also puts load on $P_2$, this case results in $n^3/(81\lg n)$ per vertex. When considering a distribution tuple like $(0, 0, 4)$, as in~\Cref{subfig:d4-case3}, by inductive hypothesis this will incur in load roughly $(n/3)^3/(4! \lg n) = n^{3}/(648 \lg n)$. Finally, when considering distributions in $\{(0,2,2), (2,0,2), (2,2,0) \}$, each vertex $v$ only needs to consider the single one that is loaded to its part; this yields a load of roughly $\binom{n/3}{2}\cdot \frac{n/3}{2 \lg n} \approx n^3/(108\lg n)$. Aggregating the different cases, the load of every vertex is at most
\[
(1+o(1))\frac{n^3}{\lg n}\left(\frac{1}{54} + \frac{1}{81} + \frac{1}{648} + \frac{1}{108}\right) = \left(\frac{1}{4!} + o(1)\right) \frac{n^3}{\lg n},
\]
as desired.

It is not obvious, however, how to generalize this reasoning for arbitrary $d$: how should we proceed with distributions like $(2,2,3,3,0,0,0,0,0)$ or $(2,2,2,2,1,1,0,0,0)$ for $d=10$? This is where the notion of equitability comes in. Every potential hyperedge (i.e., subset of $d$ vertices) corresponds to a distribution $\vec{x} = (x_1, \ldots, x_{d-1})$ according to how many vertices are in each part, and given such a distribution, we must make some choice of which part is load-bearing, which we call $f(\vec{x}) \in [d-1]$. Given a random potential hyperedge, we want the probability that each part $i \in [d-1]$ is load-bearing in the corresponding distribution to be roughly $1/(d-1)$; more specifically, we want
\[
    \sum_{(x_1,\dots,x_{d-1}) \in f^{-1}(i)} \frac{\binom{d}{x_1,\dots,x_{d-1}}}{(d-1)^d} = \frac{1}{d-1}
\]
for each $i \in [d-1]$. This motivates the following definition.

\begin{definition}
    Let $d \ge 2$ be an integer. A \emph{$d$-distribution} is a tuple $(x_1,\dots,x_{d-1}) \in \mathbb{N}^{d-1}$ such that
    \(
        x_1 + \ldots + x_{d-1} = d.
    \)
    Let $\mathcal{S}_d$ be the set of $d$-distributions. A \emph{selection strategy} is a function $f : \mathcal{S}_d \to [d-1]$ such that $f(x_1,\dots,x_{d-1}) = i$ for some $i$ such that $x_i \ge 2$. Then, we say a selection strategy is \emph{equitable} if for all $i,j \in [d-1]$,
    \[
        \sum_{(x_1,\dots,x_{d-1}) \in f^{-1}(i)} \binom{d}{x_1,\dots,x_{d-1}} = \sum_{(x_1,\dots,x_{d-1}) \in f^{-1}(j)} \binom{d}{x_1,\dots,x_{d-1}}.
    \]
\end{definition}

\begin{lemma}\label{lemma:exists-equitable}
    For every $d \ge 2$, there exists an equitable selection strategy.
\end{lemma}
\begin{proof}
    Fix an integer $d \ge 2$. Given a $d$-distribution $(x_1,\dots,x_{d-1})$, let $[(x_1,\dots,x_{d-1})]$ be its set of cyclic rotations; that is,
    \(
        [(x_1,\dots,x_{d-1})] = \{(x_i, \dots, x_{d-1}, x_1, \dots,x_{i-1}) : i \in [d-1]\}.
    \)
    Note that $[\cdot]$ partitions $\mathcal{S}_d$ into equivalence classes.

    First, we claim that for every $d$-distribution $(x_1,\dots,x_{d-1})$, we have $|[(x_1,\dots,x_{d-1})]| = d-1$. By a standard result from combinatorics on words (see, e.g., \cite[Section~1.2.1]{lothaire-alg}), this is equivalent to the claim that every $d$-distribution is \emph{primitive}, meaning that it is not of the form $(x_1,\dots,x_j)^n$ for some $j \ge 1$ and $n \ge 2$, where the exponent means repeated concatenation. And, indeed, if $(x_1,\dots,x_{d-1}) = (x_1,\dots,x_j)^n$, then
    \(
        n \cdot \left(\sum_{i=1}^{j} x_i \right) = \sum_{i=1}^{d-1} x_i = d.
    \)
    Since we also have $n \cdot j = d-1$, we have that $n$ divides both $d$ and $d-1$, so $n=1$, proving the claim.

    Now, we construct our selection strategy $f$ as follows. For each equivalence class $[\vec{x}] \in \mathcal{S}_d / [\cdot]$, pick a representative $(x_1,\dots,x_{d-1}) \in [\vec{x}]$ such that $x_1 \ge 2$. Then, let
    \(
        f(x_1,\dots,x_{d-1}) = 1
    \)
    and
    \(
        f(x_i, \dots, x_{d-1}, x_1, \dots,x_{i-1}) = 1+(d-i)
    \)
    for all $i \in [2,d-1]$. We have that the $f(\vec{x})$-th entry of $\vec{x}$ equals $x_1 \ge 2$ for each $\vec{x} \in [(x_1,\dots,x_{d-1})]$, so $f$ is a selection strategy. Furthermore, for each equivalence class $[\vec{x}] \in \mathcal{S}_d / [\cdot]$ and $i \in [d-1]$, we have
    \(
        |f^{-1}(i) \cap [\vec{x}]| = 1.
    \)
    Therefore, the selection strategy is equitable.
\end{proof}

\begin{theorem} \label{theorem:ep-hypergraph-upper}
    We have $\lmp_d(n) \le \left(\frac{1}{d!} + o_d(1)\right) \frac{n^{d-1}}{\lg n}$, and there is a $d$-clique partition realizing this bound in which the number of $d$-cliques is $O_d(n^d/\lg^3(n))$. Furthermore, a $d$-clique partition realizing this upper bound can be constructed deterministically in time $O(n^d / d!)$.
\end{theorem}

\begin{proof}
    More precisely, we prove that for every $d \ge 2$, there are functions $r_d(n) = o_d(n^{d-1}/\lg n)$ and $h_d(n) = O_d(n^d/\lg^3(n))$ such that every $d$-uniform hypergraph $H$ on $n$ vertices can be written as an edge-disjoint union of $d$-cliques such that each vertex is in at most
    \begin{equation} \label{eq-ep-bound}
       g_d(n) := \frac{1}{d!} \cdot \frac{n^{d-1}}{\lg n} + r_d(n) =  \left(\frac{1}{d!} + o_d(1)\right) \frac{n^{d-1}}{\lg n}
    \end{equation}
    $d$-cliques, where the number of $d$-cliques is at most $h_d(n)$.
    Concretely, we will take  
    \[
    r_d(n) = \frac{4^d}{d!} \cdot 2^{d 2^{100d^2}}  n^{d-1} \frac{\lg \lg n}{\lg^2 n}, \quad h_d(n) = 2^{d 2^{100d^2}} \frac{n^d}{\lg^3 n},
    \]
    which essentially amounts to assuming that $n$ is much larger than $d$, since for $n \leq 2^{2^{100d^2}}$ we have 
    \[
    n^d \leq \left(2^{2^{100d^2}}\right)^{d} = 2^{d2^{100d^2}} \leq \min\{ 
    r_d(n), h_d(n)\},
    \]
    and thus the statement holds trivially. We can thus assume from now on that $n > 2^{2^{100d^2}}$, noting that have made no effort to optimize the dependency on $d$; taking a doubly-exponential function so that even $\lg \lg n$ is much larger than $d$.


    
    The proof of the theorem proceeds by strong induction on $d$. The base case $d=2$ corresponds to~\Cref{thm:ep}, and for the inductive case, with $d \geq 3$, the proof is by strong induction on $n$.
    
    First, we partition the vertices of $H$ into parts $P_1,\dots,P_{d-1}$ such that $||P_i| - |P_j|| \le 1$ for all $i,j \in [d-1]$. Let $f$ be an equitable selection strategy on $\mathcal{S}_d$, which exists by~\Cref{lemma:exists-equitable}. For each $d$-distribution $\vec{x} = (x_1,\dots,x_{d-1}) \in \mathcal{S}_d$, we create a set of $d$-clique covers as follows. Let $i = f(\vec{x})$, and
    \[
        \mathcal{A}_{\vec{x}} := \{S \subseteq V(H) \setminus P_i : |S \cap P_j| = x_j \ \text{for all} \ j \in [d-1] \setminus \{i\}\}.
    \]
    Note that $\mathcal{A}_{\vec{x}}$ consists of sets of vertices of size $k-x_i$ each.
    For each $S \in \mathcal{A}_{\vec{x}}$, let $H_{\vec{x},S}$ be the $x_i$-uniform hypergraph defined by
    \(
        V(H_{\vec{x},S}) L= P_i\) 
       and
        \( E(H_{\vec{x},S}) = \{T : T \subseteq P_i,\; S \cup T \in E(H)\}.\)
    Let $\mathcal{C}_{\vec{x}, S}$ be an $x_i$-clique partition of $H_{\vec{x},S}$ such that each vertex of $P_i$ is in at most $g_{x_i}(|P_i|)$ many $x_i$-cliques and $|\mathcal{C}_{\vec{x}, S}| \le h_{x_i}(|P_i|)$, which exists by induction.
    Given an $x_i$-clique $(A_1,\dots,A_{x_i})$ in $\mathcal{C}_{\vec{x}, S}$, we can extend it to a $d$-clique $S \oplus \{A_1,\dots,A_{x_i}\} \subseteq H$, which is defined as the $d$-clique in $H$ whose parts are $\{\{v\} : v \in S\} \cup \{A_1,\dots,A_{x_i}\}$. Let
    \[
        \mathcal{C}'_{\vec{x}, S} = \{S \oplus \{A_1,\dots,A_{x_i}\} : (A_1,\dots,A_{x_i}) \in \mathcal{C}_{\vec{x},S}\} \qquad \text{and} \qquad 
        \mathcal{C} := \bigcup_{\substack{\vec{x} \in \mathcal{S}_d \\ S \in \mathcal{A}_{\vec{x}}}} \mathcal{C}'_{\vec{x}, S}.
    \]
    First, we argue that $\mathcal{C}$ is a $d$-clique partition of $G$.
    Let $e \in E(H)$, and let $\vec{x} = (x_1,\dots,x_{d-1}) \in \mathcal{S}_d$ be such that $e \cap P_i = x_i$ for all $i \in [d-1]$. Let $i = f(\vec{x})$. Then, there is a unique $S \in \mathcal{A}_{\vec{x}}$ such that $S \subseteq e$, and $e$ is uniquely covered by a $d$-clique in $\mathcal{C}'_{\vec{x},S}$. On the other hand, $d$ is not covered by any $d$-clique in $\mathcal{C}'_{\vec{x}',S'}$ for any $(\vec{x'},S') \neq (\vec{x},S)$.

    Next, we argue that each vertex is in at most $g_d(n)$ many $d$-cliques from $\mathcal{C}$.
    Fix a vertex $v \in V(H)$, and suppose that $v \in P_i$. There are two types of $d$-cliques from $\mathcal{C}$ that we consider:
    \begin{enumerate}
        \item (\textbf{Core}) These are the $d$-cliques in some $\mathcal{C}'_{\vec{x}, S}$ for which $f(\vec{x}) = i$.
        \item (\textbf{Auxiliary}) These are the $d$-cliques in some $\mathcal{C}'_{\vec{x}, S}$ for which $f(\vec{x}) \neq i$.
    \end{enumerate}
    
   We now bound these separately, for which we will first need some mildly technical conditions on the functions $r_d$ and $h_d$:
\begin{claim}\label{claim:properties}
        The functions $r_d$ and $h_d$ satisfy the following properties.
        \begin{enumerate}
        \item $r_d(n) \geq 4\frac{n}{d} r_{d-1}(n)$, for every $n, d$.
        \item $r_d(n+1) \geq r_d(n)$ and $h_d(n+1) \geq h_d(n)$, for every $n, d$. 
        \item $r_d(\ceil{n/(d-1)}) \leq  1.1\frac{r_d(n)}{(d-1)^{d-1}}$, for every $n \geq 2^{2^{100d^2}}$.
    \end{enumerate}
    \end{claim}

The proof of~\Cref{claim:properties} amounts to simple calculations and is thus deferred to~Appendix~\ref{sec:appendix}.

    \begin{claim}\label{claim:inner}
        The number of core $d$-cliques containing $v$ is at  most $\frac{1}{d!} \cdot \frac{n^{d-1}}{\lg n} + \frac{2}{3}r_d(n)$.
    \end{claim}
 \begin{innerproof}{\Cref{claim:inner}}
         By the inductive hypothesis (using induction on $d$ if $x_i < d$, and on $n$ if $x_i = d$), the number of core $d$-cliques containing $v$ is at most
    \begin{multline*}
        \sum_{(x_1,\dots,x_{d-1}) \in f^{-1}(i)} \left(\prod_{j \in [d-1] \setminus \{i\}} \binom{|P_j|}{x_j}\right) \cdot g_{x_i}(|P_i|) \\
        \leq \sum_{(x_1,\dots,x_{d-1}) \in f^{-1}(i)} \left(\prod_{j \in [d-1] \setminus \{i\}} \frac{|P_j|^{x_j}}{x_j!}\right) \cdot \left( \frac{1}{x_i!} \cdot \frac{|P_i|^{x_i-1}}{\lg |P_i|} + r_{x_i}(|P_i|) \right) := \dagger.
     \end{multline*}
    Since each part $P_k$ has size at most $\lceil n / (d-1) \rceil$, and $\sum_{k=1}^{d-1} x_k = d$, we have 
    \[
    \dagger \leq \sum_{(x_1,\dots,x_{d-1}) \in f^{-1}(i)} \left(\frac{1}{x_1! \cdot \ldots \cdot x_{d-1}!}\right) \cdot \left\lceil \frac{n}{d-1} \right\rceil^{d-1} \cdot \left(\frac{1}{\lg |P_i|} + \frac{x_i!}{|P_i|^{x_i-1}} r_{x_i}(|P_i|) \right).
    \]
    We now split the last summation according to whether $x_i = d$ or not. We thus have
\begin{multline*}
\dagger \leq \left\lceil \frac{n}{d-1} \right\rceil^{d-1} \cdot\frac{1}{d!} \left(\frac{1}{\lg |P_i|} + \frac{d!}{|P_i|^{d-1}} r_{d}(|P_i|) \right) \\
+  \left\lceil \frac{n}{d-1} \right\rceil^{d-1} \cdot \sum_{\substack{(x_1,\dots,x_{d-1}) \in f^{-1}(i) \\ x_i \neq d}} \left(\frac{1}{x_1! \cdot \ldots \cdot x_{d-1}!}\right) \cdot \left(\frac{1}{\lg |P_i|} + \frac{(d-1)!}{|P_i|^{d-2}} r_{d-1}(|P_i|) \right).
\end{multline*}

Now we note that \begin{equation}\label{eq:identity}
\sum_{(x_1,\dots,x_{d-1}) \in \mathcal{S}_d} \left(\frac{d!}{x_1! \cdot \ldots \cdot x_{d-1}!}\right) = (d-1)^{d},
\end{equation}
since both sides count the number of distinct strings of length $d$ over an alphabet of size $d-1$. Then, by equitability of $f$, we have that 
\[
 \sum_{\substack{(x_1,\dots,x_{d-1}) \in f^{-1}(i)}} \left(\frac{d!}{x_1! \cdot \ldots \cdot x_{d-1}!}\right) = \frac{1}{d-1}\sum_{(x_1,\dots,x_{d-1}) \in \mathcal{S}_d} \left(\frac{d!}{x_1! \cdot \ldots \cdot x_{d-1}!}\right)  = (d-1)^{d-1}.
\]
Also, there is a unique $d$-distribution $\vec{x} = (x_1, \dotsc, x_d)$ such that $x_i = d$ (and $x_j = 0$ for all $j \neq i$), and for such $\vec{x}$ we must have $f(\vec{x}) = i$.
Therefore, 
\begin{equation}\label{eq:comb_identity}
\sum_{\substack{(x_1,\dots,x_{d-1}) \in f^{-1}(i) \\ x_i \neq d}} \left(\frac{1}{x_1! \cdot \ldots \cdot x_{d-1}!}\right) = \frac{(d-1)^{d-1}-1}{d!}.
\end{equation}

Using~\Cref{eq:comb_identity} in our last inequality regarding $\dagger$, and introducing notation $w_d(x) := \frac{d!}{x^{d-1}} r_{d}(x)$,
we have
\[
\dagger \leq \frac{\ceil{n/(d-1)}^{d-1}}{d!} \cdot \left[\frac{(d-1)^{d-1}}{\lg |P_i|} + w_d(|P_i|) + ((d-1)^{d-1}-1)w_{d-1}(|P_i|) \right].
\]
As $r_d(x) \geq 4 \frac{x}{d}r_{d-1}(x)$ by~\Cref{claim:properties}, we have that $w_{d-1}(x) \leq w_d(x) / 4$.
Using this, we get
\begin{align*}
\dagger &\leq \frac{\ceil{n/(d-1)}^{d-1}}{d!} \cdot \left(\frac{(d-1)^{d-1}}{\lg |P_i|} + w_d(|P_i|) \left[ \frac{(d-1)^{d-1}}{4} + 1 \right] \right) \\
&\leq \frac{\ceil{n/(d-1)}^{d-1}}{d!} \cdot \left(\frac{(d-1)^{d-1}}{\lg (\lfloor n/(d-1)\rfloor)} + \frac{d!}{|P_i|^{d-1}} r_d(|P_i|) \left[ \frac{(d-1)^{d-1}}{4} + 1 \right] \right). 
\end{align*}

As $\ceil{n/(d-1)} \leq (n+d-2)/(d-1)$,  $|P_i| \geq \floor{n/(d-1)} \geq \frac{n - d +2}{d-1}$, and $r_d$ is an increasing function,  we now have
\[
\dagger \leq \frac{(n+d-2)^{d-1}}{d! \lg(\lfloor n/(d-1)\rfloor)} + \left(\frac{n+d-2}{n-d+2}\right)^{d-1} r_d\left(\left\lceil \frac{n}{d-1} \right\rceil \right) \left( \frac{(d-1)^{d-1}}{4} + 1 \right),
\]
and since $n \geq 2^{2^{200d^2}}$, it is easy to see that $\left(\frac{n+d-2}{n-d+2}\right)^{d-1} \leq 1.1$, from where
\begin{align*}
\dagger &\leq \frac{(n+d-2)^{d-1}}{d! \lg(\lfloor n/(d-1)\rfloor)} + 1.1 \cdot 1.1 r_d(n) \left(\frac{1}{4} + \frac{1}{(d-1)^{d-1}} \right) \\
&\leq \frac{(n+d-2)^{d-1}}{d! \lg(\lfloor n/(d-1)\rfloor)} + 1.1 \cdot 1.1 r_d(n) \left(\frac{1}{4} + \frac{1}{4} \right) \tag{As $d \geq 3$}\\
&\leq  \frac{(n+d-2)^{d-1}}{d! \lg(\lfloor n/(d-1)\rfloor)} + 0.61 r_d(n).
\end{align*}
To conclude, it suffices to prove that
\(
\frac{(n+d-2)^{d-1}}{d! \lg(\floor{n/(d-1)})} \leq \frac{n^{d-1}}{d!\lg n } + 0.05 r_d(n).
\)
First, use that $(n+d-2)^{d-1} \leq (n+d)^{d-1} \leq n^{d-1} + d^{d+1} \cdot n^{d-2}$, where the last inequality follows from upper bounding all but the first term in the binomial expansion of $(n+d)^{d-1}$. It is then easy to see that $d^{d+1} \cdot n^{d-2} \leq \frac{1}{100} r_d(n)$. Thus, it only remains to prove that
\(
\frac{n^{d-1}}{d! \lg(\floor{n/(d-1)})} \leq \frac{n^{d-1}}{d!\lg n } + 0.04 r_d(n),
\)
which follows from writing
\begin{align*}
\frac{n^{d-1}}{d! \lg(\floor{n/(d-1)})} &\leq \frac{n^{d-1}}{d! (\lg(n-d+2) - \lg(d-1))} = \frac{n^{d-1}}{d! \lg n} \cdot \frac{\lg n}{\lg(n-d+2) - \lg(d-1)}\\
&= \frac{n^{d-1}}{d! \lg n} \cdot \left(1 + \frac{\lg n - \lg(n-d+2) + \lg(d-1)}{\lg(n-d+2) - \lg(d-1)} \right) \\
& \leq \frac{n^{d-1}}{d! \lg n} \cdot \left(1 + \frac{3\lg d}{\lg n} \right) =  \frac{n^{d-1}}{d! \lg n} + \frac{3 \lg d \cdot n^{d-1}}{d!\lg^2 n} \\ & \leq  \frac{n^{d-1}}{d! \lg n} + 0.01 r_d(n).\qedhere
\end{align*}
 \end{innerproof}

    \begin{claim}\label{claim:outer}
        The number of auxiliary $d$-cliques containing $v$ is at  most $ \frac{1}{3}r_d(n)$.
    \end{claim}
    The proof of~\Cref{claim:outer} is similar to the previous one, and thus deferred to~Appendix~\ref{sec:appendix}.
    From these two claims it immediately follows that $v$ is contained in at most $g_d(n)$ different $d$-cliques.
    We also need, in order to preserve our inductive hypothesis, to establish the following claim, whose proof is also in~Appendix~\ref{sec:appendix}.
    \begin{claim}\label{claim:count}
        We have $|\mathcal{C}| \leq h_d(n)$.
    \end{claim}

    Finally, it remains to show that the construction can be deterministically carried out in time $O(n^d / d!)$. Let $T_d(n)$ be the time complexity of the above construction, and we want to show that there is an absolute constant $C$ such that $T_d(n) \le C \cdot n^d / d!$. The proof is again by a strong double induction, first on $d$ and then on $n$. We have
    \begin{align*}
        T_d(n) &\le \sum_{\vec{x} \in \mathcal{S}_d} \left(\prod_{i \in [d-1] \setminus \{f(\vec{x})\}} \binom{|P_i|}{x_i}\right) T_{x_{f(\vec{x})}}(|P_{f(\vec{x})}|) \\
        &\le \sum_{\vec{x} \in \mathcal{S}_d} \left(\prod_{i \in [d-1] \setminus \{f(\vec{x})\}} \frac{|P_i|^{x_i}}{x_i!}\right) T_{x_{f(\vec{x})}}(|P_{f(\vec{x})}|) \\
        &\le C \cdot \sum_{\vec{x} \in \mathcal{S}_d} \prod_{i \in [d-1]} \frac{|P_i|^{x_i}}{x_i!} \tag{by induction} \\
        &= \frac{C}{d!} \sum_{\vec{x} \in \mathcal{S}_d} \left( \prod_{i \in [d-1]} |P_i|^{x_i} \right) \binom{d}{x_1, \dots, x_{d-1}} \\
        &= C \cdot \frac{n^d}{d!}. \qedhere
    \end{align*}
\end{proof}

\section{Density-aware bounds}\label{sec:nechiporuk}

This section studies biclique partitions of graphs in terms of their edge density $\gamma := |E(G)|/ \binom{|V|}{2}$. 
The number of graphs on $n$ vertices with density $\gamma$ is naturally $\displaystyle \binom{\binom{n}{2}}{\gamma \binom{n}{2}}$. We can understand this quantity through the following consequence of Stirling's approximation.
\begin{remark}\label{remark:entropy}
    Let $\gamma_n \in (0,1)$ be a sequence such that $n \gamma_n \to \infty$ and $n (1-\gamma_n) \to \infty$. Then,
    \[
        \binom{n}{n \gamma_n} \sim \frac{2^{h_2(\gamma_n) n}}{\sqrt{2 \pi \gamma_n (1-\gamma_n) n}},
    \]
    where $h_2(x) := -x\lg(x) - (1-x)\lg(1-x)$ is the binary entropy function. 
\end{remark}

 A useful consequence of~\Cref{remark:entropy} is that $\lg\binom{n}{\gamma n} \sim n h_2(\gamma)$.
 Thus, the information theoretic lower bound, as in~\Cref{thm:lower-bound-cover}, yields
\[
\totalWeight(n, \gamma) \geq \left(\frac{1}{2} + o(1)\right) \cdot h_2(\gamma) \frac{n^2}{\lg n}.
\]

Our starting point for proving upper bounds in this setting is a nice result of Nechiporuk~\cite{nechiporuk}.
\nechiporuk*


Now, similarly to the proof of~\Cref{thm:ep} in the previous section, we will extend Nechiporuk's result to general graphs with an optimal constant by using tournaments.

For the proof, we will use the following combinatorial inequality.

\begin{remark}\label{remark:choosing}
    For any pair of sequences of integers $a_1, \ldots, a_k$ and $b_1, \ldots, b_k$ with $a_i \geq b_i$, we have
    \[
    \prod_{i=1}^{k} \binom{a_i}{b_i} \leq \binom{a_1 + \cdots + a_k}{b_1 + \cdots + b_k}.
    \]
\end{remark}

To see why~\Cref{remark:choosing} is true, consider that for choosing $B = b_1 + \cdots + b_k$ out of $A = a_1 + \cdots + a_k$ objects, one can always limit oneself to choose $b_1$ objects among the first $a_1$, $b_2$ among the second $a_2$ objects, and so on. The number of ways of choosing with this additional restriction, on the left side of~\Cref{remark:choosing}, cannot exceed the total number of ways $\binom{A}{B}$, on the right side.


\sparseces*
Before proceeding with the proof, let us roughly sketch its main high-level idea dating back to~\cite{nechiporuk}, and how~\Cref{remark:choosing} comes into play. While in the proof of~\Cref{thm:ep} we just split $V(G)$ into parts of equal size, the proof of~\Cref{thm:gamma-2} requires two levels of splitting: first we split $V(G)$ into parts $P_1, P_2, \dots$ of a fixed equal size, this time depending on the density $\gamma$, and set up a tournament\footnote{Note that it is unnecessary for the tournament to be almost regular for this application, since we only seek to bound the weight of the constructed biclique partition rather than the load on each vertex.} between the parts just as in~\Cref{thm:ep}. But then, given a vertex $v$, we can view its adjacencies with respect to $P_i$ as a binary string, which we further split into more granular \emph{slices}, which do not have a fixed size; the slices are constructed so that they have roughly the same ``entropy'', instead of roughly the same size. The \emph{entropy} of a string of length $\ell$ with Hamming weight $w$ is $\lg \binom{\ell}{w}$. This way, after slicing the binary string $b_v$ corresponding to $v$'s adjacency row, and constructing bicliques that will take place within a slice, the total weight incurred will be at most
\[
    \sum_{v \in V} \sum_{\text{ slice index } i} \lg \binom{\ell_i}{w_i} = \lg \left( \prod_{v \in V} \prod_{\text{ slice index } i} \binom{\ell_i}{w_i} \right) \le ^{\text{Rem.}~\ref{remark:choosing}} (1+o(1))\lg \binom{n^2/2}{|E|} \sim^{\text{Rem.}~\ref{remark:entropy}}  h_2(\gamma) \frac{n^2}{2},
\]
which explains the shape of the obtained bound.
Let us now proceed with the proof.
\begin{proof}[Proof of~\Cref{thm:gamma-2}]
    The lower bound is via information theory as in the proof of \Cref{thm:lower-bound-cover}, and it remains to prove the upper bound. Let $G$ be an $n$-vertex graph with edge density $\gamma$, where $\max\{\gamma^{-1}, (1-\gamma)^{-1}\} = n^{o(1)}$. Let $\partSize := \floor{\frac{\lg^2 n}{h_2(\gamma)}}$ and partition $V(G)$ into parts $P_1, \dots, P_{\ceil{n/\partSize}}$ of size at most $\partSize$ each. Given a vertex $v \in V(G)$, let $g(v)$ be the index of the part $P_i$ containing $v$. Orient $[\ceil{n/\partSize}]$ according to a tournament, and write $R(i,j)$ if there is a directed edge from $i$ to $j$ in this tournament.

    Label the vertices of $G$ so that $V(G)$ is identified with the set \( \{(i,j) : i \in [\ceil{n/\partSize}], j \in [|P_i|]\}.\) We write $(i,[a,b])$ as an abbreviation for $\{(i,j) : j \in [a,b]\}$. Given $v \in V(G)$, let $N_{i,a,b}(v) := N(v) \cap P_i \cap (i,[a,b])$. We define a strictly increasing sequence of indices $r(v,i,0),r(v,i,1),\ldots,r(v,i,\ell) \in [0,|P_i|]$ as follows. Let $r(v,i,0) = 0$. Then, if $r(v,i,j-1)$ has been defined and is not equal to $|P_i|$, let
    \begin{equation} \label{eqn-entropy}
        r(v,i,j) = \max \left\{x \in [r(v,i,j-1)+1,|P_i|] : \binom{x - r(v,i,j-1)}{|N_{i,r(v,i,j-1)+1,x}(v)|} < \frac{n}{\partSize^4} \right\}.
    \end{equation}
    With these indices defined, let
    \[
        \sliceIndices(v,i) := \{ (r(v,i,j-1)+1,r(v,i,j)) :  j \in [\ell]\}.
    \]

    Given $1 \le a \le b \le |P_i|$ and nonempty $S \subseteq (i,[a,b])$, let
    \[
        A_i(a,b,S) = \{v \in V(G) : N_{i,a,b}(v) = S \ \text{and} \ (a,b) \in \sliceIndices(v,i) \ \text{and} \ R(g(v),i)\},
    \]
    and let
    \[
        \mathcal{S}_{i,a,b} = \{\emptyset \neq S \subseteq (i,[a,b]) : A_i(a,b,S) \neq \emptyset\}.
    \]
    Note that $(S, A_i(a,b,S))$ is a biclique for all $i \in [\ceil{n/\partSize}$, $1 \le a \le b \le |P_i|$, and $S \in \mathcal{S}_{i,a,b}$. For ease of notation, let $\mathcal{I} := \{ (i, a, b, S) : i \in [\ceil{n/r}], 1 \leq a \leq b \leq |P_i|, S \in \mathcal{S}_{i, a, b} \}$.
    
    We define the following set of bicliques:
    \[
\mathcal{C} := \underbrace{\{ (S, A_i(a, b, S)) : (i, a, b, S) \in \mathcal{I}\}}_{\textcolor{gray!90}{\text{across parts}}} \cup \underbrace{\{(\{u\}, \{v\}) : \{u, v\} \in E(G), g(u) = g(v)\}}_{\textcolor{gray!90}{\text{within parts}}} 
\]

    First, we claim that $\mathcal{C}$ is a biclique partition. Let $\{u,v\} \in E(G)$. If $g(u) = g(v)$, then $\{u,v\}$ is covered uniquely by the biclique $(\{u\},\{v\}) \in \mathcal{C}$. Otherwise, $g(u) \neq g(v)$, and we assume without loss of generality that $R(g(u),g(v))$. According to our labeling of $V(G)$, we have $v = (g(v),j)$ for some $j \in [|P_{g(v)}|]$. There is a unique $(a,b) \in \sliceIndices(u,g(v))$ with $j \in [a,b]$. Then, $\{u,v\}$ is covered uniquely by the biclique $(N_{g(v),a,b}(u), A_{g(v)}(a,b,N_{g(v),a,b}(u))) \in \mathcal{C}$.

    Now, we bound the total weight of $\mathcal{C}$. For the following computations, note that our assumptions imply that $1/h_2(\gamma) = o(\gamma^{-1})$.
    
    First, there are the bicliques of the form $(\{u\},\{v\})$ for $g(u) = g(v)$, whose total weight is at most $O(\ceil{n/r} \cdot r^2) = O(n\lg^2(n)/h_2(\gamma)) = o(h_2(\gamma) \frac{n^2}{\lg n})$. Second, we have
    \(
    \sum_{(i, a, b, S) \in \mathcal{I}} |S| \leq \sum_{(i, a, b, S) \in \mathcal{I}} r.
    \)
    In this summation, we have $\ceil{n/\partSize}$ choices for $i$, at most $r^2$ choices for $a$ and $b$, at most $\partSize$ choices for $|S|$, and at most $n/\partSize^4$ elements of $\mathcal{S}_{i,a,b}$ of the chosen size $|S|$. Hence,
    \[
    \sum_{(i, a, b, S) \in \mathcal{I}} r \leq \frac{n^2}{r} = o\left(h_2(\gamma) \frac{n^2}{\lg n}\right).
    \]
    It remains to show that
    \[
        \sum_{(i, a, b, S) \in \mathcal{I}} |A_i(a,b,S)| \le \left(\frac{1}{2}+o(1)\right) \cdot h_2(\gamma) \frac{n^2}{\lg n}.
    \]

    We have
    \begin{align*}
        \sum_{(i, a, b, S) \in \mathcal{I}} |A_i(a,b,S)| &= \sum_{(i, a, b, S) \in \mathcal{I}} \sum_{v \in V(G)} \mathbb{1}_{[v \in |A_i(a,b,S)|]} \\
        &= \sum_{v \in V(G)} \sum_{(i, a, b, S) \in \mathcal{I}} \mathbb{1}_{[v \in |A_i(a,b,S)|]} \\
        &= \sum_{v \in V(G)} |\{(a,b) \in \sliceIndices(v,i) : R(g(v),i)\}|.
    \end{align*}
    Given $v \in V(G)$, $i \in [\ceil{n/\partSize}]$, and $(a,b) \in \sliceIndices(v,i)$ with $b < |P_i|$, we first observe that $0 < |N_{i,a,b}(v)| < b-(a-1)$; otherwise, we would have
    \[
        \frac{n}{\partSize^4} \le \binom{(b+1)-(a-1)}{|N_{i,a,b+1}(v)|} = \binom{(b+1)-(a-1)}{1} \le \partSize,
    \]
    a contradiction. Then, for any $(a,b) \in \sliceIndices(v,i)$ with $b < |P_i|$, we have
    \begin{align*}
        \frac{n}{\partSize^4} &\le \max\left\{\binom{(b+1)-(a-1)}{|N_{i,a,b}(v)|}, \binom{(b+1)-(a-1)}{|N_{i,a,b}(v)|+1}\right\} \\
        &\le \binom{b-(a-1)}{|N_{i,a,b}(v)|} \cdot \frac{b-(a-1)}{\min\{|N_{i,a,b}(v)|,b-(a-1)-|N_{i,a,b}(v)|\}} \le \binom{b-(a-1)}{|N_{i,a,b}(v)|} \cdot \partSize,
    \end{align*}
    from where
    \[
        1 \le \frac{\lg \binom{b-(a-1)}{|N_{i,a,b}(v)|}}{\lg(n) - \lg(\partSize^5)}.
    \]
    Thus,
    \begin{align*}
        \sum_{v \in V(G)} |\{(a,b) \in \sliceIndices(v,i) : b < |P_i|, R(g(v),i)\}| &\le \sum_{\substack{v \in V(G) \\ i \in [\ceil{n/\partSize}]\\ R(g(v),i)}} \sum_{\substack{(a,b) \in \sliceIndices(v,i) \\ b < |P_i|}} \frac{\lg \binom{b-(a-1)}{|N_{i,a,b}(v)|}}{\lg(n) - \lg(\partSize^5)} \\
        &= \frac{\lg\left( \prod_{\substack{v \in V(G) \\ i \in [\ceil{n/\partSize}]\\ R(g(v),i)}} \prod_{\substack{(a,b) \in \sliceIndices(v,i) \\ b < |P_i|}}\binom{b-(a-1)}{|N_{i,a,b}(v)|} \right)}{\lg(n) - \lg(\partSize^5)} \\
        &\le \frac{\lg \binom{\binom{n}{2}}{(\gamma \pm o(1)) \binom{n}{2}}}{\lg(n) - \lg(\partSize^5)} \tag{\Cref{remark:choosing}}\\
        &\le \frac{h_2(\gamma \pm o(1)) \cdot \binom{n}{2}}{\lg(n) - \lg(\partSize^5)} = \left(\frac{1}{2} + o(1)\right) \cdot h_2(\gamma) \frac{n^2}{\lg n}.
    \end{align*}
    Therefore,
    \begin{align*}
        \sum_{v \in V(G)} |\{(a,b) \in \sliceIndices(v,i) : R(g(v),i)\}| &\le \left(\frac{1}{2} + o(1)\right) \cdot h_2(\gamma) \frac{n^2}{\lg n} + \sum_{v \in V(G)} \ceil{n/\partSize} \\
        &= \left(\frac{1}{2} + o(1)\right) \cdot h_2(\gamma) \frac{n^2}{\lg n},
    \end{align*}
    as desired.

    Finally, we show that the above construction can be carried out deterministically in time $O(m)$, assuming $G$ is represented as an adjacency list. First, for each $y \in [r]$, we compute the maximum $x \in \mathbb{N}$ such that
    \(
        \binom{x}{y} < \frac{n}{r^4}
    \)
    and store the result in a lookup table.
    After partitioning $V(G)$ into parts $P_1, \dots, P_{\ceil{n/\partSize}}$, we compute the list $N(v) \cap P_i$ for each $v \in V(G)$ and $i \in \ceil{n/\partSize}$; this is so we can directly access $N(v) \cap P_i$ later without having to iterate through all of $N(v)$ each time. Now, for each $i \in [\ceil{n/r}]$, we maintain an array of linked lists $A_i$ indexed by $a$, $b$, and $S$ satisfying $1 \le a \le b \le |P_i|$ and $\emptyset \neq S \subseteq (i,[a,b])$ such that
    \(
        \binom{b-(a-1)}{|S|} < \frac{n}{r^4}.
    \) Observe that $|A_i| < |P_i|^2 \cdot r \cdot \frac{n}{r^4} \leq \frac{n}{r}$, where the $r$ term comes from the choices for $|S|$. Thus, adding over $i$, the total time allocating these arrays is $O(\frac{n^2}{r^2}) = O(\frac{n^2 h_2^2(\gamma)}{\lg^4 n})$, which is $O(m)$ since for $\gamma$ bounded away from $0$ we have $m = \Omega(n^2)$, and for $\gamma \to 0$ we have $h_2(\gamma) \sim \gamma \lg (1/\gamma)$, from where
    \[
    O\left(\frac{n^2 h_2^2(\gamma)}{\lg^4 n}\right) = O\left(\frac{n^2 \cdot \gamma^2 \lg^2(1/\gamma)}{\lg^4 n}\right) =  o\left(\frac{n^2\gamma^2}{\lg ^2 n}\right) = o(m).
    \]
    
    Then, for each $v \in V(G)$ such that $R(g(v),i)$, we compute $\sliceIndices(v,i)$ according to \eqref{eqn-entropy} by iterating through $N(v) \cap P_i$ and using our lookup table constructed above; this takes $O(N(v) \cap P_i)$ time. For each $(a,b) \in \sliceIndices(v,i)$, add $v$ to $A_i(a,b,N(v) \cap (i,[a,b]))$. For each $a$, $b$, and $S$ satisfying the above conditions, if $A_i(a,b,S) \neq \emptyset$, then output $(S,A_i(a,b,S))$; also, for each $\{u,v\} \in E(G[P_i])$, output $(\{u\}, \{v\})$. The complexity of the algorithm is thus
    \[
        \sum_{\substack{v \in V(G) \\ i \in [\ceil{n/\partSize}]}} O(N(v) \cap P_i) = O(m). \qedhere
    \]
\end{proof}

\subsection{Density-aware Erd\H{o}s--Pyber} \label{sec-density-ep}

Just as the Chung--Erd\H{o}s--Spencer theorem can be strengthened to the Erd\H{o}s--Pyber theorem, it is natural to ask if \Cref{thm:gamma-2} admits a similar strengthening. That is, given a graph $G$ of density $\gamma$, do we have $\lbp(G) \le (h_2(\gamma)/2 + o(1))\frac{n}{\lg n}$? The following example shows that this is not possible in general.

\begin{example}
    By \Cref{thm:lower-bound-cover}, there is a graph $H$ on $\floor{n/8}$ vertices such that
    \[
        \lbp(H) \ge \frac{\cover(H)}{\floor{n/8}} \ge \left(\frac{1}{16} - o(1)\right) \frac{n}{\lg n}.
    \]
    Let $G = H \sqcup \overline{K_{\ceil{7n/8}}}$, where $\overline{K_{\ceil{7n/8}}}$ is the empty graph on $\ceil{7n/8}$ vertices. Then, $G$ is an $n$-vertex graph with edge density $\gamma \le 1/64$ and $\lbp(G) = \lbp(H)$. Numerically, we have $h_2(1/64) < 1/8$, so
    \[
        \lbp(G) > \left(\frac{h_2(\gamma)}{2} + o(1)\right)\frac{n}{\lg n}.
    \]
\end{example}

On the other hand, we can prove a density-aware version of the Erd\H{o}s--Pyber theorem given a stronger assumption. Namely, we need bounds on the degree of each vertex rather than the total number of edges. Given $\vdensity \in [0,0.5]$, let $\lbp(n,\vdensity) := \max\{\lbp(G) : |V(G)| = n \ \text{and} \ \min\{d(v)/(n-1), 1 - d(v)/(n-1)\} \le \vdensity \ \text{for all} \ v \in V(G)\}$. Note that $\lbp(n,0.5) = \lbp(n)$.

\begin{theorem} \label{thm-density-ep}
    If $\vdensity \in [0,0.5]$ and $\vdensity^{-1} = n^{o(1)}$, then
    \[
        \lbp(n, \vdensity) \le \left(\frac{1}{2} + o(1)\right) \cdot h_2(\vdensity) \frac{n}{\lg n}.
    \]
    Furthermore, given a graph satisfying the degree constraints represented as an adjacency list, a biclique partition realizing the upper bound can be constructed in time $O(m)$ by a randomized algorithm with high probability.
\end{theorem}

Before proving the theorem, we need a lemma. For the following, let $d^+(v)$ be the outdegree of a vertex $v$ in a directed graph.
\begin{lemma} \label{lem-random-tournament}
    Let $G$ be an $n$-vertex graph such that $\min\{d(v)/(n-1), 1 - d(v)/(n-1)\} \le \vdensity \le 1/2$ for all $v \in V(G)$, where $\vdensity^{-1} = n^{o(1)}$. Let $\partSize := \floor{\frac{\lg^2 n}{h_2(\vdensity)}}$ and suppose that $V(G)$ is partitioned into parts $P_1, \dots, P_{\ceil{n/\partSize}}$ of size at most $\partSize$ each. Given a vertex $v \in V(G)$, let $g(v)$ be the index of the part $P_i$ containing $v$. Given a tournament $R$ on $[\ceil{n/\partSize}]$, write $d_R(v) := |\{w \in N_G(v) : R(g(v),g(w))\}|$. Then, there is a tournament $R$ on $[\ceil{n/\partSize}]$ such that (i) for every $i \in [\ceil{n/\partSize}]$, we have $d^+(i) \le (\frac{1}{2}+o(1)) \ceil{n/\partSize}$ and (ii) for every $v \in V(G)$, we have $\min\{d_R(v)/(n-1), 1-d_R(v)/(n-1)\} \le (\frac{1}{2} + o(1))\vdensity$.
\end{lemma} 
\begin{proof}
    We claim that a tournament $R$ in which the orientation of each edge is chosen uniformly and independently at random works with high probability. Note that our assumptions imply that $1/h_2(\vdensity) = o(\vdensity^{-1})$; in particular, $\partSize = n^{o(1)}$.

    First, we prove claim (i). Fix an element $i \in [\ceil{n/\partSize}]$. Without loss of generality, take $i=1$. Then, we have
    \(
        d^+(i) = \sum_{j=2}^{\ceil{n/\partSize}} \varepsilon_j,
    \)
    where each $\varepsilon_j \in \{0,1\}$ is a Bernoulli random variable. We have
    \(
        \mathbb{E}[d^+(i)] \le \frac{1}{2} \ceil{n/\partSize}.
    \)
    Let $\delta >0$. Then, by Hoeffding's inequality,
    \begin{align*}
        \mathbb{P}\left[d^+(i) \ge \left(\frac{1}{2} + \delta\right) \ceil{n/\partSize}\right] &\le \mathbb{P}\left[d^+(i)\ge \mathbb{E}[d^+(i)] + \delta \ceil{n/\partSize}\right] \\
        &\le \exp\left(\frac{-2(\delta \ceil{n/\partSize})^2}{\ceil{n/\partSize}-1}\right) \\
        &\le \exp(- \delta^2 \omega(n^{1-\varepsilon}))
    \end{align*}
    for all fixed $\varepsilon > 0$. Hence, by the union bound,
    \[
        \mathbb{P}\left[d^+(i) \ge \left(\frac{1}{2} + \delta\right) \ceil{n/\partSize} \ \text{for some} \ i \in \ceil{n/\partSize}\right] \le \exp(- \delta^2 \omega(n^{1-\varepsilon}))
    \]
    for all fixed $\varepsilon > 0$. Therefore, $d^+(i) \le (\frac{1}{2}+o(1)) \ceil{n/\partSize}$ for all $i \in [\ceil{n/\partSize}]$ with high probability.
    
    Next, we prove claim (ii). Fix a vertex $v \in V(G)$. Without loss of generality, $g(v) = 1$ and $d(v)/(n-1) \le \vdensity$ (the case where $1-d(v)/(n-1) \le \vdensity$ is symmetric). Then, we have
    \(
        d_R(v) = \sum_{i=2}^{\ceil{n/\partSize}} \varepsilon_i \cdot |N(v) \cap P_i|,
    \)
    where each $\varepsilon_i \in \{0,1\}$ is a Bernoulli random variable. Thus,
    \[
        \mathbb{E}[d_R(v)] \le \frac{1}{2} d(v) \le \frac{1}{2} \vdensity (n-1).
    \]
    Let $\delta > 0$. Then, by Hoeffding's inequality,
    \begin{align*}
        \mathbb{P}\left[d_R(v)/(n-1) \ge \left(\frac{1}{2} + \delta\right) \vdensity\right] &\le \mathbb{P}\left[d_R(v) \ge \mathbb{E}[d_R(v)] + \delta \vdensity (n-1)\right] \\
        &< \exp\left(\frac{-2 (\delta \vdensity (n-1))^2}{(\ceil{n/\partSize}-1) \partSize^2}\right) \\
        &\le \exp(-\delta^2 \omega(n^{1-\varepsilon}))
    \end{align*}
    for all fixed $\varepsilon > 0$. Hence, by the union bound,
    \[
        \mathbb{P}\left[\min\{d_R(v)/(n-1), 1-d_R(v)/(n-1)\} \ge \left(\frac{1}{2} + \delta\right) \vdensity \ \text{for some} \ v \in V(G)\right] \le \exp(-\delta^2 \omega(n^{1-\varepsilon}))
    \]
    for all fixed $\varepsilon > 0$. Therefore, $\min\{d_R(v)/(n-1), 1-d_R(v)/(n-1)\} \le (\frac{1}{2} + o(1))\vdensity$ for all $v \in V(G)$ with high probability.
\end{proof}

\begin{proof}[Proof of \Cref{thm-density-ep}]
    Let $G$ be an $n$-vertex graph such that $\min\{d(v)/(n-1), 1 - d(v)/(n-1)\} \le \vdensity \le 1/2$ for all $v \in V(G)$, where $\vdensity^{-1} = n^{o(1)}$. The construction of the biclique partition $\mathcal{C}$ is nearly identical to the one given in the proof of \Cref{thm:gamma-2}, the only difference being that rather than using an arbitrary tournament, we apply \Cref{lem-random-tournament} to $G$ to get a tournament $R$ on $[\ceil{n/\partSize}]$ such that $d^+(i) \le (\frac{1}{2}+o(1)) \ceil{n/\partSize}$ for all $i \in [\ceil{n/\partSize}]$ and $\min\{d_R(v)/(n-1), 1-d_R(v)/(n-1)\} \le (\frac{1}{2} + o(1))\vdensity$ for all $v \in V(G)$. In particular, the construction still takes $O(m)$ time, although it is now randomized.

    Fix a vertex $v \in V(G)$, and we count how many bicliques from $\mathcal{C}$ contain $v$. For the following computations, note that our assumptions imply that $1/h_2(\vdensity) = o(\vdensity^{-1})$.
    
    First, there are the bicliques of the form $(\{u\},\{v\})$ for $g(u) = g(v)$, of which there are at most $O(|P_{g(v)}|) = O(\lg^2(n)/h_2(\vdensity)) = o(h_2(\vdensity) \frac{n}{\lg n})$. Second, there are the bicliques of the form $(S, A_i(a,b,S))$ with $v \in S$. We upper bound this quantity by counting how many bicliques of the form $(S, A_i(a,b,S))$ with $\emptyset \neq S \subseteq P_{g(v)}$ there are. There are at most $\partSize^2$ choices for $a$ and $b$, at most $\partSize$ choices for $|S|$, and at most $n/\partSize^4$ elements of $\mathcal{S}_{i,a,b}$ of the chosen size $|S|$. Hence, the number of bicliques of the form $(S, A_i(a,b,S))$ with $v \in S$ is at most $O(n/\partSize) = o(h_2(\vdensity) \frac{n}{\lg n})$. Third, there are the bicliques of the form $(S, A_i(a,b,S))$ with $v \in A_i(a,b,S)$, which we now proceed to analyze.

    For any $(a,b) \in \sliceIndices(v,i)$ with $b < |P_i|$, we first observe that $0 < |N_{i,a,b}(v)| < b-(a-1)$; otherwise, we would have
    \[
        \frac{n}{\partSize^4} \le \binom{(b+1)-(a-1)}{|N_{i,a,b+1}(v)|} = \binom{(b+1)-(a-1)}{1} \le \partSize,
    \]
    a contradiction. Then, for any $(a,b) \in \sliceIndices(v,i)$ with $b < |P_i|$, we have
    \begin{align*}
        \frac{n}{\partSize^4} &\le \max\left\{\binom{(b+1)-(a-1)}{|N_{i,a,b}(v)|}, \binom{(b+1)-(a-1)}{|N_{i,a,b}(v)|+1}\right\} \\
        &\le \binom{b-(a-1)}{|N_{i,a,b}(v)|} \cdot \frac{b-(a-1)}{\min\{|N_{i,a,b}(v)|,b-(a-1)-|N_{i,a,b}(v)|\}} \le \binom{b-(a-1)}{|N_{i,a,b}(v)|} \cdot \partSize,
    \end{align*}
    from where
    \[
        1 \le \frac{\lg \binom{b-(a-1)}{|N_{i,a,b}(v)|}}{\lg(n) - \lg(\partSize^5)}.
    \]
    Thus,
    \begin{align*}
        |\{(a,b) \in \sliceIndices(v,i) : b < |P_i|, R(g(v),i)\}| &\le \sum_{\substack{i \in [\ceil{n/\partSize}]\\ R(g(v),i)}} \sum_{\substack{(a,b) \in \sliceIndices(v,i) \\ b < |P_i|}} \frac{\lg \binom{b-(a-1)}{|N_{i,a,b}(v)|}}{\lg(n) - \lg(\partSize^5)} \\
        &= \frac{\lg\left( \prod_{\substack{i \in [\ceil{n/\partSize}]\\ R(g(v),i)}} \prod_{\substack{(a,b) \in \sliceIndices(v,i) \\ b < |P_i|}}\binom{b-(a-1)}{|N_{i,a,b}(v)|} \right)}{\lg(n) - \lg(\partSize^5)} \\
        &\le \frac{\lg \binom{(1/2 + o(1)) n}{d_R(v)}}{\lg(n) - \lg(\partSize^5)} \tag{\Cref{remark:choosing}} \\
        &\le \frac{\left(\frac{1}{2} + o(1)\right) \cdot h_2(\vdensity) \cdot n}{\lg(n) - \lg(\partSize^5)} = \left(\frac{1}{2} + o(1)\right) \cdot h_2(\vdensity) \frac{n}{\lg n}.
    \end{align*}
    Therefore, the number of bicliques of the form $(S, A_i(a,b,S))$ with $v \in A_i(a,b,S)$ is
    \begin{align*}
        |\{(a,b) \in \sliceIndices(v,i) : R(g(v),i)\}| &\le \left(\frac{1}{2} + o(1)\right) \cdot h_2(\vdensity) \frac{n}{\lg n} + \ceil{n/\partSize} \\
        &= \left(\frac{1}{2} + o(1)\right) \cdot h_2(\vdensity) \frac{n}{\lg n},
    \end{align*}
    as desired.
\end{proof}

\section{Biclique representation of graphs}\label{sec:representations}
An interesting interpretation of~\Cref{thm:gamma-2} is that representing graphs by biclique partitions is information theoretically optimal for any density $\gamma$ such that $\max\{\gamma^{-1}, (1-\gamma)^{-1}\} = n^{o(1)}$. In general, when representing objects from a set $\mathcal{U}$, a representation is said to be \emph{succinct} if it uses $\lg(\mathcal{U}) (1+o(1))$ bits, and \emph{compact} if it uses $O(\lg(\mathcal{U}))$ bits~\cite{Navarro2016}. We consider the following two biclique-based representations for an undirected graph $G$:

\begin{enumerate}
    \item \textbf{SB} (Succinct Biclique Representation). Given a biclique partition $\mathcal{B} = \{(L_1, R_1), \ldots, (L_k, R_k)\}$ of $G$, 
    we store a list $L_{\mathcal{B}}$ of bicliques, where each biclique is stored as a structure with two sorted lists of vertices, corresponding to its sides. 
    \item \textbf{CB} (Compact Biclique Representation). On top of the SB representation, we store as well for each vertex $v$, a list of pointers to the bicliques $(L_i, R_i)$ it belongs to, and also one extra bit per biclique representing whether $v$ belongs to $L_i$ or to $R_i$.
\end{enumerate}

\begin{lstlisting}[caption=Example of a concrete data-layout for biclique representations., label={lst:listing-c},  style=myStyle, language=C]
/* ----- SB: Succinct Biclique Representation ----- */
typedef struct {
    int *L;     /* left  side (sorted vertex ids)  */
    int  L_len;
    int *R;     /* right side (sorted vertex ids)  */
    int  R_len;
} Biclique;

typedef struct {
    int  n;     /* number of vertices [0..n-1] */
    int  k;     /* number of bicliques         */
    Biclique *B;     /* array of k bicliques        */
} SB;

/* ----- CB: Compact Biclique Representation  ----- */
typedef struct {
    int  biclique;  /* index into sb.B (0..sb.k-1) */
    bool on_left;   /* bit for the side in which the vertex is */
} Incidence;

typedef struct {
    Incidence *a;   /* incidences for a vertex */
    int   len;
} IncList;

typedef struct {
    SB       sb;    /* the underlying SB representation */
    IncList *inc;   /* array of sb.n per-vertex lists  */
} CB;
\end{lstlisting}

Naturally, for these representations we will compute a small biclique partition leveraging~\Cref{thm:gamma-2}, or sometimes~\Cref{thm:ep}.
The total size of the SB representation, using a biclique partition $\mathcal{B}$ is $w(\mathcal{B}) \cdot \lg n$ (each vertex id uses $\lg n$ bits), and thus by~\Cref{thm:gamma-2},
 it allows to represent any $n$-vertex graph of density $\gamma$ with only $(1/2 + o(1)) \cdot h_2(\gamma)n^2$ bits, which indeed makes it a succinct representation. On the other hand, the CB representation uses three times as many bits: since there are no more than $O(n^2)$ bicliques in $\mathcal{B}$, each pointer can be stored in $2 \lg n$ bits, and thus, if $\ell_\mathcal{B}(v)$ is the number of bicliques a vertex $v$ belongs to, then the additional space used by the CB representation is
 \[
 \sum_{v \in V(G)} \ell_\mathcal{B}(v) \cdot (2\lg n + 1) \sim 2 \lg n \sum_{v \in V(G)} \sum_{(L_i, R_i) \in \mathcal{B}} \mathbb{1}_{[v \in (L_i \cup R_i)]} = 2 \lg n  \sum_{(L_i, R_i) \in \mathcal{B}} |L_i| + |R_i| = 2 \lg n \cdot w(\mathcal{B}).
 \]

 While alternative succinct and compact representations for graphs are known~\cite{Navarro2016, FARZAN201338}, we will show a few examples for which these biclique representations can result in algorithmic improvements. Furthermore, by representing bicliques with a succinct data structure for integer sequences, such as the one given by Golynski, Munro, and Rao~\cite{rankSelectSequences}, it is possible to transform the CB representation into a succinct one at the cost of accessing the $i$-th biclique of a given vertex in time $O(\lg \lg n)$ instead of $O(1)$. The main idea, as in the work of Hern\'andez and Navarro~\cite{Hernndez2013}, is that succinct data structures for an integer sequence $S$ still allow to efficiently identify the position of the $i$-th occurrence of an integer $v$ in $S$, a query known as $\textsf{select}(v, i)$, and thus if $S$ is the concatenation of the biclique parts, then identifying the $i$-th occurrence of $v$ in $S$ corresponds to identifying the $i$-th biclique in which $v$ appears.

\subsection{Independent set queries}

Suppose that, given a graph $G = (V, E)$, we wish to answer a sequence of \emph{independent set queries} of the form ``Is $S \subseteq V$ an independent set?''. This problem was studied first by Williams~\cite{williamsSubquadratic}, who showed that by preprocessing the graph in time $O(n^{2+ \varepsilon})$, for any $\varepsilon > 0$, one can answer each query in time $O( n^2 / (\varepsilon \lg n)^2)$. This was subsequently improved by a randomized preprocessing algorithm of Bansal and Williams~\cite{RegLemmaAlgorithms}, who also discussed interesting algorithmic implications, and more recently derandomized by Vassilevska Williams and Williams~\cite{subcubicEquiv}, resulting in the following theorem.
\begin{theorem}[\cite{subcubicEquiv}, {\cite[Thm.~2.3]{RegLemmaAlgorithms}}]\label{thm:BW}
    There is a deterministic algorithm that preprocesses an $n$-vertex graph in time $O(n^{2+\varepsilon})$ for any $\varepsilon > 0$, so that any batch of $\lg n$ independent set queries can be answered deterministically in time $O(\tfrac{1}{\varepsilon}n^2(\lg \lg n)^2/(\lg n)^{5/4})$.
\end{theorem}


It turns out that, by storing a graph in the SB representation, which can be computed in time $O(n^2)$ deterministically by~\Cref{thm:ep}, we can answer independent set queries in time $O(n^2/ \lg n)$, which improves on the naive $\Omega(|S|^2) = \Omega(n^2)$ algorithm. 
Note that, if the number of independent set queries is very small, say $O(\lg \lg n)$, and we want to limit the preprocessing time to e.g., $O(n^2 \lg n)$, then using~\Cref{thm:BW} as a black box yields a runtime of $O(n^2 \lg \lg n / \lg^{1/4} n)$ by taking $\varepsilon = \frac{\lg \lg n}{\lg n}$, whereas our algorithm runs in $O(n^2 \lg \lg n/ \lg n)$, and with strictly less preprocessing.

\begin{proposition}\label{prop:is-queries}
    Given an SB representation of an $n$-vertex graph based on a partition $\mathcal{B}$, an independent set query $S$ can be answered in time $O(w(\mathcal{B}) + |S|)$.
\end{proposition}
\begin{proof}
Let $\mathcal{B} = \{(L_1, R_1), \ldots, (L_k, R_k)\}$.
    Note that a subset $S \subseteq V$ is independent if and only if there is no biclique $(L_i, R_i) \in \mathcal{B}$ such that $|L_i \cap S| > 0$ and $|R_i \cap S| > 0$. 
    Therefore, we can create a bitvector $B_S[1..n]$ such that $B_S[v] = 1$ if $v \in S$ and $B_S[v] = 0$ otherwise. This takes $O(|S|)$ time, since it suffices to iterate over the elements of $S$ and mark their corresponding positions in $B_S$.
    Then for each $i$ we check whether $|S \cap L_i| > 0$ in time $O(|L_i|)$ simply by iterating over the vertices of $v \in L_i$, and checking whether $B_S[v] = 1$ for at least some $v$. We proceed analogously for $R_i$. The total runtime is thus 
    \[
    O(|S|) + \sum_{i=1}^k O(|L_i|) + O(|R_i|) = O\left(|S| + \sum_{i=1}^k |L_i| + |R_i|\right) = O(w(\mathcal{B}) + |S|). \qedhere
    \]
\end{proof}

\subsection{Cut queries}

We now consider queries in which we are given two disjoint subsets of vertices, $S, T \subseteq V$, and we wish to answer the number $c(S, T)$ of edges with one endpoint in $S$ and the other one in $T$. Applications of these queries to other algorithmic problems, in the particular case where $T = V \setminus S$, have been discussed by Apers et al.~\cite{cutFocs}, and Lee, Santha, and Zhang~\cite{leeSoda}. The case of arbitrary disjoint subsets $S, T$ is used, for instance, by Assadi et al.~\cite{assadi_et_al:LIPIcs.ESA.2021.7}.\footnote{Note that the case $c(S, V \setminus S)$ is not easier since
\(
2 \cdot c(S, T) = c(S, V \setminus S) + c(T, V \setminus T) - c(S \cup T, V \setminus (S \cup T)).
\)}
Note that if $\mathbf{A}_G$ is the adjacency matrix of the graph, and $\mathbf{x}_S$ and $\mathbf{x}_T$ are vectors representing $S$ and $T$ respectively, then $c(S, T) = \mathbf{x}_S^{\top} \mathbf{A}_G \mathbf{x}_T$. Unfortunately, the algorithms for matrix-vector multiplication achieving the runtime of~\Cref{thm:BW} work over the Boolean semi-ring, and thus they cannot be used as a black box for computing these products over $\mathbb{N}$. Our~\Cref{prop:is-queries} can be trivially extended to cut queries as follows.

\begin{proposition}\label{prop:cut-queries}
    Given an SB representation of an $n$-vertex graph based on a partition $\mathcal{B}$, a cut query can be answered in time $O(w(\mathcal{B}) + |S| + |T|)$.
\end{proposition}
\begin{proof}
    Let $\mathcal{B} = \{(L_1, R_1), \ldots, (L_k, R_k)\}$, and observe that
    \(
    c(S, T) = \sum_{i=1}^k |S \cap L_i| \cdot |T \cap R_i| + |S \cap R_i|\cdot |T \cap L_i|.
    \)
    As in the proof of~\Cref{prop:is-queries}, these intersections can be computed in $O(w(\mathcal{B}) + |S| + |T|)$ by constructing corresponding bitvectors $B_S$ and $B_T$.
\end{proof}

\subsection{Densest subgraph approximations}

 We now present a more interesting application of these biclique representations, which requires the CB representation. 

 The \emph{densest subgraph problem} asks for a non-empty subset $S \subseteq V$ that maximizes the degree density $\delta(S) := |E(G[S])|/|S|$. Algorithms for this problem are widely used in many practical scenarios, as for instance identifying dense clusters in social media networks. For a survey on this problem and its applications, we refer the reader to Lanciano et al.~\cite{surveyDensest}.

 Arguably, the most famous results regarding the densest subgraph problem were proved by Charikar~\cite{Charikar2000}, who proved that an optimal solution could be computed by linear programming, and that a $2$-approximation\footnote{We use $2$-approximation for consistency with the literature on the problem, but it is technically a $\frac{1}{2}$-approximation since it is a maximization problem.} could be computed in linear time (i.e., $\Theta(|V| + |E|)$), through an elegant greedy peeling algorithm. Charikar's algorithm is simply: start with $H_0 := G$, and then iteratively remove the vertex of smallest degree until the graph is empty (so $H_{i+1} := H_i \setminus \{ \arg\min_{v \in V(H_i)} \deg(v)\}$), and finally return the densest of all the subgraphs $H_i$ seen throughout the process. 

In this subsection we show that $o(|V|^2)$ approximation algorithms are possible if the graph is stored in the CB representation, by modifying  Charikar's algorithm.

\charikar*
For instance, setting $\alpha = \lg \lg n$, we get a $2\lg \lg n$ approximation in time $O(n ^2/\lg \lg \lg n) = o(n^2)$.
\begin{proof}[Proof of~\Cref{thm:charikar-subq}]

    The algorithm, whose pseudocode is presented in~\Cref{alg:ds-approx}, is parameterized by the desired $\alpha > 1$, and proceeds by rounds. 
     \begin{algorithm}
    \caption{Densest Subgraph Approximation}
    \label{alg:ds-approx}
    \begin{algorithmic}[1]
    \Require Graph $G = (V, E)$, assuming wlog that $V = \{1, \ldots, n\}$. Approximation parameter $\alpha > 1$

    \State $\textsf{maxDensity} \gets 0$
    \State $\textsf{densestSubgraph} \gets \varnothing$
    \State $t \gets 1$
    \State $H \gets \textsf{copy}(G)$
    \State $\textsf{removed} \gets [\textsf{false}, \ldots, \textsf{false}]$ \Comment{Array of length $n$, indexed from $1$}
    \While{$|V(H)| > 0$} 
        \State $\textsf{degreeSum} \gets 0$
        \State $\textsf{toRemove} \gets \varnothing$
        \ForAll{$v \in \{1, \ldots, n\}$} \Comment{In $O(n^2 / \lg n)$ time}
            \If{$\textsf{removed}[v] = \textsf{false}$}
           
            \State $d \gets H.\textsf{degree}(v)$  
            \State $\textsf{degreeSum} \gets \textsf{degreeSum} + d$
            \If{$d < t$} 
                \State $\textsf{toRemove} \gets \textsf{toRemove} \cup \{ v \}$
            \EndIf
             \EndIf
        \EndFor
        \State $\textsf{currDensity} \gets \textsf{degreeSum}/(2|V(H)|)$
        \If{$\textsf{currDensity} > \textsf{maxDensity}$}
            \State $\textsf{maxDensity} \gets \textsf{currDensity}$ 
            \State $\textsf{densestSubgraph} \gets \textsf{copy}(H)$ \Comment{In $O(n^2/\lg n)$ time}
        \EndIf
        \ForAll{$v \in \textsf{toRemove}$} \Comment{In $O(n^2 / \lg n)$ time}
           \State $H.\textsf{lazyRemove}(v)$ \Comment{$v$ will no longer be counted for the degree of other vertices}
           \State $\textsf{removed}[v] \gets \textsf{true}$
        \EndFor
        
        \State $t \gets \alpha \cdot t$ \Comment{Increase threshold}
    \EndWhile
    
    
\end{algorithmic}
\end{algorithm}
    Let $H_i$ be the graph at the beginning of round $i$, with $H_0 := G$. Then, at the end of round $i \geq 0$, all vertices whose degree in $H_i$ is less than $\alpha^i$ are deleted, resulting in a subgraph $H_{i+1}$. Naturally, the process stops when all vertices get removed due to having degree less than $\alpha^k$ for some $k$, and then the algorithm returns the densest among all the subgraphs $H_i$ seen throughout the rounds. To avoid introducing complicated data structures, instead of actually removing vertices from the graph, we will simply \emph{mark them as removed} in a way that ensures that they will not be considered in further rounds. More concretely, we will implement an operation $\textsf{lazyRemove}(v)$ such that a vertex $v$ lazily removed will no longer count toward the degree of other vertices.

    First, we analyze the runtime of~\Cref{alg:ds-approx}. Since each vertex has degree at most $n-1$ in $G$, the algorithm stops after $k$ rounds when $\alpha^k > n-1$, and thus in no more than $\lg_{\alpha}(n)$ rounds. The runtime of each round of~\Cref{alg:ds-approx} is $O(n^2/ \lg n)$ as long as we can query the degree of any number of vertices in $O(n^2 / \lg n)$ as well as lazily deleting any number of vertices in $O(n^2 / \lg n)$, and copying the graph in $O(n^2/ \lg n)$. Thus, provided we can support such operations with the described runtime, the total runtime is $O(n^2/\lg n) \cdot \lg_{\alpha}(n) = O(n^2/\lg \alpha)$ as desired. 

    \begin{claim}
        On a CB representation of weight $w$, we can query the degree of all vertices in time $O(w)$, and lazily delete an arbitrary subset of them in $O(w)$.\label{claim:queries}
    \end{claim}
    \begin{innerproof}{\Cref{claim:queries}}
        To compute the degree of a vertex $v$, it suffices to iterate over the list of bicliques it belongs to (i.e., \texttt{IncList} in~Listing~\ref{lst:listing-c}), and then for each such biclique $B = (L, R)$, if $v$ is on its left side (i.e., \texttt{on\_left}), then we add $|R|$ (i.e., \texttt{R\_len}) to its degree, and otherwise $|L|$ (i.e., \texttt{L\_len}). This takes time proportional to the number of bicliques containing $v$, or in other words, its load $\ell(v)$. Therefore the total runtime is $O(\sum_{v} \ell(v)) = O(w)$.  For the lazy deletion of a vertex $v$, it suffices to iterate over the bicliques $B = (L, R)$ containing $v$, and if $v \in L$ we decrease \texttt{L\_len} by $1$, and otherwise \texttt{R\_len} by 1. Since the degree computation of further vertices will be based on adding up \texttt{L\_len} or  \texttt{R\_len}, this ensures future degree computations will be correct. The runtime is again $O(\sum_{v} \ell(v)) = O(w)$. Finally, since all indices and pointers use $\Theta(\lg n)$ bits, and thus constantly many RAM words, the cost of copying is proportional to the number of  indices and pointers, which is $O(w)$.
    \end{innerproof}

    Now we prove correctness. First, we will need a standard idea from the densest-subgraph literature (see~\cite[Theorem 3]{surveyDensest}). We will use notation $\deg(v, S)$ for the degree of a vertex $v$ in the subgraph $G[S]$. 
    \begin{claim}\label{claim:optimal_density}
        Let $\delta^\star := \max \{ \delta(S) : S \subseteq V(G)\}$, and $S^\star \subseteq V$ a subset for which $\delta(S^\star) = \delta^\star$. Then, for every $v \in S^\star$, we have $\deg(v, S^\star) \geq \delta^\star$.
    \end{claim}
    \begin{innerproof}{\Cref{claim:optimal_density}}
            Let $v$ be any vertex in $S^\star$.
            By the optimality of $S^\star$ we have $\delta(S^\star) \geq \delta(S^\star \setminus \{ v\})$, and thus $\frac{|E(G[S^\star])|}{|S^\star|} \geq \frac{|E(G[S^\star])| - \deg(v, S^\star)}{|S^\star| -1}$, from where $\deg(v, S^\star) \geq |E(G[S^\star])|/|S^\star| = \delta^\star$.
            
    \end{innerproof}
        
  Now observe that at the end of round $i$, all vertices of degree less than $\alpha^i$ have been removed, and thus $\deg(v, V(H_{i+1})) \geq \alpha^i$ for every $v \in V(H_{i+1})$, and thus by the handshake lemma,
  \(
  |E(H_{i+1})| \geq \frac{\alpha^i}{2} \cdot |V(H_{i+1})|
  \).
  Therefore, as long as $|V(H_{i+1})| > 0$, we have
  \begin{equation}\label{eq:density_threshold}
  \delta(V(H_{i+1})) = \frac{|E(H_{i+1})|}{|V(H_{i+1})|} \geq \frac{\alpha^i}{2}.
  \end{equation}
  Let $i^\star$ be such that $1 \leq \alpha^{i^\star} \leq \delta^\star <\alpha^{i^{\star}+1}$, which exists since $i \mapsto \alpha^i$ is strictly increasing\footnote{Here we assume without loss of generality that $\delta^\star \ge 1$, since any cycle has degree density $1$, and for a tree $G$ on $n$ vertices the densest subgraph is $G$ itself, with density $1 - 1/n$, which can be easily checked in linear time.}. Then, note that we have  $S^\star \subseteq V(H_{i^\star + 1})$ since by~\Cref{claim:optimal_density} all vertices in $S^\star$ have degree at least $\delta^\star$ within $S^\star$ and thus could not be deleted during round $i^\star$ since the threshold was $\alpha^{i^\star} \leq \delta^\star$. As clearly $|S^\star| > 0$,~\Cref{eq:density_threshold} gives \[
    \delta(V(H_{i^\star + 1})) \geq \frac{\alpha^{i^{\star}}}{2} = \frac{\alpha^{i^\star + 1}}{2\alpha} > \frac{\delta^\star}{2\alpha}. 
  \]
As $H_{i^\star + 1}$ is one of the subsets considered by the algorithm, we conclude that~\Cref{alg:ds-approx} returns a $2\alpha$-approximation. Together with the runtime analysis above, this concludes the proof.
\end{proof}

\section{Finding large balanced bicliques}


In this section, we study the algorithmic problem of finding a balanced biclique $K_{t, t}$, for $t$ as large as possible, that is guaranteed by the edge density $\gamma := |E|/\binom{|V|}{2}$ of the input graph. In order to contextualize our results, we will consider known upper and lower bounds, as well as previous results of Mubayi and Tur\'an~\cite{mubayiFindingBipartiteSubgraphs2010}. These are summarized in~\Cref{fig:bilcique_finding}.  

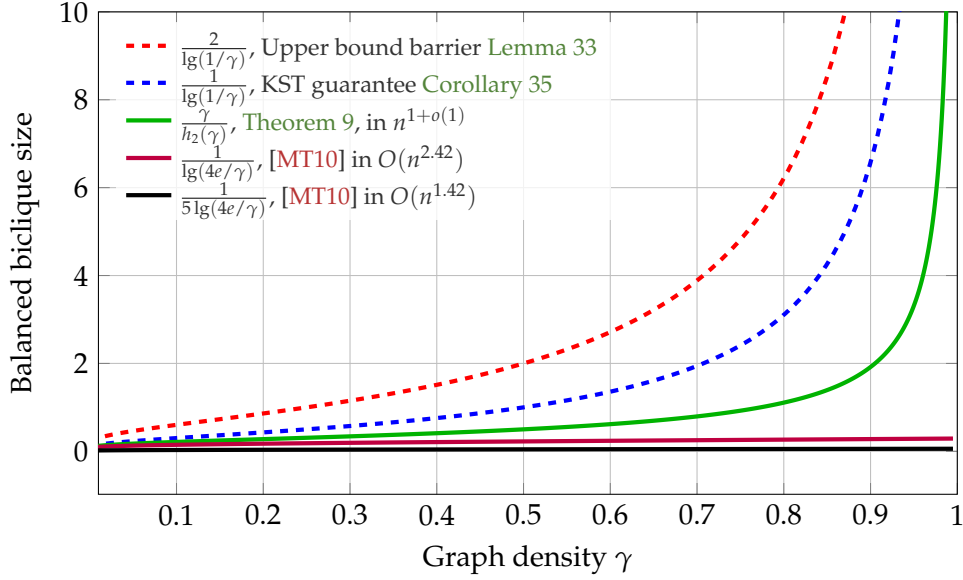
\begin{figure}
    \centering
\pgfmathdeclarefunction{logtwo}{1}{\pgfmathparse{ln(#1)/ln(2)}}
\pgfmathdeclarefunction{hbin}{1}{%
  \pgfmathparse{-(#1)*logtwo(#1) - (1-#1)*logtwo(1-#1)}%
}

\begin{tikzpicture}
  \begin{axis}[
    width=13cm, height=8cm,
    domain=0.001:0.995, samples=1000,
    xlabel={Graph density $\gamma$}, ylabel={Balanced biclique size},
    xmin=0.01,
    xmax=1,
    ymax=10,
    grid=both,
    legend cell align={left},
    legend style={at={(0.02,0.98)},anchor=north west,draw=none,fill=white,fill opacity=0.8}
  ]

    \addplot[ultra thick, dashed, red] { 2 / logtwo(1/x) };
    \addlegendentry{\footnotesize $ \frac{2}{\lg(1/\gamma)}$, Upper bound barrier \Cref{lemma:first-moment-t2}}

    \addplot[ultra thick, dashed, blue] { 1 / logtwo(1/x) };
    \addlegendentry{\footnotesize $ \frac{1}{\lg(1/\gamma)}$, KST guarantee \Cref{cor:kst-guarantee}}

    \addplot[ultra thick, green!70!black] { x / hbin(x) };
    \addlegendentry{\footnotesize $ \frac{\gamma}{h_2(\gamma)}$, \Cref{thm:rand-alg}, in $n^{1 + o(1)}$}

    \addplot[ultra thick, purple] { 1 / logtwo(4*2.7172/x) };
    \addlegendentry{\footnotesize $ \frac{1}{\lg(4e/\gamma)}$, \cite{mubayiFindingBipartiteSubgraphs2010} in $O(n^{2.42})$}

    \addplot[ultra thick, black] { 0.2 / logtwo(4*2.7172/x) };
    \addlegendentry{\footnotesize $ \frac{1}{5\lg(4e/\gamma)}$, \cite{mubayiFindingBipartiteSubgraphs2010} in $O(n^{1.42})$}

  \end{axis}
\end{tikzpicture}
    \caption{Comparison of the leading constant for $\Omega(\lg n)$ balanced bicliques, and algorithmic runtimes.}
    \label{fig:bilcique_finding}
\end{figure}

First, we shall show a limit to the largest biclique we can aim to find based only on edge density.
\begin{lemma}\label{lemma:first-moment-t2}
    Let $\gamma \in (0, 1)$. Then, with high probability, a random graph $G \sim G(n, \gamma)$ has edge density $\gamma \pm o(1)$ and yet does not contain a $K_{t,t}$ with $t \geq \frac{2}{\lg(1/\gamma)}\lg n$.
\end{lemma}
\begin{proof}
    The number of edges is $\gamma\binom{n}{2} \pm o(n^2)$ with high probability by a standard Chernoff bound.
    For the second part, we have \(
    \prob[G \text{ contains } K_{t,t}] < \binom{n}{t}\binom{n}{t} \gamma^{t^2} \leq \left(\frac{en}{t}\right)^{2t} \gamma^{t^2}
    \), and thus to show this probability goes to $0$  it suffices to show that $\lg(\left(\frac{en}{t}\right)^{2t} \gamma^{t^2}) \to -\infty$, and equivalently,
    \[
    2t \lg(n) + 2t\lg(e) - 2t\lg t - t^2\lg(1/\gamma) \to -\infty.
    \]
    But for $t \geq \frac{2}{\lg(1/\gamma)}\lg n$ we have $t^2\lg(1/\gamma) \geq  2t\lg (n)$, so
    \[
    2t \lg(n) + 2t\lg(e) - 2t\lg t - t^2\lg(1/\gamma) \leq 2t\lg(e) - 2t\lg t \to -\infty.\qedhere
    \]
\end{proof}

In the opposite direction, we use upper bounds on the extremal problem $\mathrm{ex}(n, K_{t, t})$, the maximum number of edges in a $K_{t,t}$-free graph on  $n$-vertices. While there have been improvements for many combinations of parameters since the original KST theorem (for a nice survey, see~\cite{füredi2013historydegeneratebipartiteextremal}), the following will be enough for our purposes:
\begin{restatable}[\cite{Kovari1954}, as stated by~{\cite[Thm.~2.22]{füredi2013historydegeneratebipartiteextremal}}]{theorem}{precisekst} \label{thm:precise-KST}
    The maximum number of edges in a $K_{t,t}$-free graph $G$ with $n$ vertices, denoted $\mathrm{ex}(n, K_{t, t})$, satisfies
    \[
    \mathrm{ex}(n, K_{t, t}) \leq \frac{1}{2}(t-1)^{1/t} n^{2-1/t} + \frac{1}{2}(t-1)n.
    \]
\end{restatable}

\begin{corollary}\label{cor:kst-guarantee}
    Every graph on $n$ vertices with edge density $\gamma < 1$ such that $\max\{\gamma^{-1}, (1-\gamma)^{-1}\} = n^{o(1)}$, contains a $K_{t, t}$ with $t \geq (1- o(1))\frac{\lg n}{\lg(1/\gamma)} $.
\end{corollary}
\begin{proof}
Let $t := \floor{\frac{\lg n - \lg \lg n + \lg \lg(1/\gamma)}{\lg(1/\gamma)+1}} $, and note that as $t= (1- o(1))\frac{\lg n}{\lg(1/\gamma)}$ as desired. Then, observe that
\[
t \leq \frac{\lg n - \lg \lg n + \lg \lg(1/\gamma)}{\lg(1/\gamma)+1} = \frac{\lg n - \lg(\frac{\lg n}{\lg(1/\gamma)})}{\lg(2/\gamma)} \leq \frac{\lg n - \lg t}{\lg(2/\gamma)},
\]
from where
\(
\lg(2/\gamma) \leq \frac{\lg n - \lg t}{t}
\), and thus 
\(
\lg(\gamma/2) \geq \frac{\lg t}{t} - \frac{\lg n}{t} = \lg(t^{1/t} n ^{-1/t}).
\)
By exponentiating, we obtain $
    \gamma/2 \geq t^{1/t} n^{-1/t} \geq (t-1)^{1/t}n^{-1/t}.
$
Moreover, since $\gamma^{-1} = n^{o(1)}$, we have for large $n$ that $\gamma/2 \geq t/n$. Adding this and the previous inequality,
and multiplying by $\frac{n^2}{2}$, we get
\[
\gamma \frac{n^2}{2} \geq \frac{1}{2}(t-1)^{1/t}n^{2-1/t} + \frac{tn}{2} \geq \frac{1}{2}(t-1)^{1/t}n^{2-1/t} + \frac{(t-1 + \gamma)n}{2},
\]
from where by~\Cref{thm:precise-KST},
$E(G) = \gamma \binom{n}{2} \geq \mathrm{ex}(n, K_{t,t})$.
\end{proof}

In terms of previous results, Mubayi and Tur\'an~\cite{mubayiFindingBipartiteSubgraphs2010} presented two algorithmic results: one that finds a balanced biclique of size $(1-o(1))\frac{\lg n}{\lg(4e/\gamma)}$ in time $O(n^{2.42})$, and the second finds a balanced biclique $5$ times smaller than the first one but in time $O(n^{1.42})$ assuming the degree of each vertex can be queried in $O(1)$, and $O(m)$ otherwise. We will present more efficient algorithms that also yield larger bicliques for almost all values of $\gamma$, but especially when $\gamma$ is not constant and tends to $1$, as illustrated in~\Cref{fig:bilcique_finding}. 

An important difference between our algorithms and those of Mubayi and Tur\'an is that their algorithms can never find a balanced biclique of size $\omega(\lg n)$ in polynomial time, even when $\gamma \to 1$, since their algorithm explores $\Omega(2^t)$ subsets for finding a $K_{t, t}$. In contrast, leveraging the density-aware biclique partitions we obtain from~\Cref{thm:gamma-2}, we can obtain bicliques of size e.g., $\Omega(\lg^2 n/\lg \lg n)$ when $\gamma \geq 1 - \frac{1}{\lg n}$, as we will show next.

First, we notice that while Chung, Erd\H{o}s, and Spencer leveraged the existence of large bicliques guaranteed by the KST theorem in order to find a biclique partition of small weight, the following observation shows that the implication can be reversed: finding any biclique partition of small weight implies finding a relatively large biclique.

\begin{lemma}[{\cite[Observation~1]{cardinal_et_al:LIPIcs.ESA.2025.67}}]\label{lemma:e-bp}
    Let $\mathcal{B}$ be a biclique cover of a graph $G$. Then, $\mathcal{B}$ contains a biclique $K_{a, b}$ with
    \(
        \min\{a, b\} \geq |E| / w(\mathcal{B}).
    \)
\end{lemma}
\begin{proof}
Let $\{(L_1, R_1), \ldots, (L_k, R_k)\}$ be a biclique cover for $G$ of weight $ w(\mathcal{B})$ and with $|L_i| \leq |R_i|$ for every $i$. Assume expecting a contradiction that $G$ does not contain a $K_{t,t}$ for $t \geq |E|/w(\mathcal{B})$, and thus $|L_i| < |E|/w(\mathcal{B})$ for each $i$. Then,
\[
|E| \leq \sum_{i=1}^k |L_i|\cdot |R_i| < \sum_{i=1}^k \frac{|E|}{w(\mathcal{B})} |R_i| \leq \sum_{i=1}^k \frac{|E|}{w(\mathcal{B})} (|L_i| + |R_i|) = |E|,
\]
a contradiction.
\end{proof}

Using the deterministic $O(m)$ algorithm from~\Cref{thm:gamma-2}, we obtain a biclique partition $\mathcal{B}$ of weight at most $(1 + o(1))\frac{h_2(\gamma) n^2}{2 \lg n}$, and then simply checking for the largest balanced biclique in $\mathcal{B}$, we obtain a $K_{t, t}$ which according to~\Cref{lemma:e-bp} satisfies
\[
t \geq \frac{\gamma \binom{n}{2}}{(1 + o(1))\frac{h_2(\gamma) n^2}{2 \lg n}} \geq (1- o(1))\frac{\gamma}{h_2(\gamma)} \lg n.
\]
This is a stronger quantitative guarantee than we get with Mubayi and Tur\'an's algorithm while still being efficient. In fact, building the entire biclique partition given by \Cref{thm:gamma-2} is unnecessary if we only want to find a single $K_{t,t}$ of the desired size, so we can actually find such a biclique in time $n^{1+o(1)}$ assuming the right computational model. Without randomness or $O(1)$ access to degrees, it is easy to show a lower bound of $\Omega(n^2)$ when $\gamma$ is fixed: if one queries $o(n^2)$ edges, it can be that none of them are present even if the graph has positive density. Therefore, we consider both randomized algorithms without $O(1)$ degree queries and deterministic algorithms with $O(1)$ degree queries. We assume an adjacency matrix representation for the input graph. We start with a couple lemmas.

\begin{lemma}\label{lemma:density-degrees-edges}
    Let $G = (V, E)$ be a graph on $n$ vertices and edge density $\gamma$ such that $\gamma^{-1} = n^{o(1)}$, and $\partSize = O(\lg^k n)$ for some $k$ be an integer. Then, if $D \subseteq V$ consists of the $\partSize$ vertices with highest degree, then the number of edges between $D$ and $V \setminus D$ is at least $(1-o(1))\partSize(n-\partSize)\gamma$ with the $o(1)$ term being $O\left(\frac{r}{\gamma n}\right)$.
\end{lemma}
\begin{proof}
Let us use notation $\avgdeg(S) := \frac{1}{|S|} \sum_{v \in S} \deg(v)$ for the average degree of a set of vertices. We trivially have $\avgdeg(D) \geq \avgdeg(V) = \frac{2|E|}{n} = \gamma(n-1).$
    Therefore, denoting by $e(A, B)$  the number of edges between sets $A$ and $B$, and $e(A)$ the number of edges with both endpoints in $A$, we have
    \[
    e(D, V\setminus D) = \left(\sum_{v \in D} \deg(v)\right) - 2 e(D) \geq \partSize \, \avgdeg(V) - \partSize^2 = \partSize(n-1)\gamma - \partSize ^2.
    \]
   Since
   \(
   \partSize(n-1)\gamma - \partSize ^2 \geq \partSize(n-\partSize)\gamma - \partSize^2 = \partSize(n-\partSize)\gamma \left(1 - \frac{\partSize}{\gamma(n-\partSize)}\right),
   \)
    and clearly $\frac{r}{\gamma(n-r)} = O(\frac{r}{\gamma n}) = o(1)$, this concludes the proof.
\end{proof}

\begin{lemma}
       Let $G = (V, E)$ be a graph on $n$ vertices with edge density $\gamma$ such that $\gamma^{-1} =n^{o(1)}$ , and $\partSize = O(\lg^k n)$ for some $k$ an integer. Then, there is a randomized algorithm, running in time $o(n)$ w.h.p. that finds a subset $D \subset V$ of size $\partSize$ such that with high probability there are $\gamma r(n-r)(1-O(\lg^{-1/3} n))$ edges between $D$ and $V \setminus D$.
    \label{lemma:density-alg}
\end{lemma}
\begin{proof}
    Let $\varepsilon := 1/\sqrt[3]{\lg n}$. The algorithm proceeds as follows:
        
  \begin{mdframed}
      \begin{enumerate}
        \item Initialize $D := \varnothing$.
        \item Sample a vertex $v \in V$ uniformly at random. 
        \item Estimate $\deg(v)$ by sampling a set $U \subset V \setminus \{v\}$ of size $|U| = \ceil{\sqrt{n}}$ uniformly at random, and taking the estimate $\widehat{\deg(v)} := \frac{n-1}{|U|}|\{ u \in U : \{u, v\} \in E(G) \}|.$
        
        \item If $\widehat{\deg(v)} \geq (1-\varepsilon)\gamma(n-1)$, add $v$ to $D$, and otherwise go to 2.
        \item If $|D| = r$ output $D$, otherwise go to 2.
    \end{enumerate}
  \end{mdframed}

     Let $W := \{ v \in V(G) : \deg(v) \geq (1-n^{-1/3}) \gamma(n-1)\}$. 
    Now we argue that $W$ is relatively large. Indeed,
    \[
    \gamma n (n-1) = \sum_{v \in V(G)} \deg(v) \leq (n-1)|W| +  (n-|W|)(1-n^{-1/3})\gamma(n-1),
    \]
    from where dividing by $(n-1)$ and rearranging we get
    \[
    |W| \geq n \cdot \frac{\gamma - (1-n^{-1/3})\gamma}{1 - (1-n^{-1/3})\gamma} = \frac{n^{2/3} \gamma}{1-(1-n^{-1/3})\gamma} \geq n^{2/3} \gamma.
    \]
    Therefore, every time the algorithm enters step 2, the sampled vertex $v$ will belong to $W \setminus D$ with probability at least $\prob[v \in W] - \prob[v \in D] \geq \gamma/\sqrt[3]{n} - \frac{r}{n} = \Omega(n^{-0.34})$, and thus the expected number of samples required to have sampled $r$ distinct vertices from $W$ is at most $O(rn^{0.34}) = O(n^{0.35})$.  From this, if the algorithm enters step 2 at least $n^{0.36}$ times, then with high probability it will have sampled $r$ distinct vertices from $W$.

    To prove that the algorithm finishes correctly within the desired time, it suffices to prove two things:
    \begin{enumerate}
        \item[] (\textbf{Completeness}) With high probability, every sampled vertex from $W$ will pass the test of step 4, and thus be added to $D$. 
        \item[] (\textbf{Soundness}) With high probability, no sampled vertex that passes the test of step 4 has degree less than $(1-\varepsilon)^2\gamma (n-1)$.
    \end{enumerate}
    Note that completeness implies the runtime to be  $O(n^{0.36} \cdot \sqrt{n}) = o(n)$ w.h.p., since steps 3.--5. take time $O(\sqrt{n})$, and step 2 is entered at most $n^{0.36}$ times with high probability until $|D| = r$. On the other hand, soundness implies correctness of the algorithm, since if every vertex in $D$ has degree at least $(1-\varepsilon)^2\gamma (n-1)$, then we have
     \begin{align*}
    e(D, V\setminus D) &= \left(\sum_{v \in D} \deg(v)\right) - 2 e(D) \geq r(1-\varepsilon)^2\gamma (n-1) - r^2\\
    &\geq (1-\varepsilon)^2\gamma r(n-r) - r^2 = (1-\varepsilon)^2\gamma r(n-r) \cdot \left(1 - \frac{r^2}{(1-\varepsilon)^2 \gamma r(n-r)} \right)\\
    &\geq (1 - \lg^{-1/3}(n))^2\left(1-\frac{2r}{\gamma n}\right)\gamma r(n-r) \geq (1 - \lg^{-1/3}(n))^3\gamma r(n-r)\\ &= (1 - O( \lg^{-1/3}(n)))\gamma r(n-r).
    \end{align*}
    
    Now, note that $\E[\widehat{\deg(v)}] = \deg(v)$, and observe that by Hoeffding's inequality we have
    \begin{align*}
    \prob[|\widehat{\deg(v)} - \deg(v)| \geq \varepsilon(1-\varepsilon)\gamma(n-1)] &= \prob\left[\left|\widehat{\deg(v)}\frac{|U|}{n-1} - \deg(v)\frac{|U|}{n-1}\right| \geq \varepsilon(1-\varepsilon)\gamma |U|\right]\\
    &\le 2 \exp\left(\frac{-2(\varepsilon(1-\varepsilon)\gamma |U|)^2}{|U|}\right) \le \exp(-\Omega(\sqrt[3]{n})),
    \end{align*}
    where the last inequality used that $\gamma^{-1} = n^{o(1)}$ and also $\varepsilon^{-1} = n^{o(1)}$.

    Therefore,
    \begin{align*}
    \prob[\widehat{\deg(v)} \geq (1-\varepsilon)\gamma (n-1) \mid \deg(v) < (1-\varepsilon)^2\gamma(n-1)]
    &\leq   \prob[|\widehat{\deg(v)} - \deg(v)| \geq \varepsilon(1-\varepsilon)\gamma(n-1)]\\
    &\leq \exp(-\Omega(\sqrt[3]{n})),
    \end{align*}
    and thus, after a union bound, with high probability every vertex $v$ added to $D$ in step 3 of the algorithm has degree at least $(1-\varepsilon)^2\gamma(n-1)$, which proves soundness.

    Similarly, for completeness, 
    \begin{align*}
     \prob[\widehat{\deg(v)} < (1-\varepsilon)\gamma (n-1) \mid v \in W] &=
    \prob[\widehat{\deg(v)} < (1-\varepsilon)\gamma (n-1) \mid \deg(v) \ge (1-n^{-1/3})\gamma(n-1)]\\
    &\leq   \prob[|\widehat{\deg(v)} - \deg(v)| \geq (\varepsilon - n^{-1/3})\gamma(n-1)]\\
    &\leq \exp(-\Omega(\sqrt[3]{n})). \qedhere
    \end{align*}
\end{proof}

\randalg*
\begin{proof}
    Let $\varepsilon = \varepsilon(n)$ be an error term with $\varepsilon = \Theta(1/\sqrt[3]{\lg n})$, and let $r := \floor{\frac{\lg n - 2\lg\lg n + \lg(h_2(\gamma)) + \lg(\gamma)}{h_2((1-\varepsilon)^2 \gamma)}}$.  Note that $r \sim \frac{\lg n}{h_2(
        \gamma
    )}$ since $h_2((1-\varepsilon)^2\gamma) \sim h_2(\gamma)$.
    We first find a subset $D \subseteq V(G)$ such that $|D| = r$ and the bipartite graph $G' := G[D, V \setminus D]$ has at least $(1-\varepsilon) \gamma |D| \cdot |V \setminus D|$ edges. In the deterministic model with $O(1)$ adjacency queries, we choose $D$ to be the $r$ vertices of highest degree in $G$. By sorting the vertices by degree, $D$ can be computed in $O(n \lg n)$. We show in~\Cref{lemma:density-degrees-edges} that this guarantees at least $(1-\varepsilon) \gamma |D| \cdot |V \setminus D|$ edges. In the randomized setting, \Cref{lemma:density-alg} gives a randomized $O(n)$ algorithm for constructing such a set $D$ with high probability using random sampling.

    Now, let $V^\star = \{v \in V \setminus D : d(v)/|D| \ge (1-\varepsilon)^2 \gamma\}$, where $d(v)$ denotes the degree of $v$ in $G'$. We can now lower bound the size of $V^\star$ by noting that
    \[
  (1-\varepsilon) \gamma |D| \cdot |V \setminus D|  \leq |E(G')| = \sum_{v \in V\setminus D} d(v) \leq |D| \cdot |V^\star| + |D|(1-\varepsilon)^2\gamma \cdot (|V\setminus D| - |V^\star|),
    \]
    from where dividing by $|D|$ and rearranging we obtain
    \begin{align*}
        |V^\star| &\geq |V \setminus D| \cdot \frac{(1-\varepsilon)\gamma - (1-\varepsilon)^2\gamma}{1 - (1-\varepsilon)^2\gamma} = |V \setminus D| \cdot \frac{(1-\varepsilon) \varepsilon\gamma}{1 - (1-\varepsilon)^2\gamma}\\
         &\ge |V \setminus D| \cdot \varepsilon \cdot (1-\varepsilon) \gamma \sim |V\setminus D| \cdot \varepsilon \cdot \gamma \geq \frac{n \gamma}{\lg n}.
    \end{align*}

    Next, for each $v \in V^\star$, we compute an arbitrary subset  $f(v) \subseteq N(v) \cap D$ of size $\floor{(1-\varepsilon)^2 \gamma |D|}$. By the pigeonhole principle, there is a subset $B \subseteq V$ such that $f(v) = f(w)$ for all $v,w \in B$ and
    \[
        |B| \ge \frac{|V^\star|}{\binom{|D|}{\floor{(1-\varepsilon)^2 \gamma |D|}}} \ge \frac{|V^\star|}{2^{h_2((1-\varepsilon)^2 \gamma) |D|}} \ge \frac{|V^\star| \lg^2(n)}{n \gamma h_2(\gamma)} \ge (1-o(1))\frac{\lg n}{h_2(\gamma)}.
    \]
    Let $A := f(v)$ for any $v \in B$. Then, $(A,B)$ is clearly a biclique, and it remains to observe that
    \[
        |A| = \floor{(1-\varepsilon)^2 \gamma |D|} = (1-o(1)) \frac{\gamma}{h_2(\gamma)} \lg n,
    \]
    from where $\min\{|A|, |B|\} \geq (1-o(1)) \frac{\gamma}{h_2(\gamma)} \lg n$.

    After $D$ has been chosen, the above construction can be performed in $O(\tfrac{n \lg n}{h_2(\gamma)})$ time deterministically as we show next. We initialize a list $V^\star$, and for every vertex $v \in  V\setminus D$, as long as $|V^\star| \leq \frac{n\gamma}{\lg n}$, we compute its subset $f(v)$ in time $O(|D|)$ by simply iterating over each vertex $u \in D$ and adding $u$ to $f(v)$ if $\{u, v\} \in E(G)$, stopping when $|f(v)| = \lfloor(1-\varepsilon)^2\gamma |D|\rfloor$, in which case we add the pair $(v, f(v))$ to $V^\star$; this computation takes time $O(n \cdot |D|) = O(\tfrac{n \lg n}{h_2(\gamma)})$.
    At this point, it suffices to identify the $f(v^\star)$ that appears most frequently in $V^\star$, and output $(f(v^\star), \{ v : (v, f(v^\star)) \in V^\star\})$. To identify such an $f(v^\star)$, we can simply sort all $f(v)$ lists lexicographically, in time $O(|V^\star|  \lg|V^\star| \cdot |f(v)|) = O(n\lg n\frac{\gamma^2}{h_2(\gamma)})$, where the $|f(v)|$ factor is the cost of each comparison between lists, and then do a linear scan to find the most frequent.
\end{proof}

\label{sec:large_bb}

\section{Graphs with bounded shattering}\label{sec:shattering}
This section shows how~\Cref{thm:ep} can be strengthened for graphs in which the neighborhoods $\{N(v) : v \in V(G)\}$ form a structurally simple set-system, regardless of the edge density.

Given a set-system $\mathcal{F} \subseteq \mathcal{P}(X)$, define its \emph{shatter function} $\pi_{\mathcal{F}}: \{1, \dotsc, |X|\} \rightarrow \mathbb{N}$ by $\pi_{\mathcal{F}}(z) = \max_{X' \subseteq X, |X'| = z} | \{ X' \cap F : F \in \mathcal{F} \} |$.
Given a graph $G = (V,E)$, we define its shatter function $\pi_G$ as the shatter function of the \emph{neighborhood set-family} given by $\{N(v) : v \in V(G)\}$.

We are interested in graphs whose shatter function is polynomial.
For instance, graphs whose neighborhood set-family has VC-dimension at most $d$ satisfy $\pi_{G}(z) \leq \sum_{i=0}^d \binom{z}{i} = O(z^d)$, as a consequence of the Perles--Sauer--Shelah lemma~\cite{SAUER1972145, Shelah1972}.

An interesting subfamily of graphs with polynomial shatter functions is that of \emph{$d$-dimensional semi-algebraic graphs with bounded complexity} (see, e.g.~\cite[Section 1.2.4]{cardinal_et_al:LIPIcs.ESA.2025.67} for precise definitions).
Those graphs admit polynomial shatter functions, as a consequence of the Milnor--Thom theorem~\cite[Corollary 2.3]{Fox-Pach-Sheffer-Suk-Zahl-2017}.
This latter family includes many intersection graphs of geometric objects, e.g. graphs of intersections of disks in $\mathbb{R}^2$ (see e.g.~\cite{Fox-Pach-Sheffer-Suk-Zahl-2017, cardinal_et_al:LIPIcs.ESA.2025.67} for other examples).
Do~\cite{Do-2019} (see also~\cite[Theorem 5]{cardinal_et_al:LIPIcs.ESA.2025.67}) proved that $d$-dimensional semi-algebraic graphs with bounded complexity admit, for every $\varepsilon > 0$, biclique covers of weight $O_{d, \varepsilon}(n^{2 - 2/(d+1) + \varepsilon})$.
Moreover, those biclique covers can be found efficiently~\cite{Agarwal-Aronov-Ezra-Katz-Sharir-2025}.
A further subclass of graphs is that of \emph{$d$-dimensional semi-linear graphs with bounded complexity}~\cite[Section 3]{cardinal_et_al:LIPIcs.ESA.2025.67} which includes, for example, interval graphs, and intersection graphs of axis-parallel boxes in $\mathbb{R}^d$.
Cardinal and Yuditsky~\cite[Theorem 10]{cardinal_et_al:LIPIcs.ESA.2025.67} showed that semi-linear graphs satisfy $\cover(G) = \tilde{O}(n)$.
We remark that all those families contain instances of dense graphs (with quadratic number of edges).
Also, being semi-algebraic or semi-linear is a much stronger restriction than having a polynomial shatter function (see the discussion before Theorem 1.6 in~\cite{Fox-Pach-Suk-2019}).

In the next result we consider graphs with polynomial shatter function, and we can even bound $\lbp(G)$ in such graphs.
The argument emulates that of \Cref{thm:ep}, and it can also be implemented in deterministic $O(n^2)$-time as in \Cref{alg:ep} (with the straightforward modifications).

\begin{theorem} \label{lemma:boundedshatter}
    Let $G$ be an $n$-vertex graph with shatter function $\pi_G(z) \leq c z^d$, for some $c, d \ge 1$.
    Then $\lbp(G) = O(n^{1 - 1/(d+1)})$.
\end{theorem}

\begin{proof}
    Let $\partSize = \lfloor n^{1/(d+1)} \rfloor$.
    Divide $V(G)$ into $t := \lceil n/\partSize \rceil$ sets $V_1, \dotsc, V_t$ of size as close as possible, and at most $\partSize$.
    Let $D$ be any tournament on $\{1, \dotsc, t\}$.
    The edges completely contained inside some $V_i$ form a biclique partition $\mathcal{B}_0$.
    For each $V_i$, consider the set $N_i = \bigcup_{j \in N_D(i)} V_j$.
    For each $S \subseteq V_i$, let $A_i(S) = \{ v \in N_i : N_G(v) \cap V_i = S\}$.
    There are most $\pi_G(\partSize)$ sets $S$ for which $A_i(S)$ is non-empty.
    Considering $(S, A_i(S))$ for those sets, we get a biclique partition $\mathcal{B}_i$ of the bipartite graph between $V_i$ and $N_i$.
    By construction the union over all $\mathcal{B}_i$ is a biclique partition of $G$.

    Now we estimate, for an arbitrary $1 \leq i \leq t$ and $v \in V_i$, how many bicliques contain $v$.
    Clearly $v$ is contained in at most $\partSize$ bicliques of $\mathcal{B}_0$.
    The bicliques from $\mathcal{B}_i$ contribute with at most $\pi_G(\partSize)$; and for $j \neq i$ the vertex $v$ can be in at most one biclique of $\mathcal{B}_j$.
    In total, we get $\partSize + \pi_G(\partSize) + t$.
    From the choice of $\partSize$, this is at most
    $O(n^{1 - 1/(d+1)})$, as desired.
\end{proof}

Observe that \Cref{lemma:boundedshatter} combined with \Cref{lemma:e-bp} immediately shows that every $n$-vertex $K_{t,t}$-free graph $G$ with $\pi_G(z) \leq c z^d$ can have at most $O(t n^{2 - 1/(d+1)})$ edges.
Fox, Pach, Sheffer, Suk, and Zahl~\cite[Theorem 2.1]{Fox-Pach-Sheffer-Suk-Zahl-2017} showed that such graphs can have in fact at most $O_t(n^{2 - 1/d})$ edges, recently improved by Janzer and Pohoata~\cite{Janzer-Pohoata-2024} to $o_t(n^{2 - 1/d})$ edges if $k \geq d > 2$.

\section{Dense random graphs}\label{sec:random_graphs}
Chung, Erd\H{o}s and Spencer~\cite{chungDecompositionGraphsComplete1983} also asked for the average of $\cover(G) $ and $\totalWeight(G)$ over all $n$-vertex graphs.
Phrased in probabilistic terms, they asked for the expected value of $\cover(G)$ and $\totalWeight(G)$ when $G$ is drawn according to $G(n,p)$ with $p=1/2$ (recall that $G(n,p)$ is the random binomial graph on $n$ vertices where each edge is included independently with probability $p$).
We answer their question in a more general form by considering all $p \in (0,1)$.
We also consider the uniform random graph model $G(n, m)$, which is the model of random graphs chosen uniformly from $\mathcal{G}(n,m)$, the set of $n$-vertex graphs with $m$ edges.
In fact, we can also determine the `typical' value of $\cover(G)$ and $\totalWeight(G)$ for a dense random graph; from which the results on the expected value follow easily.

At a high level, we can get upper bounds in the $G(n, m)$ model immediately since there the edge density is deterministic, and thus~\Cref{thm:gamma-2} applies directly.
The lower bounds are based on noting that the information-theoretic lower bounds also apply on average.

\begin{theorem}\label{thm:dense-random-graphs}
    Let $p \in (0,1)$ be fixed.
    Given $n$, let $m = m(n) =  \lfloor p \binom{n}{2} \rfloor$.
    Let $G$ be a random graph chosen either according to $G(n,m)$ or $G(n,p)$. Then, with probability tending to $1$ as $n$ goes to infinity, we have $\cover(G) \sim \totalWeight(G) \sim \frac{h_2(p)}{2} \cdot \frac{n^2}{\lg n}$. 
\end{theorem}

\begin{proof}
    It suffices to prove the result for $G(n,m)$, as then the result is transferred automatically to the $G(n,p)$ model by the standard couplings between the two models, see e.g.~\cite[Theorem 1.4]{Frieze-Karonski-2023}.
    Let $G$ be drawn uniformly from $\mathcal{G}(n,m)$.
    Note that the bound $\totalWeight(G) \leq (h_2(p)/2 + o(1)) n^2/\lg n$ follows deterministically from \Cref{thm:gamma-2}.
    
    We need to prove that, with probability $1 - o(1)$, we also have $\cover(G) \geq (h_2(p)/2 - o(1)) n^2/\lg n$.
    We know that each biclique cover of a graph $G$ with weight $\cover(G)$ gives an encoding $f(G) \in \{0,1\}^\star$ of $G$ using at most $\cover(G)(\lg n + O(1))$ bits.
    Since the encoding $f: \mathcal{G}(n,m) \rightarrow \{0,1\}^\star$ is injective, the set $\Gamma \subseteq \{0,1\}^\star$ of strings used by the encoding satisfies $|\Gamma| \geq |\mathcal{G}(n,m)|$.
    The number of strings in $\{0,1\}^\star$ of length at most $k$ is $2^{k+1} - 1$.
    Hence, if $k$ is such that $k = \lg |\Gamma| - \omega(1)$, then with high probability we have $\cover(G) \geq k / (\lg n + O(1))$.
    By the argument outlined at the beginning of \Cref{sec:nechiporuk} we have that $\lg |\Gamma| = (h_2(\gamma) + o(1))\binom{n}{2}$, so we can take $k = (h_2(\gamma) - o(1))\binom{n}{2}$ as well.
    This gives that with high probability we have $\cover(G) \geq (h_2(\gamma)/2 - o(1)) n^2/\lg n$, as desired.
\end{proof}

\begin{corollary}
    Let $G$ be a random graph as in \Cref{thm:dense-random-graphs}.
    Then 
    $\E[\cover(G)] \sim \E[\totalWeight(G)] \sim \frac{h_2(p)}{2} \cdot \frac{n^2}{\lg n}$.
\end{corollary}

\begin{proof}
    We consider $G \sim \mathcal{G}(n,m)$ in the uniform model first.
    Recall that $\totalWeight(G) \leq (h_2(p)/2 + o(1)) n^2/\lg n$ holds deterministically, so in particular $\E[\totalWeight(G)] \leq (h_2(p)/2 + o(1)) n^2/\lg n$ as well.
    Let $k = (h_2(p)/2 - o(1)) n^2/\lg n$. Using \Cref{thm:dense-random-graphs}, we have \[  \E[\cover(G)] \geq k \prob[\cover(G) \geq k] = k (1 - o(1)), \]
    which gives the desired lower bound.
    The proofs for $G$ drawn from the binomial model $G(n,p)$ follow along the same lines, by noting that $|E(G)| = p \binom{n}{2} + o(n^2)$ with high probability (by, e.g. Chernoff inequalities).
\end{proof}

\section{Related work and applications}

Biclique covers and partitions arise naturally in many areas of computer science, including the aforementioned applications to cryptography and computational geometry. We outline several more examples next.

Arguably, one of the most impactful applications of biclique partitions is to graph compression, kick-started by the seminal work of Feder and Motwani~\cite{feder-motwani}, which has been followed by empirically validated applications and theoretical improvements~\cite{chavan2025,Hernndez2013, Francisco2022}. 

A different line of work in which biclique coverings and partitions appear often is in circuit complexity (e.g., the \emph{star complexity} of a graph), where these decompositions allow for constructing smaller circuits for graph-related problems~\cite{juknaGraphComplexity,Jukna2012}. Biclique decompositions have been used to prove an upper bound on the size of monotone Boolean formulas for quadratic functions~\cite{Bublitz1986} and to a minimization problem for Horn formulas~\cite{horn}. Similarly, biclique coverings allow for reducing the size of CNF formulas in both graph and scheduling problems~\cite{subercaseaux2025asymptoticallysmallerencodingsgraph}. 

Partitioning a graph $G$ into bicliques $B_1 \sqcup B_2 \sqcup \ldots \sqcup B_k$ can also be understood from the adjacency matrix perspective, since then the adjacency matrix of $G$ can be written as a sum of the adjacency matrices of the graphs $B_i$, appropriately padded with $0$s. Due to the structure of adjacency matrices of bicliques, several linear-algebraic operations can be performed more efficiently over them, which has further motivated the study of biclique decompositions~\cite{Tuza1984, Jukna2013, Francisco2022}.

The study of biclique decompositions has also arisen naturally in other areas of computer science  such as automata theory, where it allowed Iv\'an et al. to prove a separation result for the minimal number of states of nondeterministic finite automata (NFAs) using $\varepsilon$-transitions versus $\varepsilon$-free NFAs~\cite{ivan2014bicliquecoveringsrectifiernetworks}.

The Erd\H{o}s--Pyber theorem has been used to give a bound on the \emph{local dimension} of posets \cite{posets-1,posets-2}, and our results improve the constant factor in that bound (see \cite[Theorem~2]{posets-1}).


More generally, our work is part of a much broader field of research in discrete mathematics: \emph{graph decompositions}. See \cite{graph-covering-survey} for a recent survey on graph decomposition and \cite{clique-biclique-survey} for a survey about clique and biclique decompositions specifically.

\section{Further directions}\label{sec:conclusion}
We have made significant progress in understanding the landscape of partite decompositions for graphs and uniform hypergraphs, showing tight results for both the total weight (Chung--Erd\H{o}s--Spencer style) and for the maximum number of $d$-cliques  each vertex belongs to (Erd\H{o}s--Pyber style). Nonetheless, several avenues for future work remain open:
\begin{itemize}
    \item We have shown how considering the edge density of graphs allows for better bounds when the density is bounded away from $1/2$ (\Cref{thm:gamma-2}). An interesting direction of future work is to study analogous density-aware results for $d$-uniform hypergraphs, with edge density $\gamma_d := |E(G)|/\binom{n}{d}$.
    
    \item We have shown in~\Cref{sec:large_bb} how our techniques allow for efficiently finding large balanced bicliques in graphs. An interesting line of research is whether our results for hypergraphs can help find large balanced $d$-cliques in dense $d$-uniform hypergraphs. Concretely, $d$-cliques with each part having size $\Omega( (\lg n)^{1/(d-1)})$ are ensured by \cite{Erdos-1964} (the hypergraph version of KST), and the best algorithmic results are given by a recent result of Espuña \cite{Espuna-2025}. In contrast, our hypergraph decomposition results (\Cref{theorem:ces-hypergraph,theorem:ep-hypergraph-upper}) yield very unbalanced $d$-cliques, in which all but $2$ parts have size $1$. Thus, an interesting question is whether similar decomposition results can be obtained with more balanced $d$-cliques.
    \item From the `representation' perspective presented in~\Cref{sec:representations}, an interesting question is whether it is possible to efficiently support dynamic graphs, handling both edge insertions and deletions, as well as vertex insertions and deletions. The lazy deletion operation used for the densest subgraph approximation might be helpful as a starting point. In general terms, a more complete analysis of the complexity of common graph queries and operations under these biclique representations is a natural direction of future research.
    \item Csirmaz, Ligeti, and Tardos~\cite{ep-hypergraph} considered \emph{fractional} coverings/partitions, in which each $d$-clique $D_1, \ldots, D_k$ of the input hypergraph $H$ must be assigned some non-negative weight $w(D_i)$ and then for each hyperedge $e$ one must have 
    \(
    \sum_{D_i \text{ contains } e} w(D_i) = 1
    \)
    in the case of partitions, and 
    \(
        \sum_{D_i \text{ contains } e} w(D_i) \geq 1
    \)
    in the case of coverings. 
    In this setting, the load of a vertex is the sum of the weights of the $d$-cliques it belongs to, and one defines
    \(
    \lmp^\ast_d(H)
    \) (resp.~$\lmc^\ast_d(H)$)  as the minimum over the fractional partitions (resp covers) of the maximum load over the vertices of $H$. Then, naturally, \(
    \lmp^\ast_d(n)
    \) is defined as the maximum \(
    \lmp^\ast_d(H)
    \) over all $n$-vertex hypergraphs $H$, and $\lmc^\ast_d(H)$ is defined analogously. Through the probabilistic method, Csirmaz, Ligeti, and Tardos proved the lower bound
    \[
    \lmp^\ast_d(n) \geq   \left(\frac{0.53}{d!} - o_d(1)\right) \frac{n^{d-1}}{\lg n}.
    \]
    It would be interesting to know the \emph{integrality ratio} 
    \[
    \lim_{n \to \infty} \frac{\lmp_d(n)}{\lmp^\ast_d(n)}
    \]
    if it exists. Our~\Cref{theorem:ep-hypergraph-upper} combined with the lower bound of~\cite{ep-hypergraph} shows  \[ 1 \leq \lim_{n \to \infty} \frac{\lmp_d(n)}{\lmp^\ast_d(n)} \leq \frac{1}{0.53} < 1.89.\]

    It is worth noting the following theorem of Csirmaz, Ligeti, and Tardos:
    \begin{theorem}[{\cite[Thm. 4]{ep-hypergraph}}]
        Let $p \in (0, 1)$ and $G$ be an $n$-vertex graph in which every vertex $v$ satisfies $\frac{\deg(v)}{n} \geq p$. Then,
        \[
        \lmc^\ast_2(n) \leq \left(\frac{(1-p)}{2 \ln 2} + o(1)\right)\frac{n}{\lg n}.
        \]
    \end{theorem}
    Together with our~\Cref{thm:dense-random-graphs}, this directly shows a lower bound on the integrality ratio for random graphs; \Cref{thm:dense-random-graphs} implies that with high probability, a random graph $G \sim G(n, p)$ satisfies $\lmc_2(G) \geq (\frac{h_2(p)}2 - o(1))\frac{n}{\lg n}$, while simultaneously satisfying $\frac{\deg(v)}{n} \geq p - o(1)$ for every vertex with high probability.

\end{itemize}

\section*{Acknowledgments}

Krapivin was supported by the Jeanne B. and Richard F. Berdik ARCS Pittsburgh Endowed Scholar Award. Przybocki was supported by the NSF Graduate Research Fellowship Program under Grant No. DGE-2140739. Sanhueza-Matamala was supported by ANID-FONDECYT Regular Nº1251121 grant. Subercaseaux was supported by NSF grant DMS-2434625.

\bibliographystyle{alpha}
\bibliography{references}

@incollection{chungDecompositionGraphsComplete1983,
  title = {On the Decomposition of Graphs into Complete Bipartite Subgraphs},
  booktitle = {Studies in {{Pure Mathematics}}: {{To}} the {{Memory}} of {{Paul Tur{\'a}n}}},
  author = {Chung, Fan and Erd{\H o}s, Paul and Spencer, Joel},
  editor = {Erd{\H o}s, Paul and Alp{\'a}r, L{\'a}szl{\'o} and Hal{\'a}sz, G{\'a}bor and S{\'a}rk{\"o}zy, Andr{\'a}s},
  year = {1983},
  pages = {95--101},
  publisher = {Birkh{\"a}user},
  address = {Basel},
  doi = {10.1007/978-3-0348-5438-2_10},
  urldate = {2025-04-29},
  isbn = {978-3-0348-5438-2},
  langid = {english}
}

@article{mubayiFindingBipartiteSubgraphs2010,
  title = {Finding Bipartite Subgraphs Efficiently},
  author = {Mubayi, Dhruv and Tur{\'a}n, Gy{\"o}rgy},
  year = {2010},
  month = feb,
  journal = {Information Processing Letters},
  volume = {110},
  number = {5},
  pages = {174--177},
  issn = {0020-0190},
  doi = {10.1016/j.ipl.2009.11.015},
  urldate = {2025-04-29}
}

@article{Kovari1954,
author = {K\H{o}vári, Thomas and Sós, Vera and Turán, Pál},
journal = {Colloquium Mathematicae},
language = {eng},
number = {1},
pages = {50-57},
title = {On a problem of {K}. {Z}arankiewicz},
url = {http://eudml.org/doc/210011},
volume = {3},
year = {1954},
}

@article {Tuza1984,
    AUTHOR = {Tuza, Zsolt},
     TITLE = {Covering of graphs by complete bipartite subgraphs: complexity
              of {$0$}-{$1$}\ matrices},
   JOURNAL = {Combinatorica},
  FJOURNAL = {Combinatorica. An International Journal of the J\'anos Bolyai
              Mathematical Society},
    VOLUME = {4},
      YEAR = {1984},
    NUMBER = {1},
     PAGES = {111--116},
      ISSN = {0209-9683},
   MRCLASS = {05C35 (05C50)},
  MRNUMBER = {739419},
MRREVIEWER = {Jean-Claude\ Bermond},
       DOI = {10.1007/BF02579163},
       URL = {https://doi.org/10.1007/BF02579163},
}

@article {Bublitz1986,
    AUTHOR = {Bublitz, Siegfried},
     TITLE = {Decomposition of graphs and monotone formula size of
              homogeneous functions},
   JOURNAL = {Acta Inform.},
  FJOURNAL = {Acta Informatica},
    VOLUME = {23},
      YEAR = {1986},
    NUMBER = {6},
     PAGES = {689--696},
      ISSN = {0001-5903,1432-0525},
   MRCLASS = {68Q25 (05A15 06E30 68R10)},
  MRNUMBER = {865502},
       DOI = {10.1007/BF00264314},
       URL = {https://doi.org/10.1007/BF00264314},
}

@article {feder-motwani,
    AUTHOR = {Feder, Tom\'as and Motwani, Rajeev},
     TITLE = {Clique partitions, graph compression and speeding-up
              algorithms},
   JOURNAL = {J. Comput. System Sci.},
  FJOURNAL = {Journal of Computer and System Sciences},
    VOLUME = {51},
      YEAR = {1995},
    NUMBER = {2},
     PAGES = {261--272},
      ISSN = {0022-0000,1090-2724},
       DOI = {10.1006/jcss.1995.1065},
       URL = {https://doi.org/10.1006/jcss.1995.1065},
}

@article{lupanov,
    author = {Lupanov, Oleg},
    title = {On Rectifier and Switching-and-Rectifier Circuits},
    journal = {Doklady Academii nauk SSSR},
    volume={111},
    year = {1956},
    note={Available at \url{https://web.vu.lt/mif/s.jukna/boolean/lupanov56.pdf}, thanks to Stasys Jukna.}
}

@incollection{juknaGraphComplexity,
  title = {Computational Complexity of Graphs},
  booktitle = {Advances in Network Complexity},
  author = {Jukna, Stasys},
  year = {2013},
  url = {https://onlinelibrary.wiley.com/doi/pdf/10.1002/9783527670468.ch05},
  pages = {99--153},
  publisher = {John Wiley \& Sons, Ltd},
  doi = {10.1002/9783527670468.ch05},
  chapter = {5},
  isbn = {978-3-527-67046-8}
}

@misc{chavan2025,
      title={A Clique Partitioning-Based Algorithm for Graph Compression}, 
      author={Akshar Chavan and Sanaz Rabinia and Daniel Grosu and Marco Brocanelli},
      year={2025},
      eprint={2502.02477},
      archivePrefix={arXiv},
      primaryClass={cs.DS},
      url={https://arxiv.org/abs/2502.02477},
    note={Arxiv preprint, \url{https://arxiv.org/abs/2502.02477}}
}

@book{lothaire-alg,
    AUTHOR = {Lothaire, M.},
     TITLE = {Algebraic combinatorics on words},
    SERIES = {Encyclopedia of Mathematics and its Applications},
    VOLUME = {90},
      NOTE = {A collective work by Jean Berstel, Dominique Perrin, Patrice
              Seebold, Julien Cassaigne, Aldo De Luca, Steffano Varricchio,
              Alain Lascoux, Bernard Leclerc, Jean-Yves Thibon, Veronique
              Bruyere, Christiane Frougny, Filippo Mignosi, Antonio Restivo,
              Christophe Reutenauer, Dominique Foata, Guo-Niu Han, Jacques
              Desarmenien, Volker Diekert, Tero Harju, Juhani Karhumaki and
              Wojciech Plandowski,
              With a preface by Berstel and Perrin},
 PUBLISHER = {Cambridge University Press, Cambridge},
      YEAR = {2002},
     PAGES = {xiv+504},
      ISBN = {0-521-81220-8},
       DOI = {10.1017/CBO9781107326019},
       URL = {https://doi.org/10.1017/CBO9781107326019},
}

@article {erdos-pyber,
    AUTHOR = {Erd\H{o}s, Paul and Pyber, L\'aszl\'o},
     TITLE = {Covering a graph by complete bipartite graphs},
   JOURNAL = {Discrete Math.},
  FJOURNAL = {Discrete Mathematics},
    VOLUME = {170},
      YEAR = {1997},
    NUMBER = {1-3},
     PAGES = {249--251},
      ISSN = {0012-365X,1872-681X},
       DOI = {10.1016/S0012-365X(96)00124-0},
       URL = {https://doi.org/10.1016/S0012-365X(96)00124-0},
}

@article {ep-hypergraph,
    AUTHOR = {Csirmaz, L\'aszl\'o{} and Ligeti, P\'eter and Tardos, G\'abor},
     TITLE = {{Erd\H{o}s-{P}yber theorem for hypergraphs and secret sharing}},
   JOURNAL = {Graphs Combin.},
  FJOURNAL = {Graphs and Combinatorics},
    VOLUME = {31},
      YEAR = {2015},
    NUMBER = {5},
     PAGES = {1335--1346},
      ISSN = {0911-0119,1435-5914},
       DOI = {10.1007/s00373-014-1448-7},
       URL = {https://doi.org/10.1007/s00373-014-1448-7},
}

@book{Jukna2012,
  title = {Boolean Function Complexity: Advances and Frontiers},
  ISBN = {9783642245084},
  ISSN = {0937-5511},
  url = {http://dx.doi.org/10.1007/978-3-642-24508-4},
  DOI = {10.1007/978-3-642-24508-4},
  journal = {Algorithms and Combinatorics},
  publisher = {Springer Berlin Heidelberg},
  author = {Jukna,  Stasys},
  year = {2012}
}

@misc{Espuna-2025,
      title={Finding Partite Hypergraphs Efficiently}, 
      author={Ferran Espuña},
      year={2025},
      eprint={2508.10641},
      archivePrefix={arXiv},
      primaryClass={math.CO},
      url={https://arxiv.org/abs/2508.10641},
      note={Arxiv preprint, \url{https://arxiv.org/abs/2508.10641}}
}

@article {Erdos-1964,
    AUTHOR = {{Erd\H{o}s}, Paul},
     TITLE = {On extremal problems of graphs and generalized graphs},
   JOURNAL = {Israel J. Math.},
  FJOURNAL = {Israel Journal of Mathematics},
    VOLUME = {2},
      YEAR = {1964},
     PAGES = {183--190},
      ISSN = {0021-2172},
   MRCLASS = {05.40},
  MRNUMBER = {183654},
MRREVIEWER = {A.\ H.\ Stone},
       DOI = {10.1007/BF02759942},
       URL = {https://doi.org/10.1007/BF02759942},
}

@article{Beimel2014,
  title = {Secret-Sharing Schemes for Very Dense Graphs},
  volume = {29},
  ISSN = {1432-1378},
  url = {http://dx.doi.org/10.1007/s00145-014-9195-8},
  DOI = {10.1007/s00145-014-9195-8},
  number = {2},
  journal = {Journal of Cryptology},
  publisher = {Springer Science and Business Media LLC},
  author = {Beimel,  Amos and Farràs,  Oriol and Mintz,  Yuval},
  year = {2014},
  month = dec,
  pages = {336–362}
}

@article{blundoInformationRateSecret1996,
  title = {On the Information Rate of Secret Sharing Schemes},
  author = {Blundo, Carlo and De Santis, Alfredo and Gargano, Luisa and Vaccaro, Ugo},
  year = {1996},
  month = feb,
  journal = {Theoretical Computer Science},
  volume = {154},
  number = {2},
  pages = {283--306},
  issn = {0304-3975},
  doi = {10.1016/0304-3975(95)00065-8}
}

@article {graph-covering-survey,
    AUTHOR = {Schwartz, Stephan},
     TITLE = {An overview of graph covering and partitioning},
   JOURNAL = {Discrete Math.},
  FJOURNAL = {Discrete Mathematics},
    VOLUME = {345},
      YEAR = {2022},
    NUMBER = {8},
     PAGES = {Paper No. 112884, 17},
      ISSN = {0012-365X,1872-681X},
       DOI = {10.1016/j.disc.2022.112884},
       URL = {https://doi.org/10.1016/j.disc.2022.112884},
}

@article {clique-biclique-survey,
    AUTHOR = {Monson, Sylvia D. and Pullman, Norman J. and Rees, Rolf},
     TITLE = {A survey of clique and biclique coverings and factorizations
              of {$(0,1)$}-matrices},
   JOURNAL = {Bull. Inst. Combin. Appl.},
  FJOURNAL = {Bulletin of the Institute of Combinatorics and its
              Applications},
    VOLUME = {14},
      YEAR = {1995},
     PAGES = {17--86},
      ISSN = {1183-1278,2689-0674},
   MRCLASS = {05C70 (05C35 15A24)},
  MRNUMBER = {1330781},
MRREVIEWER = {W.\ D.\ Wallis},
}

@misc{subercaseaux2025asymptoticallysmallerencodingsgraph,
      title={Asymptotically Smaller Encodings for Graph Problems and Scheduling}, 
      author={Bernardo Subercaseaux},
      year={2025},
      eprint={2506.14042},
      archivePrefix={arXiv},
      primaryClass={cs.LO},
      url={https://arxiv.org/abs/2506.14042}, 
      note={(\emph{ModRef workshop}), Arxiv preprint, \url{https://arxiv.org/abs/2506.14042}}
}

@article{FARZAN201338,
  title = {Succinct Encoding of Arbitrary Graphs},
  author = {Farzan, Arash and Munro, J. Ian},
  year = 2013,
  journal = {Theoretical Computer Science},
  volume = {513},
  pages = {38--52},
  issn = {0304-3975},
  doi = {10.1016/j.tcs.2013.09.031}
}

@article{Jukna2013,
  title = {Complexity of Linear Boolean Operators},
  volume = {9},
  ISSN = {1551-3068},
  url = {http://dx.doi.org/10.1561/0400000063},
  DOI = {10.1561/0400000063},
  number = {1},
  journal = {Foundations and Trends{\textregistered} in Theoretical Computer Science},
  publisher = {Emerald},
  author = {Jukna,  Stasys},
  year = {2013},
  pages = {1–123}
}

@inproceedings{horn,
  author       = {Amitava Bhattacharya and
                  Bhaskar DasGupta and
                  Dhruv Mubayi and
                  Gy{\"{o}}rgy Tur{\'{a}}n},
  editor       = {Samson Abramsky and
                  Cyril Gavoille and
                  Claude Kirchner and
                  Friedhelm Meyer auf der Heide and
                  Paul G. Spirakis},
  title        = {On Approximate Horn Formula Minimization},
  booktitle    = {Automata, Languages and Programming, 37th International Colloquium,
                  {ICALP} 2010, Bordeaux, France, July 6-10, 2010, Proceedings, Part
                  {I}},
  series       = {Lecture Notes in Computer Science},
  volume       = {6198},
  pages        = {438--450},
  publisher    = {Springer},
  year         = {2010},
  url          = {https://doi.org/10.1007/978-3-642-14165-2\_38},
  doi          = {10.1007/978-3-642-14165-2\_38},
  bibsource    = {dblp computer science bibliography, https://dblp.org}
}

@inproceedings{ivan2014bicliquecoveringsrectifiernetworks,
  author       = {Szabolcs Iv{\'{a}}n and
                  {\'{A}}d{\'{a}}m D{\'{a}}niel Lelkes and
                  Judit Nagy{-}Gy{\"{o}}rgy and
                  Bal{\'{a}}zs Sz{\"{o}}r{\'{e}}nyi and
                  Gy{\"{o}}rgy Tur{\'{a}}n},
  editor       = {Helmut J{\"{u}}rgensen and
                  Juhani Karhum{\"{a}}ki and
                  Alexander Okhotin},
  title        = {Biclique Coverings, Rectifier Networks and the Cost of {\(\varepsilon\)}-Removal},
  booktitle    = {Descriptional Complexity of Formal Systems - 16th International Workshop,
                  {DCFS} 2014, Turku, Finland, August 5-8, 2014. Proceedings},
  series       = {Lecture Notes in Computer Science},
  volume       = {8614},
  pages        = {174--185},
  publisher    = {Springer},
  year         = {2014},
  url          = {https://doi.org/10.1007/978-3-319-09704-6\_16},
  doi          = {10.1007/978-3-319-09704-6\_16},
  bibsource    = {dblp computer science bibliography, https://dblp.org}
}

@article{Francisco2022,
  title = {Graph Compression for Adjacency-Matrix Multiplication},
  volume = {3},
  ISSN = {2661-8907},
  url = {http://dx.doi.org/10.1007/s42979-022-01084-2},
  DOI = {10.1007/s42979-022-01084-2},
  number = {3},
  journal = {SN Computer Science},
  publisher = {Springer Science and Business Media LLC},
  author = {Francisco,  Alexandre P. and Gagie,  Travis and K\"{o}ppl,  Dominik and Ladra,  Susana and Navarro,  Gonzalo},
  year = {2022},
  month = mar 
}

@article{Hernndez2013,
  title = {Compressed representations for web and social graphs},
  volume = {40},
  ISSN = {0219-3116},
  url = {http://dx.doi.org/10.1007/s10115-013-0648-4},
  DOI = {10.1007/s10115-013-0648-4},
  number = {2},
  journal = {Knowledge and Information Systems},
  publisher = {Springer Science and Business Media LLC},
  author = {Hernández,  Cecilia and Navarro,  Gonzalo},
  year = {2013},
  month = apr,
  pages = {279–313}
}

@incollection {füredi2013historydegeneratebipartiteextremal,
    AUTHOR = {F\"uredi, Zolt\'an and Simonovits, Mikl\'os},
     TITLE = {The history of degenerate (bipartite) extremal graph problems},
 BOOKTITLE = {Erd\"os centennial},
    SERIES = {Bolyai Soc. Math. Stud.},
    VOLUME = {25},
     PAGES = {169--264},
 PUBLISHER = {J\'anos Bolyai Math. Soc., Budapest},
      YEAR = {2013},
      ISBN = {978-963-9453-18-0; 978-3-642-39285-6},
   MRCLASS = {05-02 (01A70)},
  MRNUMBER = {3203598},
MRREVIEWER = {Felix\ Lazebnik},
       DOI = {10.1007/978-3-642-39286-3\_7},
       URL = {https://doi.org/10.1007/978-3-642-39286-3_7},
}

@InProceedings{beimel:LIPIcs.ITC.2023.16,
  author =	{Beimel, Amos},
  title =	{{Lower Bounds for Secret-Sharing Schemes for k-Hypergraphs}},
  booktitle =	{4th Conference on Information-Theoretic Cryptography (ITC 2023)},
  pages =	{16:1--16:13},
  series =	{Leibniz International Proceedings in Informatics (LIPIcs)},
  ISBN =	{978-3-95977-271-6},
  ISSN =	{1868-8969},
  year =	{2023},
  volume =	{267},
  editor =	{Chung, Kai-Min},
  publisher =	{Schloss Dagstuhl -- Leibniz-Zentrum f{\"u}r Informatik},
  address =	{Dagstuhl, Germany},
  URL =		{https://drops.dagstuhl.de/entities/document/10.4230/LIPIcs.ITC.2023.16},
  URN =		{urn:nbn:de:0030-drops-183440},
  doi =		{10.4230/LIPIcs.ITC.2023.16}
}

@article {posets-1,
    AUTHOR = {Kim, Jinha and Martin, Ryan R. and Masa\v{r}\'ik, Tom\'a\v{s}
              and Shull, Warren and Smith, Heather C. and Uzzell, Andrew and
              Wang, Zhiyu},
     TITLE = {On difference graphs and the local dimension of posets},
   JOURNAL = {European J. Combin.},
  FJOURNAL = {European Journal of Combinatorics},
    VOLUME = {86},
      YEAR = {2020},
     PAGES = {103074, 13},
      ISSN = {0195-6698,1095-9971},
       DOI = {10.1016/j.ejc.2019.103074},
       URL = {https://doi.org/10.1016/j.ejc.2019.103074},
}

@article {posets-2,
    AUTHOR = {Dam\'asdi, G\'abor and Felsner, Stefan and Gir\~ao, Ant\'onio
              and Keszegh, Bal\'azs and Lewis, David and Nagy, D\'aniel T.
              and Ueckerdt, Torsten},
     TITLE = {On covering numbers, {Y}oung diagrams, and the local dimension
              of posets},
   JOURNAL = {SIAM J. Discrete Math.},
  FJOURNAL = {SIAM Journal on Discrete Mathematics},
    VOLUME = {35},
      YEAR = {2021},
    NUMBER = {2},
     PAGES = {915--927},
      ISSN = {0895-4801,1095-7146},
       DOI = {10.1137/20M1313684},
       URL = {https://doi.org/10.1137/20M1313684},
}

@article {nechiporuk,
    AUTHOR = {Nechiporuk, Éduard Ivanovich},
     TITLE = {The topological principles of self-correction},
   JOURNAL = {Problemy Kibernet.},
  FJOURNAL = {Problemy Kibernetiki},
    NUMBER = {21},
      YEAR = {1969},
     PAGES = {5--102},
   MRCLASS = {02C05},
  MRNUMBER = {294096},
MRREVIEWER = {J.\ Kuntzmann},
note={Available in Russian at \url{https://web.vu.lt/mif/s.jukna/boolean/Russians/Nechiporuk-1969a.pdf}, thanks to Stasys Jukna.}
}

@book{Navarro2016,
  title = {Compact Data Structures: A Practical Approach},
  ISBN = {9781316588284},
  url = {http://dx.doi.org/10.1017/CBO9781316588284},
  DOI = {10.1017/cbo9781316588284},
  publisher = {Cambridge University Press},
  author = {Navarro,  Gonzalo},
  year = {2016},
  month = sep 
}

@inproceedings{rankSelectSequences,
author = {Golynski, Alexander and Munro, J. Ian and Rao, S. Srinivasa},
title = {Rank/select operations on large alphabets: a tool for text indexing},
year = {2006},
isbn = {0898716055},
publisher = {Society for Industrial and Applied Mathematics},
address = {USA},
booktitle = {Proceedings of the Seventeenth Annual ACM-SIAM Symposium on Discrete Algorithm},
pages = {368–373},
numpages = {6},
location = {Miami, Florida},
series = {SODA '06}
}

@inproceedings{RegLemmaAlgorithms,
author = {Bansal, Nikhil and Williams, Ryan},
title = {Regularity Lemmas and Combinatorial Algorithms},
year = {2009},
isbn = {9780769538501},
publisher = {IEEE Computer Society},
address = {USA},
url = {https://doi.org/10.1109/FOCS.2009.76},
doi = {10.1109/FOCS.2009.76},
booktitle = {Proceedings of the 2009 50th Annual IEEE Symposium on Foundations of Computer Science},
pages = {745–754},
numpages = {10},
series = {FOCS '09}
}

@inproceedings{williamsSubquadratic,
author = {Williams, Ryan},
title = {Matrix-vector multiplication in sub-quadratic time: (some preprocessing required)},
year = {2007},
isbn = {9780898716245},
publisher = {Society for Industrial and Applied Mathematics},
address = {USA},
booktitle = {Proceedings of the Eighteenth Annual ACM-SIAM Symposium on Discrete Algorithms},
pages = {995–1001},
numpages = {7},
location = {New Orleans, Louisiana},
series = {SODA '07}
}

@article{subcubicEquiv,
author = {Williams, Virginia Vassilevska and Williams, R. Ryan},
title = {Subcubic Equivalences Between Path, Matrix, and Triangle Problems},
year = {2018},
issue_date = {October 2018},
publisher = {Association for Computing Machinery},
address = {New York, NY, USA},
volume = {65},
number = {5},
issn = {0004-5411},
url = {https://doi.org/10.1145/3186893},
doi = {10.1145/3186893},
journal = {J. ACM},
month = aug,
articleno = {27},
numpages = {38}
}

@inproceedings{leeSoda,
author = {Lee, Troy and Santha, Miklos and Zhang, Shengyu},
title = {Quantum algorithms for graph problems with cut queries},
year = {2021},
isbn = {9781611976465},
publisher = {Society for Industrial and Applied Mathematics},
address = {USA},
booktitle = {Proceedings of the Thirty-Second Annual ACM-SIAM Symposium on Discrete Algorithms},
pages = {939–958},
numpages = {20},
location = {Virtual Event, Virginia},
series = {SODA '21}
}

@INPROCEEDINGS {cutFocs,
author = { Apers, Simon and Efron, Yuval and Gawrychowski, Pawel and Lee, Troy and Mukhopadhyay, Sagnik and Nanongkai, Danupon },
booktitle = { 2022 IEEE 63rd Annual Symposium on Foundations of Computer Science (FOCS) },
title = {{ Cut Query Algorithms with Star Contraction }},
year = {2022},
volume = {},
ISSN = {},
pages = {507-518},
doi = {10.1109/FOCS54457.2022.00055},
url = {https://doi.ieeecomputersociety.org/10.1109/FOCS54457.2022.00055},
publisher = {IEEE Computer Society},
address = {Los Alamitos, CA, USA},
month =Nov}

@InProceedings{assadi_et_al:LIPIcs.ESA.2021.7,
  author =	{Assadi, Sepehr and Chakrabarty, Deeparnab and Khanna, Sanjeev},
  title =	{{Graph Connectivity and Single Element Recovery via Linear and OR Queries}},
  booktitle =	{29th Annual European Symposium on Algorithms (ESA 2021)},
  pages =	{7:1--7:19},
  series =	{Leibniz International Proceedings in Informatics (LIPIcs)},
  ISBN =	{978-3-95977-204-4},
  ISSN =	{1868-8969},
  year =	{2021},
  volume =	{204},
  editor =	{Mutzel, Petra and Pagh, Rasmus and Herman, Grzegorz},
  publisher =	{Schloss Dagstuhl -- Leibniz-Zentrum f{\"u}r Informatik},
  address =	{Dagstuhl, Germany},
  URL =		{https://drops.dagstuhl.de/entities/document/10.4230/LIPIcs.ESA.2021.7},
  URN =		{urn:nbn:de:0030-drops-145880},
  doi =		{10.4230/LIPIcs.ESA.2021.7}
}

@article{surveyDensest,
author = {Lanciano, Tommaso and Miyauchi, Atsushi and Fazzone, Adriano and Bonchi, Francesco},
title = {A Survey on the Densest Subgraph Problem and its Variants},
year = {2024},
issue_date = {August 2024},
publisher = {Association for Computing Machinery},
address = {New York, NY, USA},
volume = {56},
number = {8},
issn = {0360-0300},
url = {https://doi.org/10.1145/3653298},
doi = {10.1145/3653298},
journal = {ACM Comput. Surv.},
month = apr,
articleno = {208},
numpages = {40}
}

@inbook{Charikar2000,
  title = {Greedy Approximation Algorithms for Finding Dense Components in a Graph},
  ISBN = {9783540444367},
  ISSN = {0302-9743},
  url = {http://dx.doi.org/10.1007/3-540-44436-X_10},
  DOI = {10.1007/3-540-44436-x_10},
  booktitle = {Approximation Algorithms for Combinatorial Optimization},
  publisher = {Springer Berlin Heidelberg},
  author = {Charikar,  Moses},
  year = {2000},
  pages = {84–95}
}

@article {Fox-Pach-Sheffer-Suk-Zahl-2017,
    AUTHOR = {Fox, Jacob and Pach, J\'anos and Sheffer, Adam and Suk, Andrew
              and Zahl, Joshua},
     TITLE = {A semi-algebraic version of {Z}arankiewicz's problem},
   JOURNAL = {J. Eur. Math. Soc. (JEMS)},
  FJOURNAL = {Journal of the European Mathematical Society (JEMS)},
    VOLUME = {19},
      YEAR = {2017},
    NUMBER = {6},
     PAGES = {1785--1810},
      ISSN = {1435-9855,1435-9863},
   MRCLASS = {05C35 (14P10 52C10 68R10)},
  MRNUMBER = {3646875},
MRREVIEWER = {J.\ C.\ Lagarias},
       DOI = {10.4171/JEMS/705},
       URL = {https://doi.org/10.4171/JEMS/705},
}

@article {Do-2019,
    AUTHOR = {Do, Thao},
     TITLE = {Representation complexities of semialgebraic graphs},
   JOURNAL = {SIAM J. Discrete Math.},
  FJOURNAL = {SIAM Journal on Discrete Mathematics},
    VOLUME = {33},
      YEAR = {2019},
    NUMBER = {4},
     PAGES = {1864--1877},
      ISSN = {0895-4801,1095-7146},
   MRCLASS = {05C62 (05C65 05C70 14P10 52C10)},
  MRNUMBER = {4013919},
MRREVIEWER = {Jen\H o\ Lehel},
       DOI = {10.1137/18M1221606},
       URL = {https://doi.org/10.1137/18M1221606},
}

@article{Agarwal-Aronov-Ezra-Katz-Sharir-2025,
author = {Agarwal, Pankaj K. and Aronov, Boris and Ezra, Esther and Katz, Matthew J. and Sharir, Micha},
title = {Intersection Queries for Flat Semi-Algebraic Objects in Three Dimensions and Related Problems},
year = {2025},
issue_date = {July 2025},
publisher = {Association for Computing Machinery},
address = {New York, NY, USA},
volume = {21},
number = {3},
issn = {1549-6325},
url = {https://doi.org/10.1145/3721290},
doi = {10.1145/3721290},
journal = {ACM Trans. Algorithms},
month = jun,
articleno = {25},
numpages = {59}
}

@InProceedings{cardinal_et_al:LIPIcs.ESA.2025.67,
  author =	{Cardinal, Jean and Yuditsky, Yelena},
  title =	{{Compact Representation of Semilinear and Terrain-Like Graphs}},
  booktitle =	{33rd Annual European Symposium on Algorithms (ESA 2025)},
  pages =	{67:1--67:19},
  series =	{Leibniz International Proceedings in Informatics (LIPIcs)},
  ISBN =	{978-3-95977-395-9},
  ISSN =	{1868-8969},
  year =	{2025},
  volume =	{351},
  editor =	{Benoit, Anne and Kaplan, Haim and Wild, Sebastian and Herman, Grzegorz},
  publisher =	{Schloss Dagstuhl -- Leibniz-Zentrum f{\"u}r Informatik},
  address =	{Dagstuhl, Germany},
  URL =		{https://drops.dagstuhl.de/entities/document/10.4230/LIPIcs.ESA.2025.67},
  URN =		{urn:nbn:de:0030-drops-245359},
  doi =		{10.4230/LIPIcs.ESA.2025.67}
}

@article{SAUER1972145,
  title = {On the Density of Families of Sets},
  author = {Sauer, Norbert W.},
  year = 1972,
  journal = {Journal of Combinatorial Theory, Series A},
  volume = {13},
  number = {1},
  pages = {145--147},
  issn = {0097-3165},
  doi = {10.1016/0097-3165(72)90019-2}
}

@article{Shelah1972,
  title = {A combinatorial problem; stability and order for models and theories in infinitary languages},
  volume = {41},
  ISSN = {0030-8730},
  url = {http://dx.doi.org/10.2140/pjm.1972.41.247},
  DOI = {10.2140/pjm.1972.41.247},
  number = {1},
  journal = {Pacific Journal of Mathematics},
  publisher = {Mathematical Sciences Publishers},
  author = {Shelah,  Saharon},
  year = {1972},
  month = apr,
  pages = {247–261}
}

@book {Frieze-Karonski-2023,
    AUTHOR = {Frieze, Alan and Karo\'nski, {Micha\l}},
     TITLE = {Random graphs and networks: a first course},
 PUBLISHER = {Cambridge University Press, Cambridge},
      YEAR = {2023},
     PAGES = {xiii+217},
      ISBN = {978-1-009-26028-2; 978-1-009-26030-5},
   MRCLASS = {05-01 (05C40 05C62 05C80)},
  MRNUMBER = {4567944},
       DOI = {10.1017/9781009260268},
       URL = {https://doi.org/10.1017/9781009260268},
}

@article {Fox-Pach-Suk-2019,
    AUTHOR = {Fox, Jacob and Pach, J\'anos and Suk, Andrew},
     TITLE = {{Erd\H os-{H}ajnal conjecture for graphs with bounded
              {VC}-dimension}},
   JOURNAL = {Discrete Comput. Geom.},
  FJOURNAL = {Discrete \& Computational Geometry. An International Journal
              of Mathematics and Computer Science},
    VOLUME = {61},
      YEAR = {2019},
    NUMBER = {4},
     PAGES = {809--829},
      ISSN = {0179-5376,1432-0444},
   MRCLASS = {05D10 (05C55 05C65 52C10)},
  MRNUMBER = {3943496},
MRREVIEWER = {Luis\ Boza},
       DOI = {10.1007/s00454-018-0046-5},
       URL = {https://doi.org/10.1007/s00454-018-0046-5},
}

@article {Janzer-Pohoata-2024,
    AUTHOR = {Janzer, Oliver and Pohoata, Cosmin},
     TITLE = {On the {Z}arankiewicz problem for graphs with bounded
              {VC}-dimension},
   JOURNAL = {Combinatorica},
  FJOURNAL = {Combinatorica. An International Journal on Combinatorics and
              the Theory of Computing},
    VOLUME = {44},
      YEAR = {2024},
    NUMBER = {4},
     PAGES = {839--848},
      ISSN = {0209-9683,1439-6912},
   MRCLASS = {05C35 (05D05)},
  MRNUMBER = {4777869},
MRREVIEWER = {Jonathan\ Cutler},
       DOI = {10.1007/s00493-024-00095-2},
       URL = {https://doi.org/10.1007/s00493-024-00095-2},
}

@article{stinson,
  author       = {Douglas R. Stinson},
  title        = {Decomposition constructions for secret-sharing schemes},
  journal      = {{IEEE} Trans. Inf. Theory},
  volume       = {40},
  number       = {1},
  pages        = {118--125},
  year         = {1994},
  url          = {https://doi.org/10.1109/18.272461},
  doi          = {10.1109/18.272461},
  bibsource    = {dblp computer science bibliography, https://dblp.org}
}

\newpage

\appendix
\section{Appendix}\label{sec:appendix}

\begin{proof}[Proof of~\Cref{claim:properties}]


For the first item, note that $r_d$ can be written as
\[
r_d(n) = \frac{4^d}{d!}n^{d-1} \cdot a(n) \cdot b(d),
\]
where $b$ is an increasing function in $d$, and thus
\[
r_{d}(n) = \frac{4n}{d} \cdot \frac{4^{d-1}}{(d-1)!} \cdot a(n) \cdot b(d) \geq \frac{4n}{d} \cdot \frac{4^{d-1}}{(d-1)!} \cdot a(n) \cdot b(d-1) = \frac{4n}{d} r_{d-1}(n).
\]
The second item is direct from the fact that $\frac{n^{d-1} \lg\lg n}{\lg^2 n}$ is increasing in $n$, and so is $\frac{n^d}{\lg^3 n}$.
For the third item, observe first that since $r_d(\cdot)$ is increasing by the previous item,
\[
r_d(\ceil{n/(d-1)}) \leq r_d\left(\frac{n+d-2}{d-1}\right) 
\]
 Then, note that the function $a(n) := \lg \lg n / \lg^2 n$ satisfies 
 \(
 a\left(\frac{n+d-2}{d-1}\right) \leq 1.01 a(n),
 \) for $n \geq 2^{2^{200d^2}}$. Indeed,
 \[
 \frac{\lg \lg \left(\frac{n+d-2}{d-1}\right)}{\lg^2\left(\frac{n+d-2}{d-1}\right)} \leq \frac{\lg \lg n}{\lg^2 \left(\frac{n+d-2}{d-1}\right)} \leq \frac{\lg \lg n}{\lg^2(n/(d-1))} = \frac{\lg \lg n}{(\lg n - \lg (d-1))^2},
 \]
 but since $n \geq 2^{2^{200d^2}}$, we have $\lg n \geq 2^{200d^2} > 10000 \lg (d-1)$, and thus $\lg n - \lg (d-1) \geq 0.9999 \lg n$. We therefore have 
 \[
  a\left(\frac{n+d-2}{d-1}\right) \leq \frac{\lg \lg n}{(0.9999 \lg n)^2} \leq 1.01a(n).  
 \]
 Now note that as $n \geq 2^{2^{200d^2}}$,
 \[
 \left(\frac{n+d-2}{n}\right)^{d-1} = \left(1 + \frac{d-2}{n} \right)^{d-1} \leq \exp((d-2)(d-1)/n) \leq 1.01,
 \]
from where we conclude
 \begin{align*}
r_d(\ceil{n/(d-1)}) &\leq \frac{4^d}{d!} (n+d-2)^{d-1} \cdot \frac{1}{(d-1)^{d-1}} \cdot 1.01a(n) \cdot b(d)\\
&\leq 1.01^2 \cdot \frac{1}{(d-1)^{d-1}} \frac{4^d}{d!} n^{d-1} a(n)b(d) \leq 1.03 \frac{r_d(n)}{(d-1)^{d-1}}. \qedhere 
 \end{align*}

\end{proof}

\begin{proof}[Proof of~\Cref{claim:outer}]
The number of auxiliary $d$-cliques containing $v$ is at most
\begin{align*}
    (\star) := &\sum_{\substack{\vec{x} \in \mathcal{S}_d \\ x_i \neq 0}} \binom{|P_i|}{x_i-1} \cdot \left(\prod_{j \in [d-1] \setminus \{i,f(\vec{x})\}} \binom{|P_j|}{x_j}\right) h_{x_{f(\vec{x})}}(|P_{f(\vec{x})}|)\\
    &\leq \sum_{\substack{\vec{x} \in \mathcal{S}_d \\ x_i \neq 0}} \frac{|P_i|^{x_i - 1}}{(x_i-1)!} \cdot \left(\prod_{j \in [d-1] \setminus \{i,f(\vec{x})\}} \frac{|P_j|^{x_j}}{x_j!}\right) h_{x_{f(\vec{x})}}(|P_{f(\vec{x})}|)\\
    &\leq \sum_{\substack{\vec{x} \in \mathcal{S}_d \\ x_i \neq 0}} x_i 
    \left(\prod_{j \in [d-1] \setminus \{f(\vec{x})\}} \frac{1}{x_j!}\right) 
     \left(|P_i|^{x_i - 1}\prod_{j \in [d-1] \setminus \{i, f(\vec{x})\}} |P_j|^{x_j}\right) h_{x_{f(\vec{x})}}(|P_{f(\vec{x})}|)\\
     &= \sum_{\substack{\vec{x} \in \mathcal{S}_d \\ x_i \neq 0}} x_i 
    \left(\prod_{j \in [d-1] \setminus \{f(\vec{x})\}} \frac{1}{x_j!}\right) 
     \left(|P_i|^{x_i - 1}\prod_{j \in [d-1] \setminus \{i, f(\vec{x})\}} |P_j|^{x_j}\right) c(d) \cdot \frac{|P_{f(\vec{x})}|^{x_{f(\vec{x})}}}
     {\lg^3 |P_{f(\vec{x})}|},
\end{align*}
where the last equality introduced notation $c(d) := 2^{d2^{200d^2}}$.
We now observe that
\begin{align*}
(\star) &= \sum_{\substack{\vec{x} \in \mathcal{S}_d \\ x_i \neq 0}} x_i 
    \left(\prod_{j \in [d-1] \setminus \{f(\vec{x})\}} \frac{1}{x_j!}\right)  \left(|P_i|^{x_i - 1}\prod_{j \in [d-1] \setminus \{i\}} |P_j|^{x_j}\right)  \frac{c(d)}{\lg^3 |P_{f(\vec{x})}|}\\
    &\leq  \sum_{\substack{\vec{x} \in \mathcal{S}_d \\ x_i \neq 0}} x_i  x_{f(\vec{x})}!
    \left(\prod_{j \in [d-1]} \frac{1}{x_j!}\right)  \ceil{n/(d-1)}^{d-1}
    \frac{c(d)}{\lg^3\left(n/d\right)}\tag{Using $|P_{f(\vec{x})}| \geq \floor{n/(d-1)} \geq n/d$} \\
    &\leq d \cdot (d-1)^d \ceil{n/(d-1)}^{d-1}
    \frac{c(d)}{(\lg n - \lg d)^3} \tag{By~\Cref{eq:identity}}.
\end{align*}

We now need once again the fact
\begin{align*}
\ceil{n/(d-1)}^{d-1} &\leq \left(\frac{n+d-2}{d-1}\right)^{d-1} = \left(\frac{n}{d-1}\right)^{d-1} \left(1 + \frac{d-2}{n}\right)^{d-1} \\& \leq \left(\frac{n}{d-1}\right)^{d-1} e^{(d-2)(d-1)/n} \leq 1.01 \left(\frac{n}{d-1}\right)^{d-1}.
\end{align*}

From this, we have 
\[
(\star) \leq 1.01d \cdot n^{d-1} \frac{c(d)}{(\lg n - \lg d)^3} \leq 2d \cdot n^{d-1} \frac{c(d)}{\lg^3 n},\tag{Using $\lg d \leq 0.1 \lg n$}
\]
and as $r_d(n) = \frac{4^d}{d!} \cdot c(d) \cdot  \frac{n^{d-1} \lg \lg n}{\lg^2 n}$, we have 
\(
(\star)/r_d(n) \leq \frac{2d \cdot d!}{4^d \lg n \lg \lg n} \leq \frac{2d \cdot d!}{\lg n } < \frac{1}{3},
\)
where the last inequality holds since $n \geq 2^{2^{200d^2}}$ implies $\lg n \geq 2^{200d^2} $, and $d! \leq d^d = 2^{d \lg d}$.
\end{proof}

\begin{proof}[Proof of~\Cref{claim:count}]
We start by noting that
     \begin{align*}  
    |\mathcal{C}| &\leq \sum_{\vec{x} \in \mathcal{S}_d} \left(\prod_{i \in [d-1] \setminus \{f(\vec{x})\}} \binom{|P_i|}{x_i}\right) h_{x_{f(\vec{x})}}(|P_{f(\vec{x})}|)\\
    &\leq \sum_{\vec{x} \in \mathcal{S}_d} \left(\prod_{i \in [d-1] \setminus \{f(\vec{x})\}} \frac{|P_i|^{x_i}}{x_i!}\right) h_{x_{f(\vec{x})}}(|P_{f(\vec{x})}|).
      \end{align*}
Now, we split the sum according to whether some $x_i = d$ or not. Thus,
\begin{align*}
|\mathcal{C}| &\leq \left[\sum_{\substack{\vec{x} \in \mathcal{S}_d\\ d  \not\in \vec{x}}} \left(\prod_{i \in [d-1] \setminus \{f(\vec{x})\}} \frac{|P_i|^{x_i}}{x_i!}\right) h_{x_{f(\vec{x})}}(|P_{f(\vec{x})}|)\right] + (d-1)h_{d}(\ceil{n/(d-1)}) \\
&\leq \left[\sum_{\substack{\vec{x} \in \mathcal{S}_d\\ d  \not\in \vec{x}}} \left(\prod_{i \in [d-1] \setminus \{f(\vec{x})\}} \frac{|P_i|^{x_i}}{x_i!}\right) \frac{c(x_{f(\vec{x})})|P_{f(\vec{x})}|^{x_{f(\vec{x})}}}{\lg^3 |P_{f(\vec{x})}|}\right] + (d-1)h_{d}(\ceil{n/(d-1)})\\
&\leq \left[\sum_{\substack{\vec{x} \in \mathcal{S}_d\\ d  \not\in \vec{x}}} 
\ceil{n/(d-1)}^d
x_{f(\vec{x})}! \left(\prod_{i \in [d-1] } \frac{1}{x_i!}\right) \frac{c(d-1)}{\lg^3 |P_{f(\vec{x})}|}\right] + (d-1)h_{d}(\ceil{n/(d-1)})\\
&\leq \ceil{n/(d-1)}^d d! \cdot \frac{(d-1)^d}{d!} \frac{c(d-1)}{\lg^3 (n/d)} + (d-1)h_{d}(\ceil{n/(d-1)})\\
&\leq 1.01 \frac{n^d \cdot c(d-1)}{\lg^3 n} + \frac{1.01}{(d-1)^{d-1}}\frac{c(d) n^d}{\lg^3 n}\\
&= h_d(n) \cdot \left( \frac{1.01 c(d-1)}{c(d)} + \frac{1.01}{(d-1)^{d-1}} \right) < h_d(n).
\end{align*}
where the second-to-last inequality used the same analysis as in the proof of~\Cref{claim:outer}, and the last inequality used, e.g., $\frac{c(d-1)}{c(d)} < 1/10$ for $d \geq 3$.
\end{proof}


\end{document}